\documentclass[12pt]{amsart}
\pdfoutput=1
1
\usepackage{appendix}
\usepackage{amsmath}
\usepackage{mathrsfs, amssymb, upgreek}
\usepackage{euscript}
\usepackage{longtable,boldline}
\usepackage{amsthm}
\usepackage[all,pdf,cmtip]{xy}
\usepackage{mathtools}
\usepackage{hhline, multirow}
\usepackage{geometry}
\usepackage{blkarray}

\usepackage{ytableau}
\usepackage{tikz}
\usetikzlibrary{shapes.geometric}
\usetikzlibrary{intersections}
\newcommand{\xdasharrow}[2][->]{
\tikz[baseline=-\the\dimexpr\fontdimen22\textfont2\relax]{
\node[anchor=south,font=\scriptsize, inner ysep=1.5pt,outer xsep=2.2pt](x){#2};
\draw[shorten <=3.4pt,shorten >=3.4pt,dashed,#1](x.south west)--(x.south east);
}
}
\newcommand{\moplus}{\mathop{\textstyle{\bigoplus}}\limits}

\usepackage{hyperref} 

\newcommand{\FF}{\mathbb{F}}

\newcommand{\ZZ}{\mathbb{Z}}
\newcommand{\PP}{\mathbb{P}}
\newcommand{\QQ}{\mathbb{Q}}
\newcommand{\RR}{\mathbb{R}}

\newcommand{\OOO}{{\mathscr{O}}} 

\newcommand{\type}[1]{$\mathrm{#1}$}

\newcommand{\rk}{\operatorname{rk}}
\newcommand{\rr}{\operatorname{r}}
\newcommand{\dd}{\operatorname{d}}

\newcommand{\x}{\mathbf{x}}
\newcommand{\f}{\mathbf{f}}
\newcommand{\h}{\mathbf{h}}
\newcommand{\e}{\mathbf{e}}

\newcommand{\ba}{\mathbf{a}}

\newcommand{\bv}{\mathbf{v}}

\newcommand{\balpha}{{\boldsymbol{\upalpha}}}

\newcommand{\repsilon}{{\boldsymbol{\upepsilon}}}
\newcommand{\bLambda}{\boldsymbol{\Lambda}}

\newcommand{\Eff}{\operatorname{Eff}}
\newcommand{\Amp}{\operatorname{Amp}}
\newcommand{\Mov}{\operatorname{Mov}}
\newcommand{\NE}{\overline{\operatorname{NE}}}

\newcommand{\PGL}{\operatorname{PGL}}
\newcommand{\GL}{\operatorname{GL}}

\newcommand{\Sing}{\operatorname{Sing}}

\newcommand{\Aut}{\operatorname{Aut}}
\newcommand{\Pic}{\operatorname{Pic}}
\newcommand{\Cl}{\operatorname{Cl}}
\newcommand{\Gr}{\operatorname{Gr}}
\newcommand{\Bs}{\operatorname{Bs}}

\newcommand{\can}{{\operatorname{can}}}

\newcommand{\xref}[1]{\textup{\ref{#1}}}

\newcommand{\hD}{\widehat{D}}
\newcommand{\hE}{\widehat{E}}
\newcommand{\hM}{\widehat{M}}
\newcommand{\hP}{\widehat{P}}
\newcommand{\hS}{\widehat{S}}
\newcommand{\hX}{\widehat{X}}

\newcommand{\hZ}{\widehat{Z}}
\newcommand{\tS}{\tilde{S}}
\newcommand{\tX}{\tilde{X}}
\newcommand{\tD}{\tilde{D}}
\newcommand{\ti}{\tilde\imath}
\newcommand{\fD}{\mathfrak{D}}

\newcommand{\tcE}{\tilde{\mathscr{E}}}
\newcommand{\cE}{\mathscr{E}}

\newcommand{\cF}{\mathscr{F}}
\newcommand{\cI}{\mathscr{I}}

\newcommand{\cO}{\mathscr{O}}

\newcommand{\cS}{\mathscr{S}}
\newcommand{\cT}{\mathscr{T}}
\newcommand{\cU}{\mathscr{U}}

\newcommand{\bH}{\mathbf{H}}
\DeclareMathOperator{\Bl}{Bl}
\DeclareMathOperator{\Fl}{Fl}
\DeclareMathOperator{\codim}{codim}
\DeclareMathOperator{\Coker}{Coker}
\DeclareMathOperator{\Proj}{Proj}
\DeclareMathOperator{\rank}{rank}
\DeclareMathOperator{\Hom}{Hom}
\DeclareMathOperator{\Ext}{Ext}
\DeclareMathOperator{\Tor}{Tor}
\DeclareMathOperator{\Ker}{Ker}

\newcommand{\rc}{\mathrm{c}}
\newcommand{\rA}{\mathrm{A}}
\newcommand{\rD}{\mathrm{D}}
\newcommand{\rE}{\mathrm{E}}

\newcommand{\rC}{\mathrm{C}}
\newcommand{\rR}{\mathrm{R}}  

\makeatletter
\@addtoreset{equation}{section}
\makeatother

\renewcommand\labelenumi{\rm (\roman{enumi})}
\renewcommand\theenumi{\rm (\roman{enumi})}

\theoremstyle{plain}
\newtheorem{theorem}{Theorem}[section]
\newtheorem{lemma}[theorem]{Lemma}
\newtheorem{proposition}[theorem]{Proposition}
\newtheorem{corollary}[theorem]{Corollary}
\newtheorem*{claim*}{Claim}

\theoremstyle{definition}
\newtheorem{definition}[theorem]{Definition}

\newtheorem*{definition*}{Definition}

\newtheorem*{notation*}{Notation}

\newtheorem{remark}[theorem]{Remark}

\newcounter{NO}\numberwithin{NO}{subsection}

\title{On higher-dimensional del Pezzo varieties}
\author{Alexander Kuznetsov}
\address{
Steklov Mathematical Institute of Russian Academy of Sciences, Moscow, RUSSIA
\newline\indent
Laboratory of Algebraic Geometry, SU-HSE, 
7 Vavilova Str., Moscow, 117312, RUSSIA
}
\email{akuznet@mi-ras.ru} 
\author{Yuri Prokhorov}

\address{
Steklov Mathematical Institute of Russian Academy of Sciences, Moscow, RUSSIA
\newline\indent
Department of Algebra, Faculty of Mathematics, Moscow State
University, Moscow, 119 991, RUSSIA
\newline\indent
Laboratory of Algebraic Geometry, SU-HSE, 
7 Vavilova Str., Moscow, 117312, RUSSIA
}
\email{prokhoro@mi-ras.ru} 
\thanks{We were partially supported by the HSE University Basic Research Program.}
\begin{document}

\maketitle

\begin{abstract}
We study del Pezzo varieties, higher-dimensional analogues of del Pezzo surfaces.
In particular, we introduce ADE classification of del Pezzo varieties,
show that in type~$\rA$ the dimension of non-conical del Pezzo varieties is bounded by~$12 - d - r$,
where~$d$ is the degree and~$r$ is the rank of the class group, and classify maximal del Pezzo varieties.
\end{abstract}

{\small\tableofcontents}

\section{Introduction}
Del Pezzo surfaces form one of the most important and classical families of varieties in algebraic geometry.
In this paper we study their higher-dimensional analogues.

\begin{definition}
\label{def:dp}
A \emph{del Pezzo variety} is a variety of dimension~$n = \dim(X) \ge 2$ 
with at worst terminal singularities such that
\begin{equation}
\label{eq:kx-a}
- K_X = (n-1)A_X, 
\end{equation}
where~$A_X$ is an ample Cartier divisor class, called the \emph{fundamental divisor class} of~$X$.
\end{definition}

We will see later that the class group~$\Cl(X)$ of Weil divisors of a del Pezzo variety is torsion free 
(see Corollary~\ref{cor:clx-free}), so~$A_X$ is intrinsic to~$X$; 
however we will sometimes include it into the notation.
The main discrete invariants of a del Pezzo variety~$X$ is its 
\emph{degree}
\begin{align}
\label{eq:ddx}
\dd(X) &\coloneqq (A_X)^n\\
\intertext{and the rank of the class group}
\label{eq:rrx}
\rr(X) &\coloneqq \rank(\Cl(X)).
\end{align}
For~$n = 2$ we have the relation~$\dd(X) + \rr(X) = 10$, but it does not hold in higher dimensions.

We state a general classification theorem for del Pezzo varieties in Theorem~\ref{cla:dP} below.
This theorem has a long history:
first results of this sort for threefolds were obtained in the works of~G.~Fano in the first half of 20th century
and in the case of smooth threefolds the classification was completed by V.A.~Iskovskikh in~\cite[Theorem~4.2]{Isk:Fano1e}.
In higher dimensions the systematic study of del Pezzo varieties was initiated by T.~Fujita~\cite{Fujita-all};
in particular, Fujita proved Theorem~\ref{cla:dP} for smooth del Pezzo varieties of arbitrary dimension,
as well as parts~\ref{cla:dP1}-\ref{cla:dP4} of Theorem~\ref{cla:dP} in full generality.
However for other parts of the classification (especially for part~\ref{cla:dP5}) it is hard to find references,
so we provide a general statement here and give a proof in the body of the paper.
See also Theorem~\ref{thm:quintic-detailed} for more details about del Pezzo varieties of degree~$5$.

Recall that a \emph{line} on a polarized variety~$(X,A)$
is a smooth rational curve whose degree with respect to~$A$ equals~$1$. 
A polarized variety~$(X,A)$ is said to be \emph{conical}
if there is a point~$\bv \in X$ (called a \emph{vertex} of~$X$) 
such that~$X$ is swept by lines passing through~$\bv$.

\begin{theorem}
\label{cla:dP}
Let~$X$ be a del Pezzo variety of dimension~$n \ge 3$.
Then~$1 \le \dd(X) \le 8$ and 
\begin{enumerate}
\item\label{cla:dP1}
If $\dd(X)=1$, then~$X$ is isomorphic to a hypersurface in a weighted projective space
\begin{equation*}
X_6\subset \PP(1^n,2,3)
\end{equation*}
of degree~$6$ which does not pass through the singular points of~$\PP(1^n,2,3)$.

\item\label{cla:dP2}
If $\dd(X)=2$, then~$X$ is isomorphic to a hypersurface in a weighted projective space
\begin{equation*}
X_4\subset \PP(1^{n+1},2) 
\end{equation*}
of degree~$4$ which does not pass through the singular point of~$\PP(1^{n+1},2)$.

\item\label{cla:dP3}
If $\dd(X)=3$, then $X$ is isomorphic to a cubic hypersurface
\begin{equation*}
X_3 \subset \PP^{n+1}.
\end{equation*}

\item\label{cla:dP4}
If $\dd(X)=4$, then $X$ is isomorphic to a complete intersection of two quadrics
\begin{equation*}
X_{2,2} \subset \PP^{n+2}.
\end{equation*}

\item\label{cla:dP5}
If $\dd(X) = 5$ and~$X$ is non-conical, then $X$ is isomorphic to a complete intersection
\begin{equation*}
\Gr(2,5) \cap \PP^{n+3} \subset \PP^9,
\end{equation*}
in particular~$n \le 6$.

\item\label{cla:dP6}
If $\dd(X) = 6$ and~$X$ is non-conical, then~$X$ is isomorphic to a complete intersection
\begin{equation*}
(\PP^2 \times \PP^2) \cap \PP^{n+4} \subset \PP^8\qquad 
\text{or}\qquad \PP^1 \times \PP^1 \times \PP^1;
\end{equation*}
in particular~$\rr(X) \ge 2$ and~$n \le 4$.

\item\label{cla:dP7}
If~$\dd(X)=7$ and~$X$ is non-conical, then~$X$ is isomorphic to the blowup of $\PP^3$ at a point;
in particular~$\rr(X) = 2$ and~$n = 3$.

\item\label{cla:dP8}
If~$\dd(X) = 8$ and~$X$ is non-conical, then~$X \cong \PP^3$;
in particular~$\rr(X) = 1$ and~$n = 3$.
\end{enumerate}
\end{theorem}

Note that a del Pezzo variety~$X$ of degree~$d \le 4$ is conical 
if and only if in an appropriate coordinate system of the ambient (weighted) projective space
some variables of degree~$1$ do not appear in the equations. 
Similarly, conical varieties of degree~$d \ge 5$ are obtained as linear sections 
of iterated cones over the del Pezzo varieties~
\begin{equation*}
\Gr(2,5),
\quad
\PP^2 \times \PP^2,
\quad
\PP^1 \times \PP^1 \times \PP^1,
\quad
\Bl_P(\PP^3),
\quad\text{and}\quad
\PP^3
\end{equation*}
in the embeddings given by their fundamental classes.
In what follows we mostly consider non-conical varieties.

So far, the story looks quite boring and does not differ much from the story of del Pezzo surfaces.
It becomes much more interesting when one takes into account the birational aspect,
which allows one to connect in an interesting way del Pezzo varieties of different degrees.
It turns out very useful to introduce the following weakening of the definition.

\begin{definition}
\label{def:almostDP}
An \emph{almost del Pezzo variety} is a variety of dimension~$n = \dim(X) \ge 2$ 
with at worst terminal singularities for which~\eqref{eq:kx-a} holds,
where~$A_X$ is a nef and big Cartier divisor,
and the map $\Phi_{|mA_X|} \colon X \to \PP^N$ given by the linear system~$|mA_X|$ for~$m\gg 0$ 
is a \emph{small morphism}, i.e., does not contract divisors.
\end{definition}

\begin{remark}
P.~Jahnke and Th.~Peternell in~\cite{Jahnke-Pet} used the name ``almost del Pezzo'' for a different class of varieties;
on the one hand, they considered only 
smooth varieties (they called them ``almost del Pezzo manifolds'');
on the other hand they did not assume the morphism~$\Phi_{|mA_X|}$ to be small.
In this setup they provided a birational classification, 
which is quite close to our results stated in Theorem~\ref{thm:primitive-classification} below.
\end{remark} 

For surfaces Definition~\ref{def:dp} is equivalent to Definition~\ref{def:almostDP},
but in higher dimensions these definitions differ.
To explain the difference note that for an almost del Pezzo variety~$X$ 
one can always consider its \emph{anticanonical model}
\begin{equation*}
X_\can \coloneqq \Phi_{|mA_X|}(X)
\qquad\text{for~$m \gg 0$},
\end{equation*}
(i.e., the image of the morphism~$\Phi_{|mA_X|}$); see Lemma~\ref{lem:adp-dp} for another description of~$X_\can$.
Then the \emph{anticanonical morphism}
\begin{equation*}
\xi \colon X \longrightarrow X_\can
\end{equation*}
is small.
Conversely, 
$\QQ$-factorializations (see Definition~\ref{def:q-factorialization}) of any del Pezzo variety are almost del Pezzo,
and any pair of almost del Pezzo varieties~$X'$ and~$X''$ such that~\mbox{$X'_\can \cong X''_\can$}
is related by a \emph{pseudoisomorphism}, i.e., by a birational map~$\chi \colon X' \dashrightarrow X''$
which induces an isomorphism of complements of some closed subsets of codimension~$2$.
Moreover, if both~$X'$ and~$X''$ are $\QQ$-factorial, 
the map~$\chi$ is a composition of flops (Lemma~\ref{lemma:repant-models}).

Thus, classification of del Pezzo varieties up to isomorphism is equivalent 
to classification of almost del Pezzo varieties up to pseudoisomorphism
or $\QQ$-factorial almost del Pezzo varieties up to flops.
Note that the invariants~$\dd(X)$ and~$\rr(X)$ can be defined for almost del Pezzo varieties 
by the same formulas~\eqref{eq:ddx} and~\eqref{eq:rrx}, and that
\begin{equation*}
\dd(X_\can) = \dd(X),
\qquad 
\rr(X_\can) = \rr(X)
\end{equation*}
for any almost del Pezzo variety~$X$. 

Our approach to birational classification of del Pezzo varieties is quite close to the approach of Jahnke and Peternell.
Namely, to classify del Pezzo varieties birationally we introduce the following

\begin{definition}
\label{def:imprimitive}
An almost del Pezzo variety~$X$ is \emph{imprimitive} 
if it is pseudoisomorphic to a $\QQ$-factorial almost del Pezzo variety~$X'$ which admits
a $K$-negative birational extremal contraction.
Otherwise, we say~$X$ is \emph{primitive}.
\end{definition}

In fact, using a result from~\cite{Andreatta-Tasin} we show in Lemma~\ref{lemma:ext-rays0} 
that any $K$-negative birational contraction of a $\QQ$-factorial almost del Pezzo variety
is the blowup of several distinct smooth points on another almost del Pezzo variety.
Using this, we deduce the following structural result.

\begin{theorem}
\label{thm:imprimitive}
For any del Pezzo variety~$X$ 
there is a primitive del Pezzo variety~$X_0$
and a collection~$P_1,\dots,P_k \in X_0$ of distinct smooth points such that 
\begin{equation*}
X \cong (\Bl_{P_1,\dots,P_k}(X_0))_\can.
\end{equation*}
Moreover, $\dd(X) = \dd(X_0) - k$ and~$\rr(X) = \rr(X_0) + k$.
\end{theorem}

Similarly, we show in Lemma~\ref{lemma:ext-rays0a} that any $K$-negative extremal contraction 
from a $\QQ$-factorial almost del Pezzo variety
to a lower-dimensional variety is a $\PP^{n-2}$-fibration over a del Pezzo surface
or a flat quadric bundle over~$\PP^1$.
Using this we obtain 

\begin{theorem}
\label{thm:primitive-classification}
If~$X$ is a primitive non-conical del Pezzo variety then one of the following holds:
\begin{enumerate}
\item 
\label{it:r1}
$\rr(X) = 1$ and~$\dd(X) \in \{1,2,3,4,5,8\}$;
\item 
\label{it:r2q}
$\rr(X) = 2$ and~$X = \hX_\can$, where~$\hX \to \PP^1$ is a quadric bundle, and~$\dd(X) \in \{1,2,4\}$;
\item 
\label{it:r2p}
$\rr(X) = 2$ and~$X = \hX_\can$, where~$\hX = \PP_{\PP^2}(\cE)$, 
and~$\dd(X) \in \{1,2,3,5,6\}$;
\item 
\label{it:r3}
$\rr(X) = 3$ and~$X = \hX_\can$, where~$\hX = \PP_{\PP^1 \times \PP^1}(\cE)$, 
and~$\dd(X) \in \{2,4,6\}$.
\end{enumerate}
In the last two cases~$\cE$ is a vector bundle of rank~\mbox{$n - 1$} with~$\rc_1(\cE) = K$.
\end{theorem}

This motivates us to give the following

\begin{definition}
\label{def:dpb}
A vector bundle~$\cE$ on a surface~$Z$ is a \emph{del Pezzo bundle} 
if~\mbox{$\rk(\cE) \ge 2$}, \mbox{$\rc_1(\cE) = K_Z$} and~$\PP_Z(\cE)$ is an almost del Pezzo variety
such that~$\PP_Z(\cE)_\can$ is non-conical.
\end{definition} 

Looking at Theorem~\ref{thm:imprimitive} and Theorem~\ref{thm:primitive-classification} together,
it is natural to ask how one can distinguish which of types~\ref{it:r1}--\ref{it:r3}
a primitive contraction~$X_0$ of a given almost del Pezzo variety~$X$ has.

To answer this question we combine the birational point of view as in~\cite{Jahnke-Pet}
with the lattice structure of the class group~$\Cl(X)$
introduced in the three-dimensional case by the second named author.
So, following~\cite{P:GFano1}, for any almost del Pezzo variety~$X$ 
we define a symmetric bilinear form on~$\Cl(X)$ by
\begin{equation}
\label{eq:product}
\langle D_1,\, D_2 \rangle \coloneqq A_X^{n-2} \cdot D_1\cdot D_2. 
\end{equation} 
Note that~$\langle A_X, A_X \rangle = \dd(X)$.
When~$n = 2$, this boils down to the standard intersection product~$\langle D_1, D_2 \rangle = D_1 \cdot D_2$
in the Picard group of a del Pezzo surface.

In the statement of the next theorem a ``general linear surface section'' 
is an intersection of~$n - 2$ general fundamental divisors in~$X$ 
(i.e., divisors from the linear system~$|A_X|$);
we prove in~\S\ref{ss:fundamental} it is a smooth del Pezzo surface~$S$ of degree~$\dd(S) = \dd(X)$.

\begin{theorem}
\label{thm:intro-clx}
Let~$X$ be an almost del Pezzo variety of dimension~$n \ge 3$.
\begin{enumerate}
\item 
If~$i \colon S \hookrightarrow X$ is a general linear surface section of~$X$ then the restriction map 
\begin{equation*}
\Cl(X) \xrightarrow{\ i^*\ } \Cl(S)
\end{equation*}
induces an isomorphism of lattices~$\Cl(X) \cong \Xi(X)^\perp \subset \Cl(S)$,
where~$\Xi(X)^\perp$ is the orthogonal complement
of a negative definite sublattice~\mbox{$\Xi(X) \subset K_S^\perp \subset \Cl(S)$}
of rank~$m = 10 - \dd(X) - \rr(X)$ which has one of the following Dynkin types:
\begin{equation*}
\rA_m,
\quad
1 \le m \le 7,
\quad\text{or}\quad
\rD_m,
\quad 
4 \le m \le 7, 
\quad\text{or}\quad 
\rE_m, 
\quad
6 \le m \le 8.
\end{equation*}
\item 
If~$\dim(X) \ge 4$ and~$Y \subset X$ is a general fundamental divisor containing~$S$
then~$Y$ is an almost del Pezzo variety and
\begin{equation*}
\Xi(Y) = \Xi(X)
\qquad\text{and}\qquad 
\Cl(Y) = \Cl(X),
\end{equation*}
as sublattices in~$\Cl(S)$.
\item 
If~$X' = \Bl_P(X)$ where~$P \in X$ is a smooth point and~$X'$ is almost del Pezzo then
\begin{equation*}
\Xi(X') = \Xi(X)
\qquad\text{and}\qquad 
\Cl(X') = \Cl(X) \oplus \ZZ,
\end{equation*}
as sublattices in~$\Cl(S')$, where~$S' \cong \Bl_P(S)$ is a general linear surface section of~$X'$.
\end{enumerate} 

\end{theorem}

\begin{remark}
Of course, a given del Pezzo variety has many non-isomorphic linear surface sections;
but as we will see in Proposition~\ref{prop:class-divisor} their class groups can be identified 
(and such an identification can be made canonical up to monodromy action)
in such a way that the sublattices~$i^*\Cl(X) \subset \Cl(S)$ and~$\Xi(X) \subset \Cl(S)$ are identified.
\end{remark} 

The following result relates Dynkin type of the lattice~$\Xi(X)$ from Theorem~\ref{thm:intro-clx}
to the type of a primitive contraction~$X_0$ of~$X$ from Theorem~\ref{thm:primitive-classification}. 

\begin{theorem}
\label{thm:intro-bircla}
Let~$X$ be a non-conical del Pezzo variety of dimension~$n \ge 3$. 
\begin{enumerate}
\item 
\label{it:bircla-am}
$\Xi(X)$ has type~$\rA_m$, $1 \le m \le 7$, if and only if 
\begin{itemize}
\item 
$X \cong \PP^3$, hence~$\rr(X) = 1$ and~$m = 1$, or
\item 
$X \cong \Gr(2,5) \cap \PP^{n+3}$, a smooth linear section, hence~$\rr(X) = 1$ and~$m = 4$, or
\item 
$X \cong \Bl_{P_1,\dots,P_k}(\PP_Z(\cE))_\can$
where~$Z = \PP^2$ or~$Z = \PP^1 \times \PP^1$ and~$\cE$ is a del Pezzo bundle on~$Z$;
in this case~$k = \rr(X) - \rr(Z) - 1$ and~$m = \rc_2(\cE) - 1$.
\end{itemize}
\item 
\label{it:bircla-dm}
$\Xi(X)$ has type~$\rD_m$, $4 \le m \le 7$, if and only if~$X \cong \Bl_{P_1,\dots,P_k}(X_0)_\can$,
where~$X_0$ is 
\begin{itemize}
\item
a complete intersection in~$\PP^1 \times \PP^{n+2}$ of three divisors of bidegree~$(1,1)$, $(1,1)$, and~$(0,2)$;
in this case~$\dd(X_0) = 4$, $\rr(X_0) = 2$, $k = \rr(X) - 2$, and~$m = 4$;
\item 
a complete intersection in~$\PP^{n+2}$ of two quadrics;
in this case~$\dd(X_0) = 4$, $\rr(X_0) = 1$, $k = \rr(X) - 1$, and~$m = 5$;
\item 
a divisor in~$\PP^1 \times \PP^{n}$ of bidegree~$(2,2)$
in this case~$\dd(X_0) = 2$, $\rr(X_0) = 2$, $k = \rr(X) - 2$, and~$m = 6$;
\item 
a complete intersection in~$\PP^1 \times \PP^{2n}$ of $n$ divisors of bidegree~$(1,1)$ 
and one divisor of bidegree~$(1,2)$;
in this case~$\dd(X_0) = 1$, $\rr(X_0) = 2$, $k = \rr(X) - 2$, and~$m = 7$.
\end{itemize}
In all these cases~$\dd(X) \le 4$ and~$\dd(X) + \rr(X) \le 6$, while~$\dim(X) = n$ is arbitrary.
\item 
\label{it:bircla-em}
$\Xi(X)$ has type~$\rE_6$, $\rE_7$, or~$\rE_8$, if and only if~$X \cong \Bl_{P_1,\dots,P_{r - 1}}(X_0)_\can$,
where~$X_0$ is 
\begin{itemize}
\item 
a divisor of degree~$3$ in~$\PP^{n+1}$, $\dd(X_0) = 3$, $\rr(X_0) = 1$, $m = 6$;
\item 
a divisor or degree~$4$ in~$\PP(1^{n+1},2)$, $\dd(X_0) = 2$, $\rr(X_0) = 1$, $m = 7$;
\item 
a divisor or degree~$6$ in~$\PP(1^n,2,3)$, $\dd(X_0) = 1$, $\rr(X_0) = 1$, $m = 8$.
\end{itemize}
In these cases~$\dd(X) \le 3$ and~$\dd(X) + \rr(X) \le 4$, while~$\dim(X) = n$ is arbitrary.
\end{enumerate}
\end{theorem}

In the case of del Pezzo varieties of type~$\rA_m$ we can give the following uniform description.

\begin{theorem}
\label{thm:intro-dpa}
If~$X$ is a non-conical del Pezzo variety of type~$\rA_m$ with~$\rr(X) \ge 2$ then
\begin{equation*}
X \cong \PP_Z(\cE)_\can,
\end{equation*}
where~$Z$ is a del Pezzo surface, $\cE$ is a del Pezzo bundle with~$\rc_2(\cE) = m + 1$, 
\begin{equation*}
\dd(X) = K_Z^2 - \rc_2(\cE),
\qquad\text{and}\qquad
\rr(X) = 11 - K_Z^2 = \rr(Z) + 1.
\end{equation*}
In particular, $2 \le \rc_2(\cE) \le K_Z^2 - 1$.
\end{theorem}

This story would not be complete without a classification of del Pezzo vector bundles.
We say that a non-conical del Pezzo variety is \emph{maximal}, 
if it cannot be represented as a fundamental divisor in another non-conical del Pezzo variety.
Similarly, a del Pezzo bundle on a del Pezzo surface~$Z$ is \emph{maximal} 
if there is no embedding~$\cE \hookrightarrow \cE'$
into another del Pezzo bundle~$\cE'$ with~$\cE'/\cE \cong \cO_Z$. 

\begin{theorem}
\label{thm:intro-dpb}
Let~$Z$ be a del Pezzo surface and let~$\cE$ be a del Pezzo bundle on~$Z$.
\begin{enumerate}
\item 
\label{it:cemax}
There is a canonical exact sequence
\begin{equation}
\label{eq:ce-cemax}
0 \longrightarrow \cE \longrightarrow \tcE \longrightarrow H^1(Z,\cE) \otimes \cO_Z \longrightarrow 0,
\end{equation}
where~$\tcE$ is a maximal del Pezzo bundle 
and~$\PP_Z(\cE)_\can$ is a linear section of~$\PP_Z(\tcE)_\can$.
Moreover, $\rk(\cE) \le \rk(\tcE) = \rc_2(\tcE) = \rc_2(\cE)$.

\item
\label{it:uniqueness}
A maximal del Pezzo bundle~$\cE$ with~$\rc_2(\cE) \le K^2_Z - 2$ is unique up to isomorphism
except for the case~$K^2_Z = 4$, $\rc_2(\cE) = 2$, where there is a $1$-dimensional family of such bundles.

\item
\label{it:almost-uniqueness}
There is a bijection between isomorphism classes of maximal del Pezzo bundles~$\cE$ with~$\rc_2(\cE) = K^2_Z - 1$ 
and points~$z \in Z$ such that~$\Bl_z(Z)$ is del Pezzo;
in particular, if~$K^2_Z \ge 6$ then the isomorphism class of~$\PP_Z(\cE)$ is unique.
\end{enumerate}
\end{theorem}

We describe explicitly maximal del Pezzo bundles on~$\PP^2$ and~$\PP^1 \times \PP^1$
in Proposition~\ref{prop:bundle-p2} and Proposition~\ref{prop:bundle-p1p1}, respectively,
and sketch the classification on the other del Pezzo surfaces in Remark~\ref{rem:dpb-other}.

Combining the inequality~$\rk(\cE) \le \rc_2(\cE)$ from Theorem~\ref{thm:intro-dpb}\ref{it:cemax}
with inequalities of Theorem~\ref{thm:intro-bircla} for varieties of types~$\rD_m$ and~$\rE_m$,
we obtain the following boundedness results:

\begin{corollary}
\label{cor:drn}
Let~$X$ be a del Pezzo variety of dimension~$n \ge 3$. 
Then 
\begin{equation*}
\dd(X) + \rr(X)\le 9.
\end{equation*}
If, moreover, $X$ is non-conical of type~$\rA_m$ then
\begin{equation*}
\dd(X) + \rr(X) + n \le 12.
\end{equation*}
\end{corollary}

The next theorem provides a detailed classification of del Pezzo varieties of type~$\rA$.

\begin{theorem}
\label{thm:intro-a-detailed}
For each pair~$(d,r)$ of integers such that~$d \ge 1$, $r \ge 2$, and~$d + r \le 9$ and each~$3 \le n \le 12 - d - r$ 
there is a del Pezzo variety~$X = X_{d,r,n}$ with~$\dd(X) = d$, $\rr(X) = r$, $\dim(X) = n$ of type~$\rA_{10-d-r}$;
it is contained as a linear section in the unique maximal del Pezzo variety~$X_{d,r,12-d-r}$.
Moreover,
\begin{enumerate}
\item 
\label{it:a-moduli}
For~$(d,r) \in \{(1,6), (1,7), (1,8), (2,7) \}$ there is a family of dimension~$2r - d - 9$
parameterizing maximal del Pezzo varieties~$X_{d,r,12-d-r}$.
\item 
\label{it:a-pair}
For~$(d,r) \in \{(2,3), (4,3), (6,3) \}$ there are exactly two distinct 
maximal del Pezzo varieties~$X_{d,r,12-d-r}$ and~$X^*_{d,r,12-d-r}$.

\item 
\label{it:a-single}
For all other~$(d,r)$ there is a unique maximal del Pezzo variety~$X_{d,r,n}$.
\end{enumerate}
\end{theorem}

In Appendix~\ref{sec:examples} we list explicit equations for the most of maximal del Pezzo varieties of type~$\rA$,
and here we just show del Pezzo varieties of all types schematically.

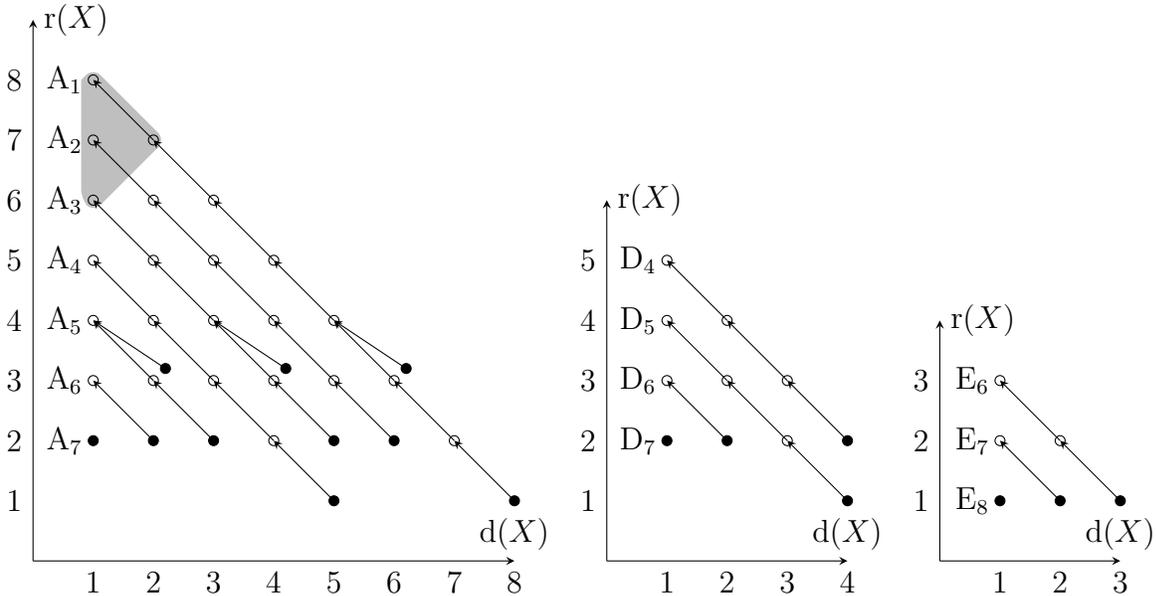
\begin{equation*}
\newcommand{\scale}{.8}
\label{eq:figure-bounded}
\vcenter{\hbox{
%%%%%%%%%%%%% A type %%%%%%%%%%%%%
\begin{tikzpicture}[>=stealth, xscale = \scale, yscale = \scale]
\draw[->] (0,0) -- (8,0) node[above]{$\dd(X)$};
\draw[->] (0,0) -- (0,9) node[right]{$\rr(X)$};
\fill[color = lightgray, rounded corners] (0.8,8) -- (0.8,6) -- (1,5.8) -- (2,6.8) -- (2.2,7) -- (2,7.2) -- (1,8.2) -- (0.8,8);
\foreach \i in {1,2,3,4,5,6,7,8} \node[below] at (\i,0) {\i};
\foreach \i in {1,2,3,4,5,6,7,8} \node[left] at (0,\i) {\i};
\foreach \i in {5,8} \filldraw[black] (\i,1) circle (.2em);
\foreach \i in {1,2,3,5,6} \filldraw[black] (\i,2) circle (.2em);
\foreach \i in {4,7} \draw (\i,2) circle (.2em); 
\foreach \i in {2,4,6} \filldraw[black] (\i+.2,3.2) circle (.2em);
\foreach \i in {1,2,3,4,5,6} \draw (\i,3) circle (.2em); 
\foreach \i in {1,2,3,4,5} \draw (\i,4) circle (.2em); 
\foreach \i in {1,2,3,4} \draw (\i,5) circle (.2em); 
\foreach \i in {1,2,3} \draw (\i,6) circle (.2em); 
\foreach \i in {1,2} \draw (\i,7) circle (.2em); 
\foreach \i in {1} \draw (\i,8) circle (.2em); 
\draw[->] (8,1) -- (7,2);
\draw[->] (6.2,3.2) -- (5,4);
\foreach \i in {2,3,4,5,6,7} \draw[->] (\i,9-\i) -- (\i-1,10-\i);
\draw[->] (6,2) -- (5,3);
\foreach \i in {2,3,4,5} \draw[->] (\i,8-\i) -- (\i-1,9-\i);
\draw[->] (5,2) -- (4,3);
\draw[->] (4.2,3.2) -- (3,4);
\foreach \i in {2,3,4} \draw[->] (\i,7-\i) -- (\i-1,8-\i);
\draw[->] (5,1) -- (4,2);
\foreach \i in {2,3,4} \draw[->] (\i,6-\i) -- (\i-1,7-\i);
\draw[->] (3,2) -- (2,3);
\draw[->] (2,3) -- (1,4);
\draw[->] (2.2,3.2) -- (1,4);
\draw[->] (2,2) -- (1,3);
\node [left] at (1,8) {$\rA_1$};
\node [left] at (1,7) {$\rA_2$};
\node [left] at (1,6) {$\rA_3$};
\node [left] at (1,5) {$\rA_4$};
\node [left] at (1,4) {$\rA_5$};
\node [left] at (1,3) {$\rA_6$};
\node [left] at (1,2) {$\rA_7$};
\end{tikzpicture}
%%%%%%%%%%%%% D type %%%%%%%%%%%%%
\begin{tikzpicture}[>=stealth, xscale = \scale, yscale = \scale]
\draw[->] (0,0) -- (4,0) node[above]{$\dd(X)$};
\draw[->] (0,0) -- (0,6) node[right]{$\rr(X)$};
\foreach \i in {1,2,3,4} \node[below] at (\i,0) {\i};
\foreach \i in {1,2,3,4,5} \node[left] at (0,\i) {\i};
\foreach \i in {4} \filldraw[black] (\i,1) circle (.2em); 
\foreach \i in {1,2,4} \filldraw[black] (\i,2) circle (.2em); 
\foreach \i in {3} \draw (\i,2) circle (.2em); 
\foreach \i in {1,2,3} \draw (\i,3) circle (.2em); 
\foreach \i in {1,2} \draw (\i,4) circle (.2em); 
\foreach \i in {1} \draw (\i,5) circle (.2em); 
\draw[->] (4,1) -- (3,2);
\draw[->] (4,2) -- (3,3);
\draw[->] (3,2) -- (2,3);
\draw[->] (2,2) -- (1,3);
\draw[->] (3,3) -- (2,4);
\draw[->] (2,3) -- (1,4);
\draw[->] (2,4) -- (1,5);
\node [left] at (1,5) {$\rD_4$};
\node [left] at (1,4) {$\rD_5$};
\node [left] at (1,3) {$\rD_6$};
\node [left] at (1,2) {$\rD_7$};
\end{tikzpicture}
%%%%%%%%%%%%% E type %%%%%%%%%%%%%
\begin{tikzpicture}[>=stealth, xscale = \scale, yscale = \scale]
\draw[->] (0,0) -- (3,0) node[above]{$\dd(X)$};
\draw[->] (0,0) -- (0,4) node[right]{$\rr(X)$};
\foreach \i in {1,2,3} \node[below] at (\i,0) {\i};
\foreach \i in {1,2,3} \node[left] at (0,\i) {\i};
\foreach \i in {1,2,3} \filldraw[black] (\i,1) circle (.2em); 
\foreach \i in {1,2} \draw (\i,2) circle (.2em); 
\foreach \i in {1} \draw (\i,3) circle (.2em); 
\draw[->] (3,1) -- (2,2);
\draw[->] (2,1) -- (1,2);
\draw[->] (2,2) -- (1,3);
\node [left] at (1,3) {$\rE_6$};
\node [left] at (1,2) {$\rE_7$};
\node [left] at (1,1) {$\rE_8$};
\end{tikzpicture}
}
}
\end{equation*}
The three ``maps'' of this ``atlas'' correspond to varieties of types~$\rA$, $\rD$, and~$\rE$, respectively.
Black dots stand for primitive and white dots for imprimitive varieties,
and the arrows correspond to the operation of blowup of a general point 
(which decreases~$\dd(X)$ by~$1$ and increases~$\rr(X)$ by~$1$, keeping their sum constant),
followed by passing to the anticanonical model.
The grey area on the first map shows those maximal varieties of type~$\rA$ that have moduli.   

\begin{remark}
\label{rem:hpd}
If~$\cE$ is a del Pezzo bundle with~$2 \le \rc_2(\cE) \le K_Z^2 - 2$ 
the dual bundle~$\cE^\vee$ is globally generated (see Lemma~\ref{lem:dpv-dpb})
and we prove in Corollary~\ref{cor:dpb-duality} that the operation
\begin{equation*}
\cE \longmapsto \cE^\perp \coloneqq \Ker\big(H^0(Z,\cE^\vee) \otimes \cO_Z \longrightarrow \cE^\vee\big)
\end{equation*}
is an involution sending a del Pezzo bundle~$\cE$ with~$\rc_2(\cE) = k$ 
to the del Pezzo bundle~$\cE^\perp$ with~$\rc_2(\cE^\perp) = K_Z^2 - k$.
Combining this with \emph{linear homological projective duality} from~\cite[\S8]{K:hpd}
we obtain an interesting relation between derived categories 
of the corresponding maximal del Pezzo varieties of type~$\rA_{k-1}$ and~$\rA_{K^2_Z - 1 - k}$,
see~\cite{BB22} for a discussion of the case where~$Z$ is the quintic del Pezzo surface and~$k = 2$.
\end{remark}

The involution of the set of all del Pezzo bundles with~$2 \le \rc_2(\cE) \le K_Z^2 - 2$ 
described in Remark~\ref{rem:hpd}
gives rise to a duality that takes a del Pezzo variety~$\PP_Z(\cE)_\can$ of degree~\mbox{$K_Z^2 - \rc_2(\cE) \ge 2$}
to the del Pezzo variety~$\PP_Z(\cE^\perp)$ of degree~$\rc_2(\cE) \ge 2$.
Our last theorem explains what happens in the missing case of del Pezzo varieties of degree~$1$.

Recall from Theorem~\ref{cla:dP} that any del Pezzo variety~$X$ with~$\dd(X) = 1$ is a hypersurface in~$\PP(1^n,2,3)$.
We denote by~$x_0 \in X$ the base point of the projection~$X \dashrightarrow \PP^{n-1}$ and by
\begin{equation*}
\fD(X) \subset \PP^{n-1}
\end{equation*}
the discriminant divisor of the elliptic fibration~$\Bl_{x_0}(X) \to \PP^{n-1}$.
On the other hand, given a del Pezzo surface~$S$ of degree~$d$ 
we denote by~$\bH(S) \subset S \times \PP^{d}$ the universal anticanonical divisor,
and given a line~$L \subset S$ we set~$\Pi_L \coloneqq (L \times \PP^{d}) \cap \bH(S)$; 
this is a divisor in~$\bH(S)$.

\begin{theorem}
\label{thm:dp1-4-cubic}
For~$3 \le n \le 9$ there is a bijection between the sets of isomorphism classes of
\begin{itemize}
\item 
maximal non-conical del Pezzo varieties~$X$ of type~$\rA_{n-2}$ with~$\dd(X) = 1$, and
\item 
smooth del Pezzo surfaces~$S$ of degree~$n-1$,
\end{itemize}
such that~$S$ is projective dual to~$\fD(X)$ and 
there is a pseudoisomorphism~$\Bl_{x_0}(X) \dashrightarrow \bH(S)$ over~$\PP^{n-1}$
that takes the exceptional divisor of~$\Bl_{x_0}(X)$ 
to the divisor~$\Pi_L \subset \bH(S)$.
\end{theorem}

This theorem generalizes and extends the main result of~\cite{ACPS:DoubleVC}, 
which is its special case where~$n = 3$.
For more details about this bijection see~\S\ref{ss:degree-1}.

\medskip 

The paper is organized as follows.
In~\S\ref{sec:geometry} we discuss the basic properties of del Pezzo varieties 
and study their birational geometry;
in particular we classify all $K$-negative extremal contractions of $\QQ$-factorial almost del Pezzo varieties.
In~\S\ref{sec:lattice} we study the lattice structure of~$\Cl(X)$ defined by the bilinear form~\eqref{eq:product},
and describe its behavior with respect to extremal contractions.
In~\S\ref{sec:dpb} we prove the basic properties of del Pezzo bundles,
including their classification.
We also describe relatively minimal almost del Pezzo quadric bundles over~$\PP^1$.
In~\S\ref{sec:proofs} we collect the proofs of all results listed in the Introduction.
In~\S\ref{sec:cones} we describe the effective and moving cones of a del Pezzo variety in terms of its class group.
We also describe all possible $\QQ$-factorializations of del Pezzo varieties with~$\rr(X) \le 3$.
In~\S\ref{sec:schubert} we discuss in more detail del Pezzo varieties of degree~$5$.
In Appendix~\ref{sec:examples} we list equations of maximal del Pezzo varieties of type~$\rA$.
Finally, in Appendix~\ref{sec:details} we describe roots and special classes of del Pezzo varieties. 

\medskip 

{\parindent=0pt\bf Conventions.}
We work over an algebraically closed field of characteristic~$0$.

\medskip 

{\parindent=0pt\bf Acknowledgements.}
We would like to thank Costya Shramov and Andrey Trepalin for useful discussions
and the anonymous referee for his very helpful comments.

\section{Geometry of del Pezzo varieties}
\label{sec:geometry}

In this section we discuss the geometry of del Pezzo varieties. 

\subsection{Fundamental linear system}
\label{ss:fundamental}

We start with basic properties of the linear system of the fundamental divisor class~$A = A_X$ defined by~\eqref{eq:kx-a}.
Recall the bilinear form on the class group~$\Cl(X)$ of Weil divisors defined in~\eqref{eq:product}.

\begin{lemma}
\label{lem:vanishings}
If~$X$ is an almost del Pezzo variety, $\dim(X) = n$, and~$D \in \Cl(X)$ is a nef class then
\begin{align*}
H^i(X,\cO_X(kA_X + D)) &= 0
&&
\text{for~$i > 0$ and~$k \ge 2 - n$},\\
H^i(X,\cO_X(kA_X)) &= 0
&&
\text{for~$0 < i < n$ and all~$k \in \ZZ$}.
\end{align*}
Moreover, if the divisor class~$kA_X + D$ is effective and non-trivial, then 
\begin{equation*}
k\dd(X)+\langle D,\, A_X\rangle > 0
\qquad\text{and}\qquad 
k\langle D,\, A_X\rangle +\langle D,\, D\rangle\ge 0. 
\end{equation*}
\end{lemma}

\begin{proof}
Using~\eqref{eq:kx-a} we write~$kA_X + D = K_X + (k + n - 1)A_X + D$.
The first equality follows from Kawamata--Viehweg vanishing 
because~$k + n - 1 \ge 1$ by the assumption and therefore~$(k + n - 1)A_X + D$ is nef and big.
The second equality follows from the first (for~\mbox{$D = 0$}) 
and Serre duality~$H^i(X,\cO_X(kA_X)) \cong H^{n-i}(X,\cO_X(1-n-k)A_X)^\vee$.

Assume~$kA_X + D$ is effective and non-trivial.
Then the first inequality is obtained by taking the product of~$kA_X + D$ with~$A_X$ 
and follows from Definition~\ref{def:almostDP} as~$\Phi_{|A_X|}$ does not contract divisors.
Similarly, the second is obtained by taking the product of~$kA_X + D$ with~$D$ 
and follows from the nef property of~$A_X$.
\end{proof} 

In what follows we denote by
\begin{equation*}
\Phi_{|A_X|} \colon X \dashrightarrow \PP^{\dim |A_X|}
\end{equation*}
the rational map induced by the fundamental linear system~$|A_X|$, 
which is regular as soon as~$|A_X|$ is base point free 
(as we prove below, this holds if~$\dd(X) \ge 2$).
We call divisors~$Y \subset X$ in~$|A_X|$ \emph{fundamental divisors}.
The following results are well known and can be found in many papers 
(see~\cite[\S2]{Isk:Fano1e}, \cite{Shin1989}, \cite{Fujita-all}, \cite{Fujita-1986-1}, \cite{Fujita1990}).

\begin{proposition}
\label{DP:|A|}
Let $(X,A_X)$ be a del Pezzo variety of dimension~$n$.
Then the following assertions hold.
\begin{enumerate}
\item 
\label{DP:|A|1}
$\dim|A_X|=\dd(X)+n-2$.

\item 
\label{DP:|A|3a}
If $\dd(X)=1$, the linear system~$|A_X|$ has a unique base point~$x_0 \in X$ which is a smooth point of~$X$.
Any fundamental divisor~$Y \in |A_X|$ is smooth at~$x_0 \in X$.

\item 
\label{DP:|A|3}
If $\dd(X)\ge 2$, the linear system~$|A_X|$ is base point free
and the morphism~$\Phi_{|A_X|}$ is finite onto its image.

\item
\label{DP:|A|4}
If $\dd(X)\ge 3$, the linear system~$|A_X|$ is very ample.

\item 
\label{DP:|A|2}
A general fundamental divisor $Y\in |A_X|$ is a del Pezzo variety and~\mbox{$\dd(Y) = \dd(X)$};
if~$Y$ is conical then so is~$X$.
\end{enumerate}
\end{proposition}

\begin{proof}
\ref{DP:|A|1} 
The dimension of~$|A_X|$ can be computed using Riemann--Roch theorem and Lemma~\ref{lem:vanishings}
(see, e.g., \cite[proof of Proposition~1-1]{Alexeev:gd}). 

\ref{DP:|A|3a}, \ref{DP:|A|3} 
We use induction on~$n$ to prove these properties for the wider class of \emph{canonical del Pezzo varieties},
i.e., varieties with canonical singularities satisfying~\eqref{eq:kx-a}.
The base of induction is the case~$n = 2$; 
in this case~$X$ is a del Pezzo surface with canonical (hence Du Val) singularities
and the properties~\ref{DP:|A|3a} and~\ref{DP:|A|3} are well known (see~\cite[\S4]{Hidaka-Watanabe}).
So, assume~$n \ge 3$ and let~$Y \in |A_X|$ be a general fundamental divisor.
The restriction map
\begin{equation*}
H^0(X,\, \OOO_X(A_X)) \longrightarrow H^0(Y,\, \OOO_Y(A_X\vert_Y)) 
\end{equation*}
is surjective because~$H^1(X,\cO_X) = 0$ by Lemma~\ref{lem:vanishings} 
(which holds for canonical del Pezzo varieties by the same argument), 
hence~$\Bs |A_X|= \Bs |A_X\vert_Y|$.
On the other hand, $Y$ has canonical singularities by~\cite[Theorem~0-5]{Alexeev:gd} and
\begin{equation}
\label{eq:ky-kx}
K_Y = (K_X + A_X)\vert_Y = (2 - n)A_X\vert_Y
\end{equation}
by adjunction, therefore, the induction hypothesis holds for~$Y$ 
and we conclude that both~\ref{DP:|A|3a} and~\ref{DP:|A|3} hold for~$Y$,
and hence they also hold for~$X$.

\ref{DP:|A|4}
Here we also apply induction on the dimension as in~\cite{Fujita:def} or~\cite[\S2]{Isk:Fano1e}.

\ref{DP:|A|2}
If~$\dd(X) \ge 2$ then~$|A_X|$ is base point free by~\ref{DP:|A|3}
hence a general fundamental divisor is terminal by~\cite[Lemma 5.17]{KM:book}.
Similarly, if~$\dd(X) = 1$ a general fundamental divisor 
is terminal away from the base point~$x_0 \in X$ of~$|A_X|$ 
and smooth at~$x_0$ by~\ref{DP:|A|3a}, hence it is terminal everywhere.
Moreover, \eqref{eq:ky-kx} still holds, hence~$Y$ is a del Pezzo variety.
Finally, the equality~$\dd(Y) = \dd(X)$ follows from the projection formula.

If~$Y$ is conical then there is a point~$\bv \in Y$ such that 
there is an $(n-2)$-dimensional family of lines on~$Y$ passing through~$\bv$.
Since~$Y$ is general, it follows that lines on~$X$ passing through~$\bv$ form an~$(n-1)$-dimensional family,
hence~$X$ is conical.
\end{proof} 

\begin{remark}
\label{rem:involutions}
If~$\dd(X) = 2$ the morphism~$\Phi_{|A_X|} \colon X \to \PP^n$ is a double covering.
Similarly, if~$\dd(X) = 1$ it is easy to prove that the linear system~$|2A_X|$ is base point free 
and the morphism~$\Phi_{|2A_X|} \colon X \to \PP(1^n,2)$ is a double covering.
The involutions of these double coverings 
are called \emph{Geiser involution} and \emph{Bertini involution}, respectively, \cite{Dolgachev-Iskovskikh}, \cite{P:invol}.
\end{remark}

Applying Proposition~\ref{DP:|A|}\ref{DP:|A|2} several times we construct a chain
\begin{equation*}
S \coloneqq X_2 \subset \dots \subset X_{n-1} \subset X_n = X
\end{equation*}
of del Pezzo varieties~$X_k$ of dimension~$k$, which is useful for various inductive arguments.
Such a chain is called a \emph{ladder}~\cite{Fujita1990}, 
and the surface~$S$ is called a \emph{linear surface section of~$X$}.

\begin{lemma}
\label{lem:s-p-smooth}
If~$X$ is a del Pezzo variety, $\dd(X) \ge 2$, and~$P \in X$ is a smooth point,
a general linear surface section~$S \subset X$ containing~$P$
is a smooth del Pezzo surface of degree~$\dd(S) = \dd(X)$.
\end{lemma}

\begin{proof}
Since~$X$ is terminal, $\codim(\Sing(X)) \ge 3$, 
hence a general linear surface section~\mbox{$S \subset X$} does not intersect~$\Sing(X)$.
Moreover, by Bertini's theorem~$S$ is smooth away from the base locus of the linear system~$|A_X - P|$.

If~$\dd(X) \ge 3$ the class~$A_X$ is very ample by Proposition~\ref{DP:|A|}\ref{DP:|A|4}, 
hence~$\Bs|A_X - P| = \{P\}$ and a general fundamental divisor containing~$P$ is smooth at~$P$ 
because~$P \in X$ is a smooth point.

If~$\dd(X) = 2$, the map~$\Phi_{|A_X|} \colon X \to \PP^n$ is a double covering
and~$\Bs|A_X - P| = \{P,\tau(P)\}$, where~$\tau$ is the Geiser involution.
Moreover, since~$P \in X$ is a smooth point, 
$\Phi_{|A_X|}(P) \in \PP^n$ is either away from the branch divisor, or on its smooth locus.
In either case, a general fundamental divisor containing~$P$ is smooth at~$P$ and~$\tau(P)$.

A combination of the above observations and a simple induction proves the lemma.
\end{proof} 

\begin{remark}
The statement of Lemma~\ref{lem:s-p-smooth} is still true for~$\dd(X) = 1$,
but as we do not need this case, we omit the proof.
\end{remark}

\subsection{Almost del Pezzo varieties}

Recall that almost del Pezzo varieties 
have been defined in Definition~\ref{def:almostDP}
and the anticanonical model~$X_\can$ of an almost del Pezzo variety~$X$ 
has been defined as the image of the morphism~$\Phi_{|mA_X|}$ for~$m$ sufficiently large.

\begin{lemma}
\label{lem:adp-dp}
If~$X$ is an almost del Pezzo variety and~$\xi \colon X \to X_\can$ is the anticanonical morphism, 
then~$X_\can$ is a del Pezzo variety and~$A_X = \xi^*A_{X_\can}$.
Moreover,
\begin{equation*}
X_{\can} \cong \Proj \left( \moplus_{m\ge 0}^{} H^0\big(X, \cO_X(mA_X)\big) \right)
\cong \Proj \left( \moplus_{m\ge 0}^{} H^0\big(X, \cO_X(-mK_X)\big) \right).
\end{equation*}
\end{lemma}

\begin{proof}
By the base point free theorem (\cite[Theorem~3.3]{KM:book}) the class~$A_X$ 
is the pullback of an ample Cartier divisor on~$X_\can$, which we denote by~$A_{X_\can}$.
Since~$\xi$ is small, the singularities of $X_\can$ are terminal and we have
\begin{equation*}
A_{X_\can} = \xi_*(\xi^*A_{X_\can}) = \xi_*A_X
\qquad\text{and}\qquad 
K_{X_\can} = \xi_*K_X,
\end{equation*}
hence~\eqref{eq:kx-a} implies~$K_{X_\can} = (1 - n)A_{X_\can}$, where~$n = \dim(X)$; 
in particular~$X_\can$ is Gorenstein and~$\xi$ is crepant.
Also note that any resolution~$\tX$ of~$X$ is also a resolution of~$X_\can$,
and since~$\xi$ is crepant, the discrepancies of~$\tX$ over~$X$ and over~$X_\can$ are the same,
in particular~$X_\can$ is terminal because~$X$ is.
This proves that~$X_\can$ is a del Pezzo variety.

For the last part note that~$H^0(X, \cO_X(mA_X)) \cong H^0(X_\can, \cO_{X_\can}(mA_{X_\can}))$ for any~$m$
because~$\xi$ induces a bijection preserving linear equivalence between Weil divisors on~$X$ and~$X_\can$,
hence
\begin{equation*}
X_{\can} 
\cong \Proj \left( \moplus_{m\ge 0}^{} H^0(X_\can, \cO_{X_\can}(mA_{X_\can})) \right)
\cong \Proj \left( \moplus_{m\ge 0}^{} H^0(X, \cO_X(mA_X) \right),
\end{equation*}
where the first isomorphism follows from ampleness of~$A_{X_\can}$.
Passing to the $(n-1)$-st Veronese subring in the right hand side, 
we obtain an identification of~$X_{\can}$ with the projective spectrum of the plurianticanonical ring of~$X$.
\end{proof}

\begin{lemma}
\label{lem:dp-adp}
If~$X$ is a del Pezzo variety and~$\xi \colon X' \to X$ is a small birational morphism
then~$X'$ is an almost del Pezzo variety, $X'_\can \cong X$, and
\begin{equation}
\label{eq:dp-adp-invariants}
\dd(X') = \dd(X),
\qquad\text{and}\qquad
\Cl(X') \cong \Cl(X);
\qquad\text{in particular}\quad 
\rr(X') = \rr(X).
\end{equation}
\end{lemma}

\begin{proof}
Since~$\xi$ is small, the singularities of $X'$ are terminal.
Moreover, \eqref{eq:kx-a} implies the equality~$K_{X'} = \xi^*K_X = (1 - n)\xi^*A_X$, 
hence~$A_{X'} \coloneqq \xi^*A_X$ is a fundamental divisor for~$X'$.
It is nef and big by construction; moreover, we have
\begin{equation*}
\Phi_{|mA_{X'}|} = \Phi_{|mA_X|} \circ \xi,
\end{equation*}
also by construction, and since~$\xi$ is small we conclude 
that~$\Phi_{|mA_{X'}|}$ does not contract divisors as soon as~$\Phi_{|mA_X|}$ does not.
Therefore, $X'$ is an almost del Pezzo variety and~$X'_\can \cong X$.
The equalities~\eqref{eq:dp-adp-invariants} follow from smallness of~$\xi$ and projection formula.
\end{proof} 

\begin{definition}
\label{def:q-factorialization}
A \emph{$\QQ$-factorialization} of a normal variety~$X$ is a proper birational morphism~\mbox{$\xi \colon \hX\to X$} 
such that~$\hX$ is $\QQ$-factorial and~$\xi$ is small.
\end{definition}

Note that any $\QQ$-factorialization of a del Pezzo variety is almost del Pezzo by Lemma~\ref{lem:dp-adp}.
Note also that by~\cite[Corollary~1.4.3]{BCHM} a $\QQ$-factorialization exists 
for any variety with at worst log terminal singularities;
in particular, it exists for del Pezzo varieties.

The following lemma uses the notion of pseudoisomorphism defined in the Introduction. 

\begin{lemma}
\label{lem:dp-weak-dp} 
If~$X' \dashrightarrow X''$ is a pseudoisomorphism of almost del Pezzo varieties then
there is an isomorphism~$X'_\can \cong X''_\can$ compatible with the anticanonical
morphisms of~$X'$ and~$X''$.
Conversely, if~$X'$ and~$X''$ are $\QQ$-factorializations of~$X$ then~$X'$ 
is pseudoisomorphic to~$X''$ over~$X$.
\end{lemma} 

\begin{proof}
Let~$\psi \colon X' \dashrightarrow X''$ be a pseudoisomorphism.
Since~$\psi$ is small, $\psi_*$ induces a bijection 
compatible with linear equivalence of Weil divisors on~$X'$ and~$X''$,
which takes~$K_{X'}$ to~$K_{X''}$. 
It follows that
we have isomorphisms
\begin{equation*}
H^0\big(X', \cO_{X'}(-mK_{X'})\big) \cong H^0\big(X'', \cO_{X''}(-mK_{X''})\big)
\end{equation*}
for all~$m$, hence~$X'_\can \cong X''_\can$ by Lemma~\ref{lem:adp-dp}.

The second claim is obvious --- if~$\xi' \colon X' \to X$ and~$\xi'' \colon X'' \to X$ are $\QQ$-factorializations,
the map~$(\xi'')^{-1} \circ \xi'$ is a pseudoisomorphism 
because neither~$\xi'$ nor~$\xi''$ contract divisors.
\end{proof} 

\begin{definition}
\label{def:crepant-model} 
A \emph{crepant model} of an almost del Pezzo variety~$X$
is an almost del Pezzo variety~$X'$ pseudoisomorphic to~$X$.
\end{definition}

\begin{lemma}
\label{lemma:repant-models}
Any almost del Pezzo variety~$X$ has a finite number of crepant models.
Any crepant model of~$X$ is a small contraction over~$X_\can$ of a $\QQ$-factorialization of~$X_\can$.
Any pair of crepant $\QQ$-factorial models are connected by a sequence of flops over~$X_\can$.
\end{lemma}

\begin{proof}
Any almost del Pezzo variety~$X$ is an FT (``Fano type'') variety~\cite[Lemma-Definition~2.6]{P-Sh:JAG}, 
i.e., there is a boundary~$B$ such that the pair~$(X,B)$ is a klt log Fano variety. 
Moreover, if $X$ is $\QQ$-factorial, then it is a Mori dream space~\cite[Corollary~1.3.2]{BCHM},
hence it has only a finite number of crepant models~\cite[Definition~1.10 and Proposition~1.11]{hu-keel:mds}. 

Let~$X'$ be a crepant model of~$X$ and let~$X''$ be a $\QQ$-factorialization of~$X'$.
Composing a pseudoisomorphism~$X' \dashrightarrow X$ with a $\QQ$-factorialization morphism~$X'' \to X'$,
we obtain a pseudoisomorphism~$\psi \colon X'' \dashrightarrow X$ of almost del Pezzo varieties.
Using Lemma~\ref{lem:dp-weak-dp} we conclude that~$X_\can \cong (X'')_\can$
hence~$X''$ is a $\QQ$-factorialization of~$X_\can$,
and hence~$X'$ is a small contraction of a $\QQ$-factorialization.

The last assertion follows from~\cite[Corollary~1.1.3]{BCHM}.
\end{proof}

The following observation is a partial converse to Proposition~\ref{DP:|A|}\ref{DP:|A|2};
its second part is extremely useful for verification of the almost del Pezzo property.

\begin{lemma}
\label{lem:criterion-divisor}
Let~$X$ be a Gorenstein variety of dimension~$n \ge 3$ such that the equality
\begin{equation*}
-K_X = (n - 1)A_X
\end{equation*}
holds for a Cartier divisor class~$A_X$.
\begin{enumerate}
\item 
\label{it:positive-divisors-dp}
The variety~$X$ is almost del Pezzo if and only if~$X$ is terminal, $A_X$ is nef and big, 
and~$A_X^{n-1}\cdot D > 0$ for any effective non-trivial divisor~$D \subset X$.
\item 
\label{it:divisor-dp}
Let~$Y \subset X$ be a divisor in the linear system~$|A_X|$.
If~$(Y,A_X\vert_Y)$ is an almost del Pezzo variety, 
$X$ has terminal singularities away from~$Y$,
and there are no effective divisors on~$X$ disjoint from~$Y$
then~$X$ is also an almost del Pezzo variety.
\end{enumerate}
\end{lemma} 

\begin{proof}
\ref{it:positive-divisors-dp}
If~$X$ is an almost del Pezzo variety we only need to check the inequality.
Since the morphism~$\xi \colon X \to X_\can$ does not contract divisors 
and~$A_X = \xi^*A_{X_\can}$, we have
\begin{equation*}
A_X^{n-1} \cdot D = (A_{X_\can})^{n-1} \cdot \xi(D) > 0
\end{equation*}
because~$A_{X_\can}$ is ample.
Conversely, the morphism~$\Phi_{|mA_X|}$ cannot contract a divisor~$D$ 
because~\mbox{$A_X^{n-1} \cdot D > 0$},
hence~$X$ is an almost del Pezzo variety.

\ref{it:divisor-dp}
First of all, $X$ has terminal singularities by~\cite[Theorems~9.1.14]{Ishii:book}.
Further, consider the exact sequence of sheaves
\begin{equation*}
0 \longrightarrow \cO_X((m-1)A_X) \longrightarrow \cO_X(mA_X) \longrightarrow \cO_Y(mA_X\vert_Y) \longrightarrow 0.
\end{equation*}
Since~$H^1(Y,\cO_Y(mA_X\vert_Y)) = 0$ for all~$m$ by Lemma~\ref{lem:vanishings}
and~$H^1(X,\cO_X(mA_X)) = 0$ for all~\mbox{$m \ll 0$} (by Serre vanishing), 
a simple induction proves that~$H^1(X,\cO_X(mA_X)) = 0$ for all~$m$,
and hence the restriction morphism
\begin{equation*}
H^0(X,\cO_X(mA_X)) \longrightarrow H^0(Y,\cO_Y(mA_X\vert_Y)) 
\end{equation*}
is surjective for all~$m$.
Since~$mY \in |mA_X|$, the base locus of~$|mA_X|$ is contained in~$Y$, 
and since for~$m \gg 0$ the linear system~$|mA_X\vert_Y|$ is base point free
(because~$Y$ is an almost del Pezzo variety),
it follows that~$|mA_X|$ is base point free as well, hence~$A_X$ is nef.
Moreover, we have~$A_X^n = (A_X\vert_Y)^{n-1} > 0$, hence~$A_X$ is also big.
It remains to note that 
\begin{equation*}
A_X^{n-1} \cdot Y = (A_X\vert_Y)^{n-1}
\qquad\text{and}\qquad 
A_X^{n-1} \cdot D = (A_X\vert_Y)^{n-2} \cdot (D \cap Y)
\end{equation*}
for any irreducible effective divisor~$D \subset X$ not contained in~$Y$.
Applying part~\ref{it:positive-divisors-dp}, 
the almost del Pezzo property of~$Y$,
and the fact that~$D \cap Y \ne \varnothing$,
we conclude that both these numbers are positive, hence~$X$ is an almost del Pezzo variety,
again by part~\ref{it:positive-divisors-dp}.
\end{proof}

Recall that the blowup of a general point on a del Pezzo surface of degree at least~$2$ is again a del Pezzo surface.
We prove below that almost the same holds for del Pezzo varieties, and make the generality assumptions explicit. 

\begin{proposition}
\label{prop:dP4:constr-i}
Let~$(X, A)$ be an almost del Pezzo variety of dimension~$n$.
Let~$P\in X$ be a smooth point and let $\sigma \colon \tX\to X$ be the blowup of~$P$ with exceptional divisor~$E$.
Then
\begin{equation*}
K_{\tX} = (1 - n)A_{\tX},
\qquad\text{where}\qquad 
A_{\tX} = \sigma^*A_X - E,
\end{equation*}
and~$\Cl(\tX) = \Cl(X) \oplus \ZZ E$.

Moreover, if~$X$ is del Pezzo and~\mbox{$\dd(X)\ge 2$}, 
then~$\tX$ is almost del Pezzo if and only if
\begin{enumerate}
\item\label{prop:dP4:constr-ia}
the family of lines in~$X$, passing through~$P$ has dimension at most~$n-3$, and
\item\label{prop:dP4:constr-ib}
if $\dd(X)=2$, the point~$P$ does not lie on the ramification divisor of the 
double covering~\mbox{$\Phi_{|A_X|} \colon X \to \PP^n$} given by~$|A_X|$.
\end{enumerate}
Finally, $\dd(\tilde X) = \dd(X)-1$ and~$\rr(\tilde X) = \rr(X)+1$.
\end{proposition}

\begin{proof}
The first part follows immediately from the blowup formula for the canonical class.

Assume the conditions~\ref{prop:dP4:constr-ia} and~\ref{prop:dP4:constr-ib} are satisfied.
Recall that by Lemma~\ref{lem:s-p-smooth} a general linear surface section~\mbox{$S \subset X$} through~$P$ is smooth.
Moreover, $\dd(S) = \dd(X) \ge 2$, $P$ does not lie on a line in~$S$, and in the case where~$\dd(S) = 2$, 
$P$ does not lie on the ramification divisor of~$\Phi_{|A_S|}$.
Therefore, $\Bl_P(S)$ is a del Pezzo surface (see, e.g., \cite[Lemma~4.3]{Yas}).
Now the equality~$A_{\tX} = \sigma^*A_X - E$ shows that~$\tS \coloneqq\Bl_P(S)$ 
is a linear surface section of~$\tX$.
We check below that any effective divisor~$D \subset \tX$ has a non-trivial intersection with~$\tS$.

Indeed, if~$D \cap E \ne \varnothing$, then~$D \cap E \subset E \cong \PP^{n-1}$ is a hypersurface,
while~$\tS \cap E \subset E$ is a line, hence the intersection~$D \cap E \cap \tS$ is nontrivial.
On the other hand, if~\mbox{$D \cap E = \varnothing$}, 
then~\mbox{$D = \sigma^{-1}(D_0)$} for a divisor~$D_0 \subset X$,
and~$D_0 \cap S \ne \varnothing$ because~$X$ is almost del Pezzo, hence~$D \cap \tS \ne \varnothing$ as well.

Now applying Lemma~\ref{lem:criterion-divisor}\ref{it:divisor-dp} several times 
we see that~$\tX$ is an almost del Pezzo variety.

Conversely, if~$\tX$ is an almost del Pezzo variety then~$\Bl_P(S)$ must be a del Pezzo surface,
hence $P$ does not lie on a line in~$S$, and in the case where~$\dd(S) = 2$, 
$P$ does not lie on the ramification divisor of~$\Phi_{|A_S|}$.

The last part of the proposition is obvious.
\end{proof} 

\begin{remark}
One possibility for condition~\ref{prop:dP4:constr-ia} to fail (similarly to the case of surfaces) 
is if~$X$ itself is a blowup of a smooth point on~$X'$ and~$P$ lies on the exceptional divisor of this blowup.
However this is not the only possibility;
another example is if~\mbox{$P \in X$} is an \emph{Eckardt point}, i.e., the vertex of a conical divisor in~$X$
(this is possible only for $\dd(X)\le 3$).
\end{remark} 

\subsection{Minimal model program for almost del Pezzo varieties}
\label{ss:contractions}

In this section we describe $K$-negative extremal contractions of almost del Pezzo varieties.
Recall that a \emph{contraction} is a proper surjective morphism with connected fibers to a normal scheme.
We say that a contraction~$f \colon X \to Z$ is \emph{$K$-negative} if~$-K_X$ is $f$-ample,
and \emph{extremal}, if the relative Picard number of~$f$ is~$1$.
We start with birational contractions.

\begin{lemma} 
\label{lemma:ext-rays0}
Let~$X$ be a $\QQ$-factorial almost del Pezzo variety
and let~\mbox{$f \colon X\to X'$} be a $K$-negative extremal birational contraction.
Then~$X'$ is a $\QQ$-factorial almost del Pezzo variety of degree~$\dd(X') = \dd(X) + 1$,
$f$ is the blowup of a smooth point, and the exceptional divisor~$E$ of~$f$ is isomorphic to~$\PP^{n-1}$ and satisfies
\begin{equation}
\label{eq:coe-e}
\cO_E(E) \cong \cO_{\PP^{n-1}}(-1)
\qquad\text{and}\qquad
\cO_E(A_X\vert_E) \cong \cO_{\PP^{n-1}}(1).
\end{equation} 
Moreover, the point~$P \coloneqq f(E)\in X'$ does not lie on a $K$-trivial curve.
\end{lemma}

\begin{proof}
We apply~\cite[Theorem~1.1]{Andreatta-Tasin}.
Let~$\rR$ be the extremal ray contracted by~$f$.
Since~$A_X$ is $f$-ample, we can choose a sufficiently ample divisor~$H$ on~$X'$ 
so that the divisor~$M = A_X + f^*H$ is ample.
Then
\begin{equation*}
(K_X + (n - 2)M) \cdot \rR = -A_X \cdot \rR +(n - 2)f^*H\cdot \rR = -A_X \cdot \rR < 0
\end{equation*}
because~$\rR$ is contracted by~$f$, so the assumptions of~\cite[Theorem~1.1]{Andreatta-Tasin} are satisfied,
and therefore~$f$ is a weighted blowup of a smooth point with weights~$(1,1,b,\dots,b)$, where~\mbox{$b > 0$}.
But as~$X$ is Gorenstein, we have~$b = 1$.

Let~$E \subset X$ be the exceptional divisor.
Then~$E \cong \PP^{n-1}$ and the first isomorphism in~\eqref{eq:coe-e} holds.
Moreover, by Proposition~\ref{prop:dP4:constr-i} we have~$A_X = f^*A_{X'} - E$;
restricting this equality to~$E$,
we obtain the second isomorphism in~\eqref{eq:coe-e}.

Now it only remains to show that~$X'$ is almost del Pezzo.
For this we consider a general linear surface section~$S \subset X$ and set~$S' \coloneqq f(S) \subset X'$.
Then~$S$ is a del Pezzo surface and the morphism~$f\vert_S \colon S \to S'$ is birational 
and contracts the line~$L = S \cap E$.
Therefore, $S'$ is a smooth del Pezzo surface of degree
\begin{equation*}
\dd(S') = \dd(S) + 1 = \dd(X) + 1,
\end{equation*}
(so that~$S$ is the blowup of a point on~$S'$) and~$S' \subset X'$ is a linear surface section.
If~$D \subset X'$ is a divisor disjoint from~$S'$, 
then~$f^{-1}(D) \subset X$ is a divisor disjoint from~$S$, which is impossible as~$X$ is almost del Pezzo.
Therefore, applying Lemma~\ref{lem:criterion-divisor}\ref{it:divisor-dp} several times, 
we conclude that~$X'$ is almost del Pezzo.

Finally, if~$P$ lies on a $K$-trivial curve~$C' \subset X'$ and~$C \subset X$ is its strict transform, then
\begin{equation*}
K_X \cdot C = (f^*K_{X'} + (n-1)E) \cdot C = K_{X'} \cdot C' + (n - 1)E \cdot C = (n - 1)E \cdot C > 0,
\end{equation*}
which contradicts to the nef property of~$-K_X = (n-1)A_X$.
\end{proof} 

Next, we describe non-birational $K$-negative extremal contractions. 
Recall Definition~\ref{def:dpb} of del Pezzo bundles.

\begin{lemma}
\label{lemma:ext-rays0a}
Let~$X$ be a $\QQ$-factorial almost del Pezzo variety, $\dim(X) \ge 3$ and~\mbox{$\rr(X)>1$}, 
and let~\mbox{$f \colon X\to Z$} be a $K$-negative extremal contraction to a lower-dimensional variety.
Then $\dim (Z)\le 2$ and 
\begin{enumerate}
\item 
\label{lemma:ext-rays0a2}
if $\dim (Z)=2$, then~$Z$ is a smooth del Pezzo surface
and there is a vector bundle~$\cE$ on~$Z$ with~$\rc_1(\cE) = K_Z$ such that~$X \cong \PP_Z(\cE)$;
moreover, if~$X$ is non-conical then~$\cE$ is a del Pezzo bundle on~$Z$;
\item 
\label{lemma:ext-rays0a1}
if $\dim (Z)=1$, then $Z\cong \PP^1$ and~$f \colon X \to \PP^1$ is a flat quadric bundle.
\end{enumerate}
\end{lemma}

\begin{proof}
Let~$F$ be the geometric general fiber~of~$f$. 
By adjunction we have
\begin{equation}
\label{eq:kf}
-K_F=(n-1)A_X\vert_F,
\end{equation} 
hence~$F$ is a Fano variety of positive dimension with Gorenstein terminal singularities. 
Applying~\cite[Theorem~3.1.14]{IP99} we conclude that~$\dim(F) \ge n - 2$, hence~$\dim(Z) \le 2$. 

\ref{lemma:ext-rays0a2}
Let $\dim(Z) = 2$. 
Since~$\dim(F) = n - 2$, it follows from~\eqref{eq:kf}, ampleness of~$-K_F$, and~\cite[Theorem~3.1.14]{IP99} 
that~$F \cong \PP^{n-2}$ and~$\cO_X(A_X)\vert_F \cong \cO_{\PP^{n-2}}(1)$.

Let~$S \subset X$ be a general linear surface section.
It follows from the description of~$F$ that
the morphism~$f\vert_S \colon S\to Z$ is birational,
and since~$S$ is a smooth del Pezzo surface,
$f\vert_S$ is an iterated contraction of~$(-1)$-curves, 
hence~$Z$ is a smooth del Pezzo surface as well.

Since~$f$ is extremal and~$X$ is $\QQ$-factorial, all fibers of~$f$ have dimension~$n - 2$,
therefore the sheaf~$f_*\cO_X(A_X)$ is reflexive by~\cite[Corollary~1.7]{Har80},
and since~$Z$ is a smooth surface, it is locally free by~\cite[Corollary~1.4]{Har80}.
Denote
\begin{equation*}
\cE \coloneqq (f_*\cO_X(A_X))^\vee.
\end{equation*}
We claim that the natural morphism
\begin{equation}
\label{eq:fs-fs-oxa}
f^*\cE^\vee = f^*f_*\cO_X(A_X) \longrightarrow \cO_X(A_X)
\end{equation} 
is surjective.
If~$\dd(X) \ge 2$ this is obvious, because the composition
\begin{equation*}
H^0(X,\cO_X(A_X)) \otimes \cO_X 
= H^0(Z, f_*\cO_X(A_X)) \otimes \cO_X
\longrightarrow f^*f_*\cO_X(A_X) \longrightarrow \cO_X(A_X)
\end{equation*}
coincides with the natural epimorphism~$H^0(X,\cO_X(A_X)) \otimes \cO_X \to \cO_X(A_X)$, 
hence the last arrow is surjective. 
So assume~$\dd(X) = 1$.
Then the same argument shows that~\eqref{eq:fs-fs-oxa} 
is surjective away from the base point~$x_0 \in X$ of~$|A_X|$,
so it remains to check the surjectivity at~$x_0$.
Consider the restriction morphism~$\cO_X(A_X) \to \cO_{x_0}$; we will check that its pushforward
\begin{equation}
\label{eq:fs-coxa-coz0}
f_*\cO_X(A_X) \longrightarrow \cO_{f(x_0)}
\end{equation}
is still surjective.
For this we take a general linear surface section~$S \subset X$ (then~$x_0 \in S$)
and consider the composition~$\cO_X(A_X) \to \cO_S(A_X\vert_S) = \cO_S(-K_S) \to \cO_{x_0}$, and its pushforward
\begin{equation*}
f_*\cO_X(A_X) \longrightarrow f_*\cO_S(-K_S) \longrightarrow \cO_{f(x_0)}.
\end{equation*}
Here the first arrow is surjective because~$R^{>0}f_*\cO_X(-k A_X) = 0$ for~$0 \le k \le n - 3$ 
by Kawamata--Vieweg vanishing~\cite[Theorem~1-2-5, Remark~1-2.6]{KMM},
and the second arrow is surjective, because~$f\vert_S \colon S \to Z$ is a birational morphism of del Pezzo surfaces,
hence it is an isomorphism near the base point~$x_0 \in S$ of~$|-K_S|$,
and the surjectivity of~\eqref{eq:fs-coxa-coz0} follows.
Now consider the commutative diagram
\begin{equation*}
\xymatrix{
f^*f_*\cO_X(A_X) \ar[r] \ar[d] & 
\cO_X(A_X) \ar[d]
\\
f^*\cO_{f(x_0)} \ar[r] &
\cO_{x_0}.
}
\end{equation*}
Its left vertical arrow is surjective because~\eqref{eq:fs-coxa-coz0} is, 
its bottom arrow is surjective for obvious reasons,
and its right vertical arrow is an isomorphism at~$x_0$.
Therefore, its top arrow is surjective at~$x_0$.

Now the surjectivity of~\eqref{eq:fs-fs-oxa} implies that the morphism~$f$ factors as the composition 
\begin{equation*}
X \xrightarrow{\ f'\ } 
X' \coloneqq \PP_Z(\cE) \xrightarrow{\ f''\ } 
Z,
\end{equation*}
where~$f'$ is a regular morphism and~$f''$ is the projection.
The restriction of~$f'$ to~$F$ is an isomorphism onto the general fiber of~$f''$, hence~$f'$ is birational,
and since~$f = f'' \circ f'$ is extremal, $f'$ is an isomorphism.

The standard formula for the canonical class of a projective bundle gives
\begin{equation*}
(1-n)A_X = K_X = K_{\PP_Z(\cE)} = f^*(K_Z - \rc_1(\cE)) + (1-n)A_X,
\end{equation*}
hence~$\rc_1(\cE) = K_Z$.
Finally, if~$X$ is non-conical, we conclude that~$\cE$ is a del Pezzo bundle.

\ref{lemma:ext-rays0a1}
Let~$\dim(Z) = 1$. 
Since~$\dim(F) = n - 1$, 
it follows from~\eqref{eq:kf}, ampleness of~$-K_F$, and~\cite[Theorem~3.1.14]{IP99} 
that~$F$ is a quadric in~$\PP^n$ and~$\cO_X(A_X)\vert_F \cong \cO_{\PP^{n}}(1)\vert_F$. 

Let~$S \subset X$ be a general linear surface section.
It follows from the description of~$F$ that
the morphism~$f\vert_S \colon S \to Z$ is dominant,
hence~$Z$ is unirational and normal, hence~$Z \cong \PP^1$.
Furthermore, the sheaf~$f_*\cO_X(A_X)$ on~$\PP^1$ is torsion free, hence locally free,
and defining~$\cF$ as its dual and arguing as in part~\ref{lemma:ext-rays0a2}
we prove that~$f$ factors as 
\begin{equation*}
X \xrightarrow{\ f'\ } 
X' \hookrightarrow 
\PP_Z(\cF) \xrightarrow{\ f''\ } 
Z.
\end{equation*}
Now the restriction of~$f'$ to~$F$ is the natural embedding~$F \subset \PP^n$,
hence the image~$X'$ of~$f'$ is a divisor in~$\PP_Z(\cF)$ of relative degree~$2$,
and since~$f$ is extremal, $f'$ is an isomorphism onto this divisor.
Finally, $X$ is irreducible, hence flat over~$Z = \PP^1$.
\end{proof} 

Summarizing the results of Lemma~\ref{lemma:ext-rays0} and Lemma~\ref{lemma:ext-rays0a}
we obtain the main result of this section:

\begin{proposition}
\label{propo:ext-rays}
Let~X be a $\QQ$-factorial almost del Pezzo variety such that~\mbox{$\rr(X)>1$}
and let~$f \colon X\to Z$ be an extremal $K$-negative contraction.
Then 
\begin{enumerate}
\item
\label{prop:ext-rays1}
either $Z$ is a $\QQ$-factorial almost del Pezzo variety 
and~$f \colon X\to Z$ is the blowup of a smooth point on~$Z$ not lying on a $K$-trivial curve,

\item
\label{prop:ext-rays2}
or~$Z$ is a smooth del Pezzo surface and~$X \cong \PP_Z(\cE)$, 
where~$\cE$ is a vector bundle on~$Z$ with~$\rk(\cE) \ge 2$ and~$\rc_1(\cE) = K_Z$;
moreover, if~$X$ is non-conical then~$\cE$ is a del Pezzo bundle,

\item
\label{prop:ext-rays3}
or~$Z\cong\PP^1$ and $X \to Z$ is a flat quadric bundle.
\end{enumerate}
\end{proposition}

We conclude this section with a few useful Minimal Model Program results.

\begin{corollary}
\label{cor:targets}
Let~$X$ be an almost del Pezzo variety.
If~$f \colon X \to Z$ is any contraction 
then either~$f$ is birational and~$Z$ is almost del Pezzo, 
or~$Z$ is a del Pezzo surface, or~$\PP^1$, or a point.
\end{corollary}

\begin{proof}
Let~$\xi \colon \hX \to X$ be a $\QQ$-factorialization;
$\hX$
is a $\QQ$-factorial almost del Pezzo variety.
Running the Minimal Model Program on~$\hX$ over~$Z$,
i.e., contracting $K$-negative extremal rays over~$Z$,
we obtain a birational map~$\hX \dashrightarrow X'$ over~$Z$
such that~$X'$ is an almost del Pezzo variety 
and we have one of the following two situations (see~\cite[\S2.14]{KM:book}):

\renewcommand\labelenumi{\rm (\Alph{enumi})}
\renewcommand\theenumi{\rm (\Alph{enumi})}
\begin{enumerate}
\item
There exists a non-birational $K$-negative extremal contraction~$X' \to Z'$ over~$Z$.
\item 
The divisor~$K_{X'}$ is nef over~$Z$.
\end{enumerate}
In the first case~$Z'$ is a del Pezzo surface, or~$\PP^1$, or a point by Proposition~\ref{propo:ext-rays}, 
and hence the same holds for $Z$.
In the second case, since~$X'$ is almost del Pezzo, the divisor~$-K_{X'}$ is nef, 
and since~$K_{X'}$ is nef over~$Z$, it is numerically trivial over~$Z$, 
hence the map~$f' \colon X' \to Z$ is $K$-trivial.
Therefore, the curves contracted by~$f'$ are $K_{X'}$-trivial;
in other words, they form a subset of the set of curves contracted by the anticanonical morphism~$X' \to X'_\can$.
Thus, the anticanonical morphism factors through~$f'$, hence
$Z$ is an almost del Pezzo variety with~$Z_\can = X'_\can$;
in particular, $Z$ is birational to~$X$.
\end{proof} 

\begin{lemma}
\label{lem:d-mmp}
Let~$X$ be a del Pezzo variety.
For any class~$D \in \Cl(X) \otimes \RR$ there exists a $\QQ$-factorialization~$\xi \colon \hX \to X$ 
such that every $(\xi^{-1})_*D$-negative extremal ray is $K$-negative.
In particular, we have the following alternative:
\renewcommand\labelenumi{\rm (\Alph{enumi})}
\renewcommand\theenumi{\rm (\Alph{enumi})}
\begin{enumerate}
\item
\label{cone:case1}
There exists a $(\xi^{-1})_*D$-negative $K$-negative extremal contraction $f \colon \hX\to Z$;
moreover, if~$D$ is pseudoeffective then~$f$ is birational.
\item 
\label{cone:case2}
The divisor~$(\xi^{-1})_*D$ is nef.
\end{enumerate}
\end{lemma} 

\begin{proof} 
Let~$\tX \to X$ be any $\QQ$-factorialization and let~$\tD$
be the strict transform of~$D$.
Note that~$\tX$ is an FT variety \cite[Lemma-Definition~2.6]{P-Sh:JAG},
therefore we can run $\tD$-MMP on~$\tX$ in the category of FT 
varieties \cite[Corollary~2.7,\ Lemma~2.8]{P-Sh:JAG}.
At each step we flop a~$\tD$-negative $K$-trivial extremal ray;
clearly this does not spoil the nef and big properties of the anticanonical class,
hence the resulting variety is still almost del Pezzo.
After a finite number of such flops we obtain a $\QQ$-factorial crepant model~$\hX$ of~$X$ 
such that every $\hD$-negative extremal ray is $K$-negative, where~$\hD$ is the strict transform of~$\tD$.
By Lemma~\ref{lemma:repant-models} the variety~$\hX$ is a $\QQ$-factorialization of~$X$,
hence~$\hD = (\xi^{-1})_*D$, where~$\xi \colon \hX \to X$ is the anticanonical morphism.
Now obviously we are in the case~\ref{cone:case1} (if there are~$\hD$-negative curves) 
or otherwise in the case~\ref{cone:case2}.
\end{proof}   

\section{Lattice structure of the class group}
\label{sec:lattice}

Recall the bilinear form~\eqref{eq:product} defined on the class group~$\Cl(X)$ of any almost del Pezzo variety.
In this section we investigate its properties.

\subsection{Restriction morphisms}

In this subsection we discuss the restriction morphisms between the class groups of a ladder of (almost) del Pezzo varieties.

\begin{lemma}
\label{lem:product-divisor}
For any almost del Pezzo variety~$X$
the bilinear form~\eqref{eq:product} is well defined.
Moreover, if~$i \colon Y \hookrightarrow X$ is a general fundamental divisor
and~\mbox{$i^* \colon \Cl(X) \to \Cl(Y)$} is the induced morphism then
\begin{equation*}
\langle D_1, D_2 \rangle = \langle i^*D_1,\, i^*D_2\rangle
\end{equation*}
for any~$D_1,D_2 \in \Cl(X)$. 
\end{lemma}

\begin{proof}
All results follow from the projection formula because the base locus of~$|A_X|$ is empty or zero-dimensional 
and~$Y \subset X$ is a Cartier divisor.
\end{proof} 

\begin{lemma}
\label{lem:pseudo-clx}
If~$\psi \colon X' \dashrightarrow X''$ is a pseudoisomorphism of almost del Pezzo varieties 
then the map~$\psi_* \colon \Cl(X') \to \Cl(X'')$ is an isomorphism compatible with the bilinear form~\eqref{eq:product}.
\end{lemma}

\begin{proof}
By Lemma~\ref{lem:dp-weak-dp} we have~$X'_\can \cong X''_\can$, 
and if~$S$ is a general linear surface section of that del Pezzo variety, 
then~$S$ is also a linear section of both~$X'$ and~$X''$.
Therefore, applying Lemma~\ref{lem:product-divisor} several times we obtain
\begin{equation*}
\langle D_1,\, D_2 \rangle_{X'} =
\langle D_1\vert_S,\, D_2\vert_S \rangle_{S} =
\langle \psi_*D_1,\, \psi_*D_2 \rangle_{X''} 
\end{equation*}
for any Weil divisor classes~$D_1,D_2 \in \Cl(X')$.
\end{proof} 

The following result provides a step towards Theorem~\ref{thm:intro-clx}.
In the case $\dd(X)\ge 2$ its first part follows from~\cite{Ravindra-Srinivas},
but we provide a uniform simple proof.

\begin{proposition}
\label{prop:class-divisor}
Let~$X$ be a del Pezzo variety of dimension~$n \ge 3$.
Consider a ladder
\begin{equation*}
S = X_2 \xhookrightarrow{\ i_2\ } X_3 \xhookrightarrow{\ i_3\ } \cdots \xhookrightarrow{\ i_{n-2}\ } X_{n-1} \xhookrightarrow{\ i_{n-1}\ } X_n = X
\end{equation*}
of linear sections \textup(where~$X_2 = S$ is a del Pezzo surface\textup) and the chain of linear maps
\begin{equation*}
\Cl(X) = \Cl(X_n) \xrightarrow{\ i_{n-1}^*\ } \Cl(X_{n-1}) \xrightarrow{\ i_{n-2}^*\ } \dots 
\xrightarrow{\ i_3^*\ } \Cl(X_3) \xrightarrow{\ i_2^*\ } \Cl(X_2) = \Cl(S).
\end{equation*}
The maps~$i_k^*$ are 
isomorphisms
for~$k \ge 3$, while~$i_2^*$ is an embedding.
In particular, if
\begin{equation*}
\Xi(X) \coloneqq \Cl(X)^\perp \subset K_S^\perp \subset \Cl(S)
\end{equation*}
is the orthogonal complement of the image of the composition map~$\Cl(X) \to \Cl(S)$ then
\begin{equation*}
\Xi(X_n) = \Xi(X_{n-1}) = \dots = \Xi(X_3)
\end{equation*}
is a negative definite sublattice of rank~$10 - \dd(X) - \rr(X)$ in~$K_S^\perp \subset \Cl(S)$.
Moreover,
\begin{equation}
\label{eq:clx-xix}
\Cl(X) = \Xi(X)^\perp
\end{equation}
is a non-degenerate lattice of signature~$(1,\rr(X) - 1)$.
\end{proposition} 

As the proof given below shows, the sublattice~$\Xi(X) \subset \Cl(S)$ is generated 
by the vanishing cycles of a Lefschetz pencil of hyperplane sections of~$X_3$,
and since such a pencil must have singular fibers, $\Xi(X) \ne 0$.

\begin{proof}
Let~$C \subset S$ be a general anticanonical divisor; this is a smooth elliptic curve.
Let~$\tX \coloneqq \Bl_C(X)$ be the blowup and let
\begin{equation*}
\rho\colon \tX \longrightarrow \PP^{n-2}
\end{equation*}
be the morphism induced by the linear system~$|A_X - C|$.
Then all fibers of~$\rho$ are irreducible surfaces (because~$C$ is irreducible) 
and the general fiber is a del Pezzo surface of degree~$\dd(X)$.
Let~$U \subset \PP^{n-2}$ be the complement of the locus of singular fibers of~$\rho$.
Then after restriction we obtain a family
\begin{equation*}
\rho_U \colon \tX_U \coloneqq \rho^{-1}(U) \longrightarrow U
\end{equation*}
of smooth del Pezzo surfaces.
Let~$\Pic_{\tX_U/U}$ denote the \'etale sheafification of the relative Picard group.
Note that~$S$ is one of the fibers of~$\rho_U$; we denote by~$u_0 \in U$ the point such that~$\tX_{u_0} = S$.
Then by~\cite[Corollary~2.3]{K22} we have an isomorphism
\begin{equation}
\label{eq:relative-pic}
\Pic_{\tX_U/U}(U) \cong \Pic(S)^{\uppi_1(U,u_0)},
\end{equation}
where~$\uppi_1(U,u_0)$ is the \'etale fundamental group of~$U$ 
acting on~$\Pic(S)$ by monodromy (see~\cite[\S2.1]{K22}).

Now consider the commutative diagram with exact rows
\begin{equation*}
\xymatrix{
0 \ar[r] & 
\Pic(\PP^{n-1}) \ar[d] \ar[r]^-{\rho^*} &
\Cl(\tX) \ar[r] \ar[d] &
\Cl(X) \ar[r] \ar[d] &
0
\\
0 \ar[r] & 
\Pic(U) \ar[r]^-{\rho_U^*} & 
\Pic(\tX_U) \ar[r] & 
\Pic_{\tX_U/U}(U) \ar[r] &
0.
}
\end{equation*}
Here the top row is obtained from the direct sum decomposition
\begin{equation*}
\Cl(\tX) = \Cl(X) \oplus \ZZ E,
\end{equation*}
where~$E \subset \tX$ is the exceptional divisor of the blowup~$\tX \to X$, 
because the pullback to~$\tX$ of the hyperplane class of~$\PP^{n-1}$ is equal to the class~$A_X - E$
(thus, the second arrow is identical on the first summand~$\Cl(X)$ 
and takes the generator~$E$ of the second summand to~$A_X$).
Furthermore, the bottom row is the sequence from~\cite[Proposition~2.5]{Liedtke} 
(it is exact on the right because any point of~$C$ gives a section for~$\tX_U/U$).
Finally, the vertical arrows are the restriction maps.
Clearly, the left and middle vertical arrows in the diagram are surjective 
with kernels generated by the divisorial irreducible components of~$\PP^{n-1} \setminus U$ and of~$\tX \setminus \tX_U$, respectively,
hence~$\rho^*$ induces an isomorphism of the kernels.
It follows that the right vertical arrow is an isomorphism.
Combining this with~\eqref{eq:relative-pic}, we conclude that the restriction map induces an isomorphism
\begin{equation}
\Cl(X) \cong \Pic(S)^{\uppi_1(U,u_0)}.
\end{equation} 
In particular, $\Cl(X)$ is a saturated subgroup in~$\Pic(S) = \Cl(S)$, 
hence~$\Cl(X) = (\Cl(X)^\perp)^\perp$,
and so~\eqref{eq:clx-xix} holds.

Now assume~$n \ge 4$ and consider the general fundamental divisor~$i_{n-1} \colon X_{n-1} \hookrightarrow X$ 
containing~$S$, hence also~$C$.
Let~$H = \rho(\Bl_C(X_{n-1})) \subset \PP^{n-1}$ be the corresponding hyperplane
and set~$U_H \coloneqq U \cap H$.
The natural morphism~$\uppi_1(U_H,u_0) \to \uppi_1(U,u_0)$ is surjective by~\cite[Th\'eor\`eme~0.1.1]{HL73}.
Hence
\begin{equation*}
\Cl(X) \cong \Pic(S)^{\uppi_1(U,u_0)} = \Pic(S)^{\uppi_1(U_H,u_0)} \cong \Cl(X_{n-1})
\end{equation*}
and the composition of the isomorphisms is given by~$i_{n-1}^*$. 

The equality of the orthogonals~$\Xi(X_i)$ follows immediately.
The rank of~$\Xi(X)$ equals
\begin{equation*}
\rr(S) - \rr(X) = (10 - \dd(S)) - \rr(X) = 10 - \dd(X) - \rr(X),
\end{equation*}
where the first equality follows from the standard relation between~$\rr(S)$ and~$\dd(S)$ for a del Pezzo surface,
and the second follows from the equality~$\dd(S) = \dd(X)$.
Since the class~$A_S = -K_S \in \Cl(S)$ is equal to the image of~$A_X \in \Cl(X)$, it follows that~$\Xi(X) \subset K_S^\perp$,
hence~$\Xi(X)$ is negative definite by the Hodge index theorem.
Therefore~$\Cl(X) = \Xi(X)^\perp$ is non-degenerate and has signature~$(1,\rr(X) - 1)$.
\end{proof}

\begin{corollary}
\label{cor:clx-free}
The class group~$\Cl(X)$ is a free abelian group;
in particular, the fundamental class~$A_X$ is canonically defined by~\eqref{eq:kx-a}.
\end{corollary} 

We will need the following simple observation.

\begin{lemma}
\label{lem:clx-xix-blowup}
If~$\sigma \colon \tX \to X$ is the blowup of a smooth point such that~$\tX$ is almost del Pezzo 
then the general surface linear section~$\tS \subset \tX$ is the blowup of a general surface linear section~$S \subset X$
and the blowup morphism~$\sigma_S \colon \tS \to S$ 
induces an equality of sublattices
\begin{equation*}
\Xi(\tX) = \sigma_S^*(\Xi(X)) \subset \Cl(\tS).
\end{equation*}
\end{lemma}

\begin{proof}
Let~$\ti \colon \tS \hookrightarrow \tX$ be the embedding.
Since~$A_{\tX} = \sigma^*A_X - E$ (see Proposition~\ref{prop:dP4:constr-i}), 
the surface~$\tS$ is the blowup of a linear surface section~$i \colon S \hookrightarrow X$,
and there is a commutative diagram
\begin{equation*}
\xymatrix{
\Cl(\tX) \ar[r]^{\ti^*} &
\Cl(\tS)
\\
\Cl(X) \ar[r]^{i^*} \ar[u]^{\sigma^*} &
\Cl(S), \ar[u]_{\sigma_S^*}
}
\end{equation*}
where~$\sigma_S \colon \tS \to S$ is the blowup.
Moreover, we have
\begin{equation*}
\Cl(\tX) = \sigma^*\Cl(X) \oplus \ZZ\, E,
\qquad 
\Cl(\tS) = \sigma_S^*\Cl(S) \oplus \ZZ\, E_S,
\end{equation*}
where~$E \subset \tX$ and~$E_S \subset \tS$ are the exceptional divisors of the blowups and~$\ti^*E = E_S$.
Therefore $\ti^*\Cl(\tX) = \ti^*\, \sigma^*\Cl(X) \oplus \ZZ\, \ti^*E = \sigma_S^*\, i^*\Cl(X) \oplus \ZZ\, E_S$ and
\begin{equation*}
\Xi(\tX) = 
\ti^*\Cl(\tX)^\perp =
(\sigma_S^*\, i^*\Cl(X))^\perp \cap E_S^\perp =
\sigma_S^*(i^*\Cl(X)^\perp) = 
\sigma_S^*\, \Xi(X),
\end{equation*}
hence the claim.
\end{proof} 

We will need the following useful observation.

\begin{lemma}
\label{lem:ad=1}
If~$X$ is an almost del Pezzo variety and~$D \subset X$ is an effective nontrivial divisor 
then we have~$\langle A_X, D \rangle \ge 1$.
Moreover, if~$\langle A_X, D \rangle = 1$ then~$\langle D, D \rangle \ge -1$.
\end{lemma}

\begin{proof}
For del Pezzo varieties the first follows from ampleness of~$A_X$, 
and in the almost del Pezzo case from the fact that the morphism~$X \to X_\can$ does not contract divisors.

To prove the second, let~$S \subset X$ be a general linear surface section.
Then the divisor~\mbox{$D_S \coloneqq D \cap S$} is a curve on~$S$ of anticanonical degree~$1$.
If~$\dd(S) = \dd(X) \ge 2$, it is a line, hence~$D_S^2 = -1$,
and if~$\dd(S) = \dd(X) = 1$, it is a line or an anticanonical divisor, hence~$D_S^2 = -1$ or~$D_S^2 = 1$.
In all cases the result follows.
\end{proof} 

\subsection{Special classes} 

Let~$\bLambda$ be a lattice of signature~$(1,r)$ endowed with a positive characteristic element~$\ba \in \bLambda$, 
i.e., such that~$\langle \ba, \ba \rangle > 0$ and~$\langle \ba, \x \rangle \equiv \langle \x, \x \rangle \bmod 2$
for all~$\x\in \bLambda$.
Then
\begin{equation*}
\ba^\perp \coloneqq \{ \x \in \bLambda \mid \langle \ba, \x \rangle = 0 \}
\end{equation*}
is a negative definite even lattice of rank~$r$.

For each integer~$s \in \ZZ$ we define the set 
\begin{equation}
\label{def:theta}
\Theta_s(\bLambda,\ba) \coloneqq \{ \x \in \bLambda \mid 
\langle \ba, \x \rangle = s,\ 
\langle \x, \x \rangle = s - 2 \}.
\end{equation}
Note that elements of~$\Theta_0(\bLambda, \ba)$ are nothing but the roots of~$\ba^\perp$, so we write
\begin{equation}
\label{def:delta}
\Delta(\bLambda, \ba) \coloneqq \Theta_0(\bLambda, \ba).
\end{equation}
We also introduce the following terminology.

\begin{definition}
\label{def:exP1P2}
We say that an element~$\x \in \bLambda$ is
\begin{itemize}
\item 
\emph{exceptional}, if~$\x \in \Theta_1(\bLambda, \ba)$;
\item 
\emph{a $\PP^1$-class}, if~$\x \in \Theta_2(\bLambda, \ba)$;
\item 
\emph{a $\PP^2$-class}, if~$\x \in \Theta_3(\bLambda, \ba)$
and~$\x$ does not lie in a linear span of~$\ba$ and an exceptional class~$\repsilon \in \Theta_1(\bLambda,\ba)$.
We write~$\Theta_3^\circ(\bLambda, \ba)$ for the set of~$\PP^2$-classes.
\end{itemize}
\end{definition}

\begin{remark}
It is easy to check that the second part of the definition of $\PP^2$-class 
is void unless~$\langle \ba, \ba \rangle = 1$ and~$\x = \ba + 2\e$ for an appropriate exceptional class~$\e$.
The geometric counterpart of this will be explained in Lemma~\ref{lem:2dP}.
\end{remark} 

\begin{remark}
\label{rem:d=1}
In the case where~$\langle \ba, \ba \rangle = 1$ there exists a natural bijection
\begin{equation*}
\Delta(\bLambda, \ba) \overset{\cong}\longrightarrow \Theta_1(\bLambda, \ba),
\qquad 
\balpha \longmapsto \balpha + \ba.
\end{equation*}
Similarly, the maps~$\x \mapsto \x - s\,\ba$ are isomorphisms of~$\Theta_s(\bLambda, \ba)$, $s \in \{2,3\}$,
onto the sets of elements in~$\ba^\perp$ with square~$-4$ and~$-8$, respectively,
and~$\PP^2$-classes correspond to elements not proportional to roots;
in particular, all these sets are finite.
\end{remark}

We will apply all these notions to the lattice~$\bLambda = \Cl(X)$ with~$\ba = A_X$
(note that the class~$A_X \in \Cl(X)$ is characteristic by Riemann--Roch theorem).
In~\S\ref{sec:cones} we will explain the geometric meaning of~$\PP^1$- and~$\PP^2$-classes 
and in this section we restrict our attention to exceptional classes.
We start with the well known case of del Pezzo surfaces.

\begin{lemma}
\label{lem:surface-exceptional-class}
Let~$(X,A_X)$ be a del Pezzo surface and let~$E \in \Cl(X)$ be an exceptional class.
Then~$\dim H^0(X,\cO_X(E)) = 1$ and
\begin{equation}
\label{eq:h0-h1-vanishing}
H^0(X, \cO_X(E - kA_X)) = H^1(X, \cO_X(E - kA_X)) = 0
\qquad 
\text{for $k \ge 1$}.
\end{equation}
\end{lemma}

\begin{proof}
We have~$H^2(X,\cO_X(E)) = H^0(X,\cO_X(-A_X-E))^\vee = 0$ by Lemma~\ref{lem:ad=1} and Serre duality, 
therefore Riemann--Roch theorem gives
\begin{equation*}
\dim H^0(X,\cO_X(E)) - \dim H^1(X,\cO_X(E)) = (E^2 + A_X\cdot E)/2 + 1 = 1,
\end{equation*}
hence~$E$ is effective.
Since~$A_X \cdot E = 1$, it is irreducible, and since~$E^2 = -1$, it is rigid; thus~$\dim H^0(X,\cO_X(E)) = 1$. 

Since~$\langle A_X, E - kA_X \rangle < \langle A_X, E \rangle = 1$ the vanishing of~$H^0(X, \cO_X(E - kA_X))$ 
follows from Lemma~\ref{lem:ad=1}.
To prove the vanishing of~$H^1(X, \cO_X(E - kA_X))$ by Serre duality 
it is enough to show that~$H^1(X,\cO_X(kA_X - E)) = 0$ for~$k \ge 0$.
For this we consider the standard exact sequence
\begin{equation*}
0 \longrightarrow \cO_X(kA_X - E) \longrightarrow \cO_X(kA_X) \longrightarrow \cO_E(k) \longrightarrow 0.
\end{equation*}
The vanishing of~$H^1(X,\cO_X(kA_X))$ (proved in Lemma~\ref{lem:vanishings}) shows that we need 
to check the surjectivity of the restriction map~$H^0(X,\cO_X(kA_X)) \to H^0(E,\cO_E(k))$.
Since~$E$ is a line, the surjectivity easily follows for~$k \in \{0,1\}$, 
and for higher~$k$ it follows from the fact that the algebra~$\oplus H^0(E,\cO_E(k))$ is generated by the first component.
\end{proof} 

The following proposition extends the above property to higher dimensions.

\begin{proposition}
\label{prop:E-eff}
Let~$X$ be an almost del Pezzo variety of dimension~$n \ge 2$.
If~\mbox{$E \in \Cl(X)$} is an exceptional class, then $\dim(H^0(X,\cO_X(E))) = 1$.
\end{proposition} 

\begin{proof}
First, we show by induction on~$n$ that~\eqref{eq:h0-h1-vanishing} holds.
The base of the induction, $n = 2$, where~$X$ is a del Pezzo surface, is proved in Lemma~\ref{lem:surface-exceptional-class}.
Now assume~$n \ge 3$.
To prove the vanishing of~$H^i(X, \cO_X(E-kA_X))$ we choose a general fundamental divisor~$Y \in |A_X|$ 
and consider the long exact sequence of cohomology
associated with the exact sequence of sheaves
\begin{equation}
\label{eq:xy-sequence}
0 \longrightarrow \cO_X(E - (k+1)A_X) \longrightarrow \cO_X(E - kA_X) 
\longrightarrow \cO_Y(E - kA_X) \longrightarrow 0.
\end{equation}
By induction we have~$H^0(Y, \cO_Y(E - kA_X)) = H^1(Y, \cO_Y(E - kA_X)) = 0$, hence
\begin{equation*}
H^i(X, \cO_X(E - (k+1)A_X)) = H^i(X, \cO_X(E - kA_X))
\end{equation*}
for~$i \in \{0,1\}$ and all~$k \ge 1$.
On the other hand, for~$k \gg 0$ we have~$H^i(X, \cO_X(E - kA_X)) = 0$ by Serre vanishing.
Therefore, the same holds for any~$k \ge 1$.

Now we compute~$\dim H^0(X,\cO_X(E))$, again by induction on~$n$.
The base of induction is again given by Lemma~\ref{lem:surface-exceptional-class}, 
and for the step we use the exact sequence
\begin{equation*}
H^0(X, \cO_X(E - A_X)) \longrightarrow H^0(X, \cO_X(E)) \longrightarrow 
H^0(Y, \cO_Y(E)) \longrightarrow H^1(X, \cO_X(E - A_X)) 
\end{equation*}
obtained from~\eqref{eq:xy-sequence} with~$k = 0$.
Its first and last terms are zero by the first part of the proposition,
and so we conclude that~$H^0(X, \cO_X(E)) = H^0(Y, \cO_Y(E))$.
\end{proof}

\begin{proposition}
\label{prop:exceptional-contraction}
Let~$X$ be a del Pezzo variety and let~$E_1,\dots,E_k \in \Cl(X)$ be a
collection of pairwise orthogonal exceptional divisor classes, i.e.,
\begin{equation*}
\langle E_i,\, E_j\rangle=0\quad \text{for}\quad 1 \le i < j \le k.
\end{equation*}
Then there exists an almost del Pezzo variety~$\hX_0$ 
and a collection~$P_1,\dots,P_k \in \hX_0$ of smooth points 
such that~$\Bl_{P_1,\dots,P_k}(\hX_0)$ is almost del Pezzo and~$X \cong (\Bl_{P_1,\dots,P_k}(\hX_0))_\can$.

In particular, if~$\Cl(X)$ contains an exceptional class, $X$ is imprimitive.
\end{proposition}

\begin{proof}
By Proposition~\ref{prop:E-eff} we can assume that each class~$E_i$ is effective and irreducible.
Let $D \coloneqq \sum E_i$ and let~$\xi \colon \hX \to X$ be a $\QQ$-factorialization provided by Lemma~\ref{lem:d-mmp},
set~$\hD \coloneqq (\xi^{-1})_*D$ and~$\hE_i \coloneqq (\xi^{-1})_*E_i$.
Then the pair~$(\hX,\hD)$ has the following properties:
\begin{enumerate}
\item 
\label{en:d-negative-rays}
every~$\hD$-negative extremal ray is also~$K$-negative, and
\item 
\label{en:d-sum}
$\hD = \sum \hE_i$, where~$\hE_i$ are pairwise orthogonal effective exceptional divisor classes.
\end{enumerate}
Indeed, the first is ensured by Lemma~\ref{lem:d-mmp}, 
and the second follows from the assumption about the~$E_i$ and Lemma~\ref{lem:pseudo-clx}.
We show that these properties imply~$\hX \cong \Bl_{P_1,\dots,P_k}(\hX_0)$, 
where~$\hX_0$ is almost del Pezzo, $P_1,\dots,P_k \in\hX_0$ are smooth points, 
and~$\hE_i$ are the exceptional divisors of the blowup.

To prove this we use induction on~$k$.
If~$k = 0$ there is nothing to prove, so assume~$k > 0$.
The class~$\hD$ cannot be nef because by~\ref{en:d-sum} we have
\begin{equation*}
\langle \hD, \hE_i \rangle = \langle \hE_i, \hE_i \rangle = -1,
\end{equation*}
therefore there is a $\hD$-negative extremal ray~$\rR$. 
This ray cannot be $K$-trivial by~\ref{en:d-negative-rays}, 
hence the corresponding contraction~$f \colon \hX \to \hX'$ is an extremal $K$-negative contraction.
Since~\mbox{$\hD \cdot \rR < 0$} and~$\hD$ is effective, the map~$f$ is birational, 
so Lemma~\ref{lemma:ext-rays0} implies that~$\hX'$ is a $\QQ$-factorial almost del Pezzo variety, 
$f$ is the blowup of a smooth point~$P'_1 \in \hX'$,
and the exceptional divisor of~$f$ is a component of~$\hD$, say~$\hE_1$.
Note that~$\hE_1$ is Cartier and~$A_{\hX}$ is ample on~$\hE_1$ by~\eqref{eq:coe-e}, 
hence~\ref{en:d-sum} implies that~$\hE_1 \cap \hE_i = \varnothing$ for~$i \ge 2$.
Therefore, $\hE'_i = f_*(\hE_i)$, $2 \le i \le k$, 
are pairwise orthogonal exceptional divisor classes 
(see Lemma~\ref{lem:contractions-class}\ref{it:class-blowup} below for this simple computation).

On the other hand, if~$\hD' \coloneqq f_*\hD = \sum_{i=2}^k \hE'_i$
and~$C \subset \hX'$ is a $K$-trivial $\hD'$-negative curve then~$P'_1 \not\in C$,
hence~$f^{-1}(C)$ is a $K$-trivial $\hD$-negative curve, which contradicts~\ref{en:d-negative-rays}.
This shows that the pair~$(\hX',\hD')$ satisfies the properties~\ref{en:d-negative-rays} and~\ref{en:d-sum}
and by induction we have~$\hX' \cong \Bl_{P_2,\dots,P_k}(\hX_0)$.
It remains to note that~$P'_1$ does not lie on an exceptional divisor of this blowup
by Proposition~\ref{prop:dP4:constr-i}\ref{prop:dP4:constr-ia}, 
hence~$\hX \cong \Bl_{P'_1}(\Bl_{P_2,\dots,P_k}(\hX_0)) \cong \Bl_{P_1,P_2,\dots,P_k}(\hX_0)$,
where~$P_1$ is the image of~$P'_1$ in~$\hX_0$.
\end{proof} 

\subsection{Extremal contractions}

Recall from Proposition~\ref{propo:ext-rays} the classification of extremal contractions of almost del Pezzo varieties.

\begin{lemma}
\label{lem:contractions-class}
Let~$(X,A_X)$ be a $\QQ$-factorial almost del Pezzo variety 
and let~$f \colon X \to Z$ be a $K$-negative extremal contraction.
\begin{enumerate}
\item 
\label{it:class-blowup}
If~$f$ is birational, i.e., the blowup of a smooth point on~$Z$
with exceptional divisor~$E$ 
then~\mbox{$\Cl(X) = f^*\Cl(Z) \oplus \ZZ E$} and for any~$D,D_1,D_2 \in \Cl(Z)$ one has
\begin{equation*}
\langle f^*D_1, f^*D_2 \rangle = \langle D_1, D_2 \rangle,
\qquad 
\langle f^*D, E \rangle = 0,
\qquad 
\langle E, E \rangle = -1.
\end{equation*}

\item 
\label{it:class-pn-2}
If~$f$ is a $\PP^{n-2}$-bundle over a smooth surface~$Z$
then~$\Cl(X) = f^*\Cl(Z) \oplus \ZZ A_X$ and for any~$D,D_1,D_2 \in \Cl(Z)$ one has
\begin{equation*}
\langle f^*D_1, f^*D_2 \rangle = D_1 \cdot D_2,
\qquad 
\langle f^*D, A_X \rangle = -K_Z \cdot D.
\end{equation*}

\item 
\label{it:class-quadric}
If~$f$ is a flat quadric bundle over~$Z = \PP^1$ and~$F \in \Cl(Z)$ is the class of a fiber 
then~$\Cl(Z) = \ZZ F \oplus \ZZ A_X$ and
\begin{equation*}
\langle F, F \rangle = 0,
\qquad 
\langle A_X, F \rangle = 2.
\end{equation*}
\end{enumerate}
\end{lemma}

\begin{proof}
Part~\ref{it:class-blowup} follows easily from the blowup formula, projection formula (note that the blowup point on~$Z$ is smooth),
the equality~$A_X = f^*A_Z - E$ (see Proposition~\ref{prop:dP4:constr-i})
and isomorphisms~\eqref{eq:coe-e}. 

The first formula in part~\ref{it:class-pn-2} is clear.
Further, $X \cong \PP_Z(\cE)$ and~$\rc_1(\cE) = K_Z$ by Proposition~\ref{propo:ext-rays},
so if~$D$ is a smooth curve 
and~$X_D \coloneqq f^{-1}(D) \cong \PP_D(\cE\vert_D)$ then~$A_X\vert_{X_D}$ 
is the relative hyperplane class of~$X_D$, and
the intersection theory implies 
\begin{equation*}
\langle f^*D, A_X \rangle = 
A_X^{n-1} \cdot f^*D =
(A_X\vert_{X_D})^{n-1} =
\deg(\rc_1(\cE^\vee)\vert_D) = -K_Z \cdot D.
\end{equation*}
Since classes of smooth curves generate~$\Cl(Z)$, the same formula holds for any~$D$.

Part~\ref{it:class-quadric} is again obvious because~$F^2 = 0$ 
and the restriction of~$A_X$ to any fiber~$Q^{n-1}$ of~$f$ is the hyperplane class.
\end{proof}

\begin{corollary}
\label{cor:imprimitivity}
An almost del Pezzo variety~$X$ is imprimitive if and only if~$\Cl(X)$ contains an exceptional class.
\end{corollary}

\begin{proof}
If~$X$ is imprimitive we apply Lemma~\ref{lem:contractions-class}\ref{it:class-blowup},
and if~$X$ contains an exceptional class we apply Proposition~\ref{prop:exceptional-contraction}.
\end{proof}

\begin{corollary}
\label{cor:dP4:constr2}
If~$X = \PP_Z(\cE)$ is a primitive almost del Pezzo variety 
then~\mbox{$Z \cong \PP^2$} or~\mbox{$Z \cong \PP^1 \times \PP^1$}.
In particular, $\rr(X) = 2$ or~$\rr(X) = 3$.
\end{corollary}

\begin{proof}
By Proposition~\ref{propo:ext-rays}\ref{prop:ext-rays2} $Z$ is a del Pezzo surface, 
so it remains to show that~$Z$ is minimal.
Indeed, if it is not, $Z$ contains a $(-1)$-curve~$D \subset Z$, so that~$K_Z \cdot D = D^2 = -1$;
then by Lemma~\ref{lem:contractions-class}\ref{it:class-pn-2} the class~$E \coloneqq f^*D \in \Cl(X)$ is exceptional,
hence~$X$ is imprimitive by Corollary~\ref{cor:imprimitivity}.
\end{proof} 

\begin{corollary}
\label{cor:primitive-rank}
If~$X$ is a primitive almost del Pezzo variety then $\rr(X) \le 3$.
\end{corollary}

\begin{proof}
If~$\rr(X) = 1$ there is nothing to prove, so assume~$\rr(X) \ge 2$.
Then an appropriate crepant model~$\hX$ of~$X$
has a $K$-negative extremal contraction~$f \colon \hX \to Z$.
By Proposition~\ref{propo:ext-rays} either~$\hX \cong \PP_Z(\cE)$, and then~$\rr(X) \le 3$ by Corollary~\ref{cor:dP4:constr2},
or~$f$ is a quadric bundle over~$\PP^1$, and then~$\rr(X) = 2$ since~$f$ is extremal.
\end{proof}  

The next proposition describes the lattice structure of~$\Cl(X)$, where~$X$ is a projective bundle over~$\PP^1 \times \PP^1$.
We will see in Proposition~\ref{prop:bundle-p1p1} that any such~$X$ has~$\dd(X) \le 6$.

\begin{lemma}
\label{lem:lattice-p-p1p1}
Let~$f \colon X \to \PP^1 \times \PP^1$ be a $K$-negative extremal contraction 
of a $\QQ$-factorial almost del Pezzo variety~$X$ of degree~$d = \dd(X) \le 6$.
Let~$F_1,F_2 \in \Cl(X)$ be the pullbacks to~$X$ of the rulings of~$\PP^1 \times \PP^1$.
The bilinear form~\eqref{eq:product} in the basis~$F_1,\, F_2,\, A$ of~$\Cl(X)$ has the matrix
\begin{equation*}
\begin{pmatrix}
0 & 1 & 2 \\
1 & 0 & 2 \\
2 & 2 & d
\end{pmatrix}.
\end{equation*}
The variety~$X$ is primitive for~$d \in \{2,4,6\}$ and imprimitive for~$d \in \{1,3,5\}$.
\end{lemma}

\begin{proof}
To compute the matrix of the lattice we apply Lemma~\ref{lem:contractions-class}\ref{it:class-pn-2}.
For the second part note that by Corollary~\ref{cor:imprimitivity}
imprimitivity is equivalent to the existence of an exceptional class.
The definition of exceptionality for an element~$x_1F_1 + x_2F_2 + yA \in \Cl(X)$ reads as
\begin{equation*}
2x_1 + 2x_2 + dy = 1
\qquad 
2x_1x_2 + 4x_1y + 4x_2y + dy^2 = -1.
\end{equation*}
Squaring the first and subtracting the second multiplied by~$d$, we obtain
\begin{equation*}
4(x_1 + x_2)^2 - 2dx_1x_2 = 1 + d.
\end{equation*}
Clearly, $d$ must be odd, and for~$1 \le d \le 6$ the integer solutions are
\begin{itemize}
\item 
$d = 1$: $E_\pm = A \pm (F_1 - F_2)$;
\item 
$d = 3$: $E_1 = A - F_1$, $E_2 = A - F_2$;
\item 
$d = 5$: $E = A - F_1 - F_2$.
\end{itemize}
In particular, for~$d \in \{2,4,6\}$ the variety~$X$ is primitive.
\end{proof} 

A similar computation describes~$\Cl(X)$ when~$X$ is a projective bundle over~$\PP^2$.
We will see in Proposition~\ref{prop:bundle-p2} that any such~$X$ has~$\dd(X) \le 7$. 

\begin{lemma}
\label{lem:lattice-p-p2}
Let~$f \colon X \to \PP^2$ be a $K$-negative extremal contraction 
of a $\QQ$-factorial almost del Pezzo variety of degree~$d = \dd(X)\le 7$.
Let~$G \in \Cl(X)$ be the pullback to~$X$ of the line class of~$\PP^2$.
The bilinear form~\eqref{eq:product} in the basis~$G,\, A$ of~$\Cl(X)$ has the matrix
\begin{equation*}
\begin{pmatrix}
1 & 3 \\
3 & d
\end{pmatrix}
.
\end{equation*}
The variety~$X$ is imprimitive for~$d \in \{4,7\}$ and is primitive otherwise.
\end{lemma}

\begin{proof}
The proof is similar to that of Lemma~\ref{lem:lattice-p-p1p1}.
Now the equations of exceptionality for an element~$xG + yA \in \Cl(X)$ are
\begin{equation*}
3x + dy = 1,
\qquad 
x^2 + 6xy + dy^2 = -1.
\end{equation*}
Squaring the first and subtracting the second multiplied by~$d$, we obtain
\begin{equation*}
(9 - d)x^2 = 1 + d.
\end{equation*}
The only integer solutions with~$1 \le d \le 7$ are
\begin{itemize}
\item 
$d = 4$: $E = A - G$ and
\item 
$d = 7$: $E = A -2G$.
\end{itemize}
Therefore, for~$d \in \{1,2,3,5,6\}$ the variety~$X$ is primitive.
\end{proof}

Next, we describe~$\Cl(X)$ when~$X$ is a quadratic fibration over~$\PP^1$.
We will see in Lemma~\ref{lem:quadric-bundles} that any such~$X$ has~$\dd(X) \le 6$.

\begin{lemma}
\label{lem:lattice-q-p2}
Let~$f \colon X \to \PP^1$ be a $K$-negative extremal contraction 
of a $\QQ$-factorial almost del Pezzo variety of degree~$d = \dd(X) \le 6$.
Let~$F \in \Cl(X)$ be the class of a fiber.
The bilinear form~\eqref{eq:product} in the basis~$F,\, A$ of~$\Cl(X)$ has the matrix
\begin{equation*}
\begin{pmatrix}
0 & 2 \\
2 & d
\end{pmatrix}
.
\end{equation*}
The variety~$X$ is imprimitive for~$d = 3$ and is primitive otherwise.

Moreover, $\Cl(X)$ contains a $\PP^1$-class~$F' \ne F$ 
if and only if~$d \in \{1,2,4\}$.
\end{lemma}

\begin{proof}
The proof is similar to that of Lemma~\ref{lem:lattice-p-p1p1}.
Now the equations of exceptionality for an element~$xF + yA \in \Cl(X)$ are
\begin{equation*}
2x + dy = 1,
\qquad 
4xy + dy^2 = -1.
\end{equation*}
Squaring the first and subtracting the second multiplied by~$d$, we obtain
\begin{equation*}
4x^2 = 1 + d.
\end{equation*}
The only integer solution with~$1 \le d \le 6$ is
\begin{itemize}
\item 
$d = 3$: $E = A - F$.
\end{itemize}
Therefore, for~$d \in \{1,2,4,5,6\}$ the variety~$X$ is primitive.

If~$F' = xF + yA$ is a $\PP^1$-class then 
\begin{equation*}
2x + dy = 2,
\qquad 
4xy + dy^2 = 0.
\end{equation*}
The only solution distinct from~$F$ is~$F' = \frac{4}{d}A - F$; it is integral if an only if~$d \in \{1,2,4\}$.
\end{proof}

Finally, we will need the following similar computation.

\begin{lemma}
\label{lem:another-blowup}
Let~$X$ be a del Pezzo variety with~$\dd(X) \ge 2$ and~$\rr(X) = 1$.
Let~\mbox{$\tX = \Bl_P(X)$} be its blowup at a point $P\in X$ with exceptional divisor~$E$
such that $\tX$ is an almost del Pezzo variety.
Then~$E$ is the only exceptional class in~$\Cl(\tX)$ if and only if a crepant model of~$\tX$
admits a non-birational extremal contraction if and only if~$\dd(X) \ge 4$.
\end{lemma}

\begin{proof}
Let~$E$ be the exceptional divisor of the blowup~$\sigma \colon \tX \to X$.
By Lemma~\ref{lem:contractions-class}\ref{it:class-blowup} 
the classes~$E$ and~$A_{\tX} = \sigma^*A_X - E$ form a basis of~$\Cl(\tX)$, 
and in this basis the matrix of the bilinear form is
\begin{equation*}
\begin{pmatrix}
-1 & 1\hphantom{-1} \\
\hphantom{-}1 & d - 1
\end{pmatrix},
\end{equation*}
where~$d = \dd(X)$.
Since~$\rr(\tX) = 2$, a $\QQ$-factorial crepant model of~$\tX$ 
must have a second $K$-negative extremal contraction, 
and if it is birational, $\Cl(\tX)$ must have an exceptional class distinct from~$E$.
If this class is~$xE + yA_{\tX}$, the conditions of exceptionality are
\begin{equation*}
x + (d - 1)y = 1,
\qquad 
-x^2 + 2xy + (d - 1)y^2 = -1.
\end{equation*}
Squaring the first and subtracting the second multiplied by~$d - 1$, we obtain
\begin{equation*}
dx^2 = d,
\end{equation*}
hence~$x = \pm 1$. 
If~$x = 1$ then~$y = 0$, and this solution corresponds to the original exceptional class~$E$.
If~$x = -1$ then~$y = 2/(d - 1)$ is integral if and only if~$d \in \{2,3\}$.
Thus, for~$d = \dd(X) \ge 4$, the other extremal contraction must be non-birational.
\end{proof} 

We refer to~\S\ref{ss:r2} for a description of the other extremal contractions. 

\section{Del Pezzo vector bundles and quadric bundles}
\label{sec:dpb} 

In this section we investigate del Pezzo bundles (see Definition~\ref{def:dpb}) and almost del Pezzo varieties
that have a structure of a $\PP^{n-2}$-bundle or flat quadric bundle.
We always assume~$n \ge 3$.

\subsection{Del Pezzo varieties and del Pezzo bundles}

First, we prove the basic properties of del Pezzo vector bundles.
We will say that a vector bundle~$\cF$ on a variety~$Z$ is \emph{almost globally generated}
if the cokernel of the natural morphism~$H^0(Z,\cF) \otimes \cO_Z \to \cF$ is the structure sheaf of a point. 

\begin{lemma}
\label{lem:dpv-dpb}
If~$Z$ is a surface and~$\cE$ is a del Pezzo bundle on~$Z$,
so that~$X = \PP_Z(\cE)$ is a non-conical almost del Pezzo variety,
then~$Z$ is a smooth del Pezzo surface and
\begin{equation}
\label{eq:dpb-inequality}
\dim(X) - 1 = \rk(\cE) \le \rc_2(\cE) = K_Z^2 - \dd(X);
\end{equation} 
in particular, $2 \le \rc_2(\cE) \le K_Z^2 - 1$.
Moreover, 
$H^0(Z, \cE) = H^2(Z, \cE) = 0$, 
and
\begin{itemize}
\item 
if~$\rc_2(\cE) \le K_Z^2 - 2$ then~$\cE^\vee$ is globally generated;
\item 
if~$\rc_2(\cE) = K_Z^2 - 1$ then~$\cE^\vee$ is almost globally generated.
\end{itemize}
\end{lemma}

\begin{proof}
We proved in Lemma~\ref{lemma:ext-rays0a}\ref{lemma:ext-rays0a2} that~$Z$ is a smooth del Pezzo surface.
We also proved that~$A_X$ is a relative hyperplane class for~$\PP_Z(\cE)$.
Therefore, (almost) global generation of~$\cO_X(A_X)$ implies the analogous property of~$\cE^\vee$,
which in its turn implies the inequality~$\rc_2(\cE) = \rc_2(\cE^\vee) \ge 0$.
On the other hand, by Definition~\ref{def:dpb} we have~$\rc_1(\cE) = K_Z$.

Furthermore, multiplying the standard relation
\begin{equation*}
A_X^{n-1} + A_X^{n-2}\cdot f^*\rc_1(\cE) + A_X^{n-3} \cdot f^*\rc_2(\cE) = 0
\end{equation*}
in the Chow ring of~$X$ by~$A_X$ 
we obtain~$\dd(X) = A_X^n = \rc_1(\cE)^2 - \rc_2(\cE) = K_Z^2 - \rc_2(\cE)$, 
i.e., the last equality in~\eqref{eq:dpb-inequality}.

To prove that~$H^2(Z,\cE) = 0$ note that Lemma~\ref{lem:contractions-class}\ref{it:class-pn-2} implies
\begin{equation*}
A_X^{n-1} \cdot (A_X + f^*K_Z) = \dd(X) - K_Z^2 = - \rc_2(\cE) \le 0,
\end{equation*}
therefore the divisor class~$A_X + f^*K_Z$ 
(which is nontrivial, because~$\rk(\cE) \ge 2$)
cannot be effective by Lemma~\ref{lem:ad=1}.
Combining this with Serre duality we obtain the required vanishing
\begin{equation*}
H^2(Z,\cE) = H^0(Z,\cE^\vee(K_Z))^\vee = H^0(X,\cO_X(A_X + f^*K_Z))^\vee = 0.
\end{equation*}

To prove that~$H^0(Z,\cE) = 0$
we let by contradiction~$s$ be a nontrivial section of~$\cE$. 
Since~$\cE^\vee$ is (almost) globally generated, the composition
\begin{equation*}
H^0(Z, \cE^\vee) \otimes \cO_Z \longrightarrow \cE^\vee \xrightarrow{\ s\ } \cO_Z
\end{equation*}
is generically surjective,
and since this is a morphism of trivial vector bundles, 
it is surjective everywhere, and a copy of~$\cO_Z$ splits from~$\cE$,
i.e., $\cE \cong \cO_Z \oplus \cE'$.
It follows that the section of~$\PP_Z(\cE)$ corresponding to the trivial summand~$\cO_Z$ 
is contracted by the morphism~$\Phi_{|A_X|}$ to a point~$\bv \in X_\can$,
and furthermore, that~$X_\can$ is covered by lines through~$\bv$, so it is a cone.
This contradicts Definition~\ref{def:dpb} and proves the vanishing of~$H^0(Z,\cE)$.

Using the vanishing of~$H^0(Z,\cE)$ and~$H^2(Z,\cE)$ and Riemann--Roch we obtain
\begin{equation}
\label{eq:dim-h1}
0 \le \dim H^1(Z,\cE) = 
- \rk(\cE) + \rc_2(\cE),
\end{equation}
which gives the inequality in~\eqref{eq:dpb-inequality}.
As the first equality in~\eqref{eq:dpb-inequality} is obvious, $\dim(X) \ge 3$, and~$\dd(X) \ge 1$, 
we obtain~$2 \le \rc_2(\cE) \le K_Z^2 - 1$.
This completes the proof.
\end{proof}

\begin{corollary}
\label{cor:bundle-rdn}
If~$\cE$ is a del Pezzo bundle on a del Pezzo surface~$Z$ and~$X = \PP_Z(\cE)$ then
\begin{equation*}
\dd(X) + \rr(X) + n \le 12.
\end{equation*}
\end{corollary}

\begin{proof}
This follows from~$K_Z^2 = 10 - \rr(Z) = 11 - \rr(X)$ combined with~\eqref{eq:dpb-inequality}.
\end{proof}

Recall that a del Pezzo bundle~$\cE$ is \emph{maximal} 
if there is no embedding~$\cE \hookrightarrow \cE'$
into another del Pezzo bundle~$\cE'$ with~$\cE'/\cE \cong \cO_Z$. 

\begin{lemma}
\label{lem:dpb-maximal}
For any del Pezzo bundle~$\cE$ on a surface~$Z$ there is a canonical exact sequence
\begin{equation}
\label{eq:ce-tce}
0 \longrightarrow \cE \longrightarrow \tcE \longrightarrow H^1(Z,\cE) \otimes \cO_Z \longrightarrow 0,
\end{equation}
where~$\tcE$ is a maximal del Pezzo bundle.
Moreover, $\PP_Z(\cE)$ is a linear section of~$\PP_Z(\tcE)$.
A del Pezzo bundle~$\cE$ is maximal if and only if~$\rk(\cE) = \rc_2(\cE)$ if and only if~$H^\bullet(Z,\cE) = 0$.
\end{lemma}

\begin{proof}
The extension~\eqref{eq:ce-tce} is given by the canonical element in
\begin{equation*}
\Ext^1 \big( H^1(Z,\cE) \otimes \cO_Z, \cE \big) \cong 
H^1(Z,\cE)^\vee \otimes \Ext^1(\cO_Z, \cE) \cong 
H^1(Z,\cE)^\vee \otimes H^1(Z,\cE).
\end{equation*}
Let~$\tX \coloneqq \PP_Z(\tcE)$.
Then the second arrow in the exact sequence~\eqref{eq:ce-tce} gives an element in
\begin{equation*}
\Hom(\tcE, H^1(Z,\cE) \otimes \cO_Z) \cong 
H^1(Z,\cE) \otimes H^0(Z,\tcE^\vee) \cong
\Hom \big(H^0(\tX, \cO_{\tX}(A_{\tX}))^\vee, H^1(Z,\cE) \big)
\end{equation*}
hence a linear subspace in~$\PP(H^0(\tX, \cO_{\tX}(A_{\tX}))^\vee)$.
It is clear that the corresponding linear section of~$\tX$ is~$X = \PP_Z(\cE)$.
If~$D \subset \PP_Z(\tcE)$ is a nontrivial effective divisor, 
$D \cap \PP(\tcE_z)$ contains a hyperplane for at least one point~$z \in Z$, 
while~$\PP(\cE_z) \subset \PP(\tcE_z)$ contains a line,
hence~$D \cap \PP_Z(\cE) \ne \varnothing$.
Therefore, applying several times Lemma~\ref{lem:criterion-divisor}\ref{it:divisor-dp} 
we conclude that~$\tX$ is an almost del Pezzo variety, 
hence~$\tcE$ is a del Pezzo bundle.

Looking at the cohomology exact sequence of~\eqref{eq:ce-tce} and using the cohomology vanishings of Lemma~\ref{lem:dpv-dpb},
we conclude that~$H^\bullet(Z, \tcE) = 0$.
Now the equality of~\eqref{eq:dim-h1} applied to~$\tcE$ implies that~$\tcE$ is maximal.
Conversely, for any maximal~$\tcE$ the equality of~\eqref{eq:dim-h1} 
combined with the cohomology vanishings of Lemma~\ref{lem:dpv-dpb} proves that~$H^\bullet(Z,\tcE) = 0$.
\end{proof} 

\subsection{Del Pezzo bundles and blowups}
\label{ss:dpb-blowups}

In this section we investigate the behavior of del Pezzo bundles under birational transformations of underlying surfaces.

\begin{proposition}
\label{prop:dpb}
Assume~$Z'$ is a del Pezzo surface and~$\sigma \colon Z' \to Z$ is the blowup of a point~$z \in Z$.
If~$\cE'$ is a del Pezzo bundle on~$Z'$ then~$\cE \coloneqq \sigma_*\cE'$ is a del Pezzo bundle on~$Z$ 
and there is an exact sequence
\begin{equation}
\label{eq:dpb-sequence}
0 \longrightarrow \sigma^*\cE \longrightarrow \cE' \longrightarrow \cO_L(-1) \longrightarrow 0,
\end{equation}
where~$L \subset Z'$ is the exceptional line of~$\sigma$. 
Moreover, there is a point~$P \in \PP_Z(\cE)$ and a pseudoisomorphism
\begin{equation*}
\chi \colon \Bl_P(\PP_Z(\cE)) \dashrightarrow \PP_{Z'}(\cE')
\end{equation*}
which is an isomorphism away from the subvarieties
\begin{equation*}
\Bl_P(\PP(\cE_z)) \subset \Bl_P(\PP_Z(\cE))
\qquad\text{and}\qquad 
L \times \PP^{n-3} \subset \PP_{Z'}(\cE'),
\end{equation*}
where the latter corresponds to a trivial subbundle in~$\cE'\vert_L$.
\end{proposition} 

\begin{proof}
Since~$\cE'$ is a del Pezzo bundle, its dual is (almost) globally generated by Lemma~\ref{lem:dpv-dpb},
hence the bundle~${\cE'}^\vee\vert_L$ is globally generated
(because an almost globally generated vector bundle on~$\PP^1$ is globally generated).
On the other hand, $\deg({\cE'}\vert_L) = K_{Z'} \cdot L = -1$, therefore 
\begin{equation}
\label{eq:cep-l}
\cE'\vert_L \cong \cO_L^{\oplus (n-2)} \oplus \cO_L(-1).
\end{equation}
Consider the epimorphism~$\cE' \twoheadrightarrow \cE'\vert_L \twoheadrightarrow \cO_L(-1)$ and let~$\cE''$ be its kernel, 
so that we have an exact sequence
\begin{equation}
\label{eq:cepp}
0 \longrightarrow \cE'' \longrightarrow \cE' \longrightarrow \cO_L(-1) \longrightarrow 0.
\end{equation}
Restricting it to~$L$ and taking into account that~$\Tor_1(\cO_L,\cO_L) \cong \cO_L(-L) \cong \cO_L(1)$,
we obtain an exact sequence
\begin{equation}
\label{eq:ceppl}
0 \longrightarrow \cO_L \longrightarrow \cE''\vert_L \longrightarrow \cE'\vert_L \longrightarrow \cO_L(-1) \longrightarrow 0,
\end{equation}
which implies that~$\cE''\vert_L \cong \cO_L^{\oplus (n-1)}$.
Therefore, $\cE'' \cong \sigma^*\cE$ for a vector bundle~$\cE$ on~$Z$,
so that~\eqref{eq:cepp} takes the form of~\eqref{eq:dpb-sequence},
and pushing forward~\eqref{eq:dpb-sequence} to~$Z$ we also obtain~$\cE \cong \sigma_*\cE'$.

Now consider the following diagram of birational maps
\begin{equation*}
\xymatrix@C=5em{
&
\Bl_{L \times \PP^{n-3}}(\PP_{Z'}(\cE')) \ar@{..>}[r]^{\hat\sigma} \ar[dl]_{\pi_1} \ar[dr]^{\pi_2} &
\Bl_P(\PP_Z(\cE)) \ar[dr]^{\rho}
\\
\PP_{Z'}(\cE') &&
\PP_{Z'}(\sigma^*\cE) \ar[r]^{\tilde\sigma} \ar@{-->}[ll] &
\PP_Z(\cE).
}
\end{equation*}
Here~$\pi_1$ is the blowup of the projectivization~$L \times \PP^{n-3}$ of the first summand in~\eqref{eq:cep-l}
and~$\pi_2$ is the blowup of the section~$\tilde{L} \subset \PP_L(\sigma^*\cE\vert_L)$ 
corresponding to the first arrow in~\eqref{eq:ceppl},
so that the dashed arrow at the bottom is the elementary transformation of projective bundles
induced by the morphism~$\sigma^*\cE \to \cE'$.
Furthermore, $\tilde\sigma$ is the blowup of the fiber~$\PP(\cE_z)\cong \PP^{n-2}$ of~$\PP_Z(\cE)$ over the point~$z$,
and~$\rho$ is the blowup of the point~$P \coloneqq \tilde\sigma(\tilde{L})$.
As the maps~$\pi_2$, $\tilde\sigma$, and~$\rho$ are birational, 
there is a birational map~$\hat\sigma$ that makes the right square commutative.

Note that the scheme preimage of the point~$P$ in~$\PP_{Z'}(\sigma^*\cE)$ 
is the curve~$\tilde{L} \subset \PP_L(\sigma^*\cE\vert_L)$,
and its preimage in~$\Bl_{L \times \PP^{n-3}}(\PP_{Z'}(\cE'))$ is the exceptional divisor of~$\pi_2$.
As this is a Cartier divisor, the universal property of blowup 
shows that the rational map~$\hat\sigma$ is in fact regular.
Note also that the strict transform of the exceptional divisor~$L \times \PP(\cE_z)$ of~$\tilde{\sigma}$
is equal to the exceptional divisor~$E_1$ of~$\pi_1$,
that its image in~$\Bl_P(\Bl_Z(\cE))$ is the smooth subvariety~$\Bl_P(\PP(\cE_z))$ of codimension~$2$
and the relative Picard number of~$\hat\sigma$ is~$1$.
Applying~\cite[Lemma~2.5]{K18} we conclude that~$\hat\sigma$ is the blowup of~$\Bl_P(\PP(\cE_z))\cong \Bl_P(\PP^{n-2})$.

Since both~$\pi_1$ and~$\hat\sigma$ are blowups of smooth subvarieties of codimension~$2$
with the same exceptional divisor~$E_1$, we conclude that the composition
\begin{equation*}
\chi \coloneqq \pi_1 \circ \hat\sigma^{-1} \colon \Bl_P(\PP_Z(\cE)) \dashrightarrow \PP_{Z'}(\cE')
\end{equation*}
is a pseudoisomorphism.
\end{proof} 

The converse to Proposition~\ref{prop:dpb} is also true.

\begin{lemma}
\label{lem:dpb-converse}
If~$X = \PP_Z(\cE) \xrightarrow{\ f\ } Z$ is an almost del Pezzo variety
and~$P \in \PP_Z(\cE)$ is a point such that~$X' = \Bl_P(X)$ is also almost del Pezzo 
then there is a pseudoisomorphism~\mbox{$X' \dashrightarrow \PP_{Z'}(\cE')$},
where~$\sigma \colon Z' \to Z$ is the blowup of the point~$z \coloneqq f(P) \in Z$
and~$\cE'$ is defined by the exact sequence~\eqref{eq:dpb-sequence}.
In particular, $\cE'$ is a del Pezzo bundle.
\end{lemma}

\begin{proof}
Consider the variety~$X'' \coloneqq \PP_{Z'}(\sigma^*\cE)$.
If~$L \subset Z'$ is the exceptional line of~$\sigma$ then
\begin{equation*}
(\sigma^*\cE)\vert_L \cong
\cE_z \otimes \cO_L.
\end{equation*}
In particular, the point~$P \in f^{-1}(z) = \PP(\cE_z)$ corresponds to 
an embedding~$\cO_L \hookrightarrow \sigma^*\cE\vert_L$.
Now we can define~$\cE'$ as 
\begin{equation*}
\cE' \coloneqq \Ker(\sigma^*\cE^\vee \longrightarrow \cO_L)^\vee,
\end{equation*}
where the map is the dual of the above embedding;
this is the unique vector bundle that fits into exact sequence~\eqref{eq:dpb-sequence}
that after restriction to~$L$ gives the sequence~\eqref{eq:ceppl} (where~$\cE'' = \sigma^*\cE$)
with the first arrow equal to the above embedding.
Now the argument of Proposition~\ref{prop:dpb} provides a pseudoisomorphism~$\PP_{Z'}(\cE') \dashrightarrow \Bl_P(X)$,
hence~$\cE'$ is a del Pezzo bundle.
\end{proof} 

\subsection{Del Pezzo bundles and duality}
\label{ss:dpb-duality}

In the next lemma we give a useful explicit description of del Pezzo bundles.
In what follows we say that points~$z_1,\dots,z_d$ on a del Pezzo surface~$Z$ are in \emph{general position}
if either~$d < K^2_Z$ and~$\Bl_{z_1,\dots,z_d}(Z)$ is a del Pezzo surface,
or~$d = K^2_Z$ and~$\Bl_{z_1,\dots,z_d}(Z)$ is a rational elliptic surface with nef anticanonical class and no~$(-2)$-curves.
Note that if the set~$z_1,\dots,z_d$ is in general position the same is true for any its subset.

\begin{lemma}
\label{lem:dpb-ideal}
Let~$\cE$ be a del Pezzo bundle on a del Pezzo surface~$Z$.
Set~$X = \PP_Z(\cE)$, $n = \dim(X)$, and~$d = \dd(X)$.
Then there is an exact sequence
\begin{equation}
\label{eq:ce-ideal}
0 \longrightarrow \cE \longrightarrow \cO_Z^{\oplus n} \longrightarrow \cI_{z_1,\dots,z_d}(-K_Z) \longrightarrow 0,
\end{equation} 
where~$z_1,\dots,z_d \in Z$ are points in general position. 

Conversely, if~$z_1,\dots,z_d \in Z$ are points in general position
and~$\cE$ is defined by~\eqref{eq:ce-ideal} 
where the second arrow is induced by an $n$-dimensional subspace in~$H^0(Z, \cI_{z_1,\dots,z_d}(-K_Z))$
then~$\cE$ is a del Pezzo bundle and~$\dd(\PP_Z(\cE)) = d$.
Finally, the bundle~$\cE$ is maximal if and only if the middle term is~$H^0(Z, \cI_{z_1,\dots,z_d}(-K_Z)) \otimes \cO_Z$ 
and the map is the evaluation.
\end{lemma}

\begin{proof}
Recall from Proposition~\ref{DP:|A|} that~$|A_X|$ is almost base point free,
and if~$\dd(X) = 1$, every divisor in~$|A_X|$ is smooth at the base point of~$|A_X|$.
Consider a general linear surface section~$S \subset X$ and a general intersection~$R \subset S$ 
of two anticanonical divisors on~$S$.
Then~$S$ is a del Pezzo surface of degree~$d$ and~$R = \{x_1,\dots,x_d\}$ is a reduced scheme of length~$d$ 
such that the elliptic surface~$\hS \coloneqq \Bl_{x_1,\dots,x_d}(S)$ 
has nef anticanonical class and no~$(-2)$-curves.

Since~$H^0(X, \cO_X(A_X)) = H^0(Z,\cE^\vee) = \Hom(\cE, \cO_Z)$,
the fundamental divisors cutting out~$R$ in~$X$ correspond to a morphism~$\varphi\colon \cE \to \cO_Z^{\oplus n}$ of vector bundles
such that its degeneration scheme~$f(R) = \{z_1,\dots,z_d\} \subset Z$ is also reduced, where~$f \colon X \to Z$ is the projection.
Therefore, $\Coker(\varphi)$ is a torsion-free sheaf of rank~$1$ and its first Chern class equals~$- \rc_1(\cE) = - K_Z$,
so that we have an exact sequence of the form~\eqref{eq:ce-ideal}.

It remains to show that~$z_1,\dots,z_d$ are in general position.
For this note that the morphism~$f\vert_S \colon S \to Z$ is birational, hence~$S = \Bl_{z_{d+1},\dots,z_{d+k}}(Z)$,
where~$k = K_Z^2 - d$ and~$z_{d+1},\dots,z_{d+k} \in Z$ are in general position because~$S$ is del Pezzo.
Recall that
\begin{equation*}
\hS = 
\Bl_{x_1,\dots,x_d}(S) = 
\Bl_{x_1,\dots,x_d}(\Bl_{z_{d+1},\dots,z_{d+k}}(Z)) \cong
\Bl_{z_1,\dots,z_d,z_{d+1},\dots,z_{d+k}}(Z)
\end{equation*}
is an elliptic surface with nef anticanonical class and no~$(-2)$-curves,
hence any subset in~$z_1,\dots,z_d,z_{d+1},\dots,z_{d+k}$ 
of cardinality less than~$d + k$ is in general position;
in particular so is the set~$z_1,\dots,z_d$.

Conversely, let~$z_1,\dots,z_d \in Z$ be points in general position,
choose an $n$-dimensional generating space of global sections of~$\cI_{z_1,\dots,z_d}(-K_Z)$,
and let~$\cE$ be defined by~\eqref{eq:ce-ideal}.
Consider a general anticanonical pencil in~$Z$ corresponding to a $2$-dimensional subspace of this~$n$-dimensional space
and let~$z_{d+1},\dots,z_{d+k}$ be its extra base points.
It is easy to see that these points form the degeneracy locus of the corresponding composition
\begin{equation*}
\cE \longrightarrow \cO_Z^{\oplus n} \longrightarrow \cO_Z^{\oplus(n-2)},
\end{equation*}
therefore the surface~$S = \Bl_{z_{d+1},\dots,z_{d+k}}(Z)$ is a linear section of~$\PP_Z(\cE)$.
It remains to note that~$\Bl_{z_1,\dots,z_d,z_{d+1},\dots,z_{d+k}}(Z)$ 
is a rational elliptic surface with nef anticanonical class and no~$(-2)$-curves,
hence~$S$ is a del Pezzo surface.
If~$D \subset \PP_Z(\cE)$ is a nontrivial effective divisor, 
either~$D \cap \PP(\cE_z)$ contains a hyperplane for each point~$z \in Z$, 
while the fiber of~$S$ over~$z_i$ contains a line for~$d + 1 \le i \le d + k$,
or~$D$ is the preimage of a divisor in~$Z$.
In either case~$D \cap \PP_Z(\cE) \ne \varnothing$,
and hence~$\PP_Z(\cE)$ is almost del Pezzo by Lemma~\ref{lem:criterion-divisor}\ref{it:divisor-dp}.

Finally, maximality of~$\cE$ is equivalent to the vanishing of~$H^\bullet(Z,\cE)$ by Lemma~\ref{lem:dpb-maximal}, 
hence the last claim.
\end{proof}

Using this construction we easily deduce the following duality property.

\begin{corollary}
\label{cor:dpb-duality}
Let~$\cE$ be a del Pezzo bundle on a del Pezzo surface~$Z$ such that
\begin{equation*}
2 \le\rc_2(\cE) \le K_Z^2 - 2,
\end{equation*}
and hence~$\cE^\vee$ is globally generated.
Then the vector bundle~$\cE^\perp$ defined by the exact sequence
\begin{equation}
\label{eq:ceperp}
0 \longrightarrow \cE^\perp \longrightarrow H^0(Z,\cE^\vee) \otimes \cO_Z \longrightarrow \cE^\vee \longrightarrow 0
\end{equation}
is a maximal del Pezzo bundle with~$\rc_2(\cE^\perp) = K_Z^2 - \rc_2(\cE)$.
If~$\cE$ is maximal then~$(\cE^\perp)^\perp \cong \cE$.
\end{corollary}

\begin{proof}
By definition~$\rc_1(\cE^\perp) = K_Z$.
Dualizing the sequence~\eqref{eq:ce-ideal} we obtain
\begin{equation*}
0 \longrightarrow \cO_Z(K_Z) \longrightarrow \cO_Z^{\oplus n} \longrightarrow \cE^\vee \longrightarrow \cO_{z_1,\dots,z_d} \longrightarrow 0.
\end{equation*}
Combining it with sequence~\eqref{eq:ceperp} we obtain an exact sequence
\begin{equation*}
0 \longrightarrow \cO_Z(K_Z) \longrightarrow \cE^\perp 
\longrightarrow \cO_Z^{\oplus (d-1)} \longrightarrow \cO_{z_1,\dots,z_d} \longrightarrow 0.
\end{equation*}
Geometrically, this means that the linear section of~$\PP_Z(\cE^\perp)$ 
corresponding to the morphism~$\cE^\perp \to \cO_Z^{\oplus (d-1)}$ 
is isomorphic to the blowup of point~$z_1,\dots,z_d$ in~$Z$.
Since these points are in general position, the blowup is a del Pezzo surface.
Applying the argument from the proof of Lemma~\ref{lem:dpb-ideal}
we conclude that~$\PP_Z(\cE^\perp)$ is an almost del Pezzo variety, hence~$\cE^\perp$ is a del Pezzo bundle.
It is maximal because~\eqref{eq:ceperp} implies~$H^1(Z,\cE^\perp) = 0$,
and the relation between the second Chern classes also follows immediately from~\eqref{eq:ceperp}.

If~$\cE$ is maximal, $H^\bullet(Z,\cE) = 0$ by Lemma~\ref{lem:dpb-maximal}, 
hence the dual of~\eqref{eq:ceperp} 
\begin{equation*}
0 \longrightarrow \cE \longrightarrow H^0(Z, \cE^\vee)^\vee \otimes \cO_Z \longrightarrow (\cE^\perp)^\vee \longrightarrow 0
\end{equation*}
implies that~$H^0(Z, \cE^\vee)^\vee \cong H^0(Z, (\cE^\perp)^\vee)$, hence~$(\cE^\perp)^\perp \cong \cE$.
\end{proof}

The case where~$\rc_2(\cE) = K^2_Z - 1$ missing in Corollary~\ref{cor:dpb-duality} is discussed below.
Note that as~$\rc_2(\cE) \ge 2$ by Lemma~\ref{lem:dpv-dpb}, we may and will assume~$K_Z^2 \ge 3$.

Recall from the Introduction that for a smooth del Pezzo surface~$Z$ we denote by 
\begin{equation}
\label{eq:bhz}
\bH(Z) \coloneqq \{ (D,z) \in \PP(H^0(Z,\cO(-K_Z))) \times Z \mid z \in D \}
\end{equation}
the universal anticanonical divisor on~$Z$.
Note that the projection~$\bH(Z) \to Z$ is a projectivization of a vector bundle 
(because~$\cO(-K_Z)$ is globally generated), 
hence~$\bH(Z)$ is smooth 
and, moreover, if~$L \subset Z$ is a line, the subvariety
\begin{equation}
\label{eq:pil}
\Pi_L \coloneqq \{ (D,z) \in \bH(Z) \mid z \in L \} \subset \bH(Z)
\end{equation}
is a divisor in~$\bH(Z)$. 

\begin{proposition}
\label{prop:bl-pze-hblz}
Let~$\cE$ be a maximal del Pezzo bundle with~$\rc_2(\cE) = K_Z^2 - 1$ on a del Pezzo surface~$Z$. 
Let~$x_0 \in X \coloneqq \PP_Z(\cE)$ be the base point of~$|A_X|$ and let~$z_0 \in Z$ be its image.
Set~$S \coloneqq \Bl_{z_0}(Z)$ and let~$L_0 \subset S$ be the exceptional line over~$z_0$.
Then there is a commutative diagram
\begin{equation*}
\xymatrix{
E_0 \ar@{^{(}->}[r] &
\Bl_{x_0}(\PP_Z(\cE)) \ar@{-->}[rr]^\chi \ar[dr]_{\Phi_{|A|}} && 
\bH(S) \ar[dl]^{\operatorname{pr}} &
\Pi_{L_0} \ar@{_{(}->}[l]
\\
&& \PP(H^0(S, \cO(-K_{S})))
}
\end{equation*}
where~$\Phi_{|A|}$ is induced by the anticanonical map of~$\PP_Z(\cE)$, 
$\operatorname{pr}$ is the natural projection,
and~$\chi$ is a pseudoisomorphism. 
Moreover, the exceptional divisor~$E_0 \subset \Bl_{x_0}(\PP_Z(\cE))$
is mapped by~$\chi$ to the divisor~$\Pi_{L_0} \subset \bH(S)$, 
and the indeterminacy loci of~$\chi$ and~$\chi^{-1}$ are contracted 
to the linear subspace~$\PP(H^0(S,\cO(-K_{S}-L_0))) \subset \PP(H^0(S,\cO(-K_{S})))$ 
of codimension~$2$.
\end{proposition}

\begin{proof}
Let~$\sigma \colon S \coloneqq \Bl_{z_0}(Z) \to Z$ be the blowup.
The argument of Lemma~\ref{lem:dpb-ideal} shows that~$z_0$ is in general position, 
hence~$S$ is a del Pezzo surface.
Applying Lemma~\ref{lem:dpb-converse} we find a vector bundle~$\cE'$ on~$S$ and a pseudoisomorphism
\begin{equation*}
\Bl_{x_0}(\PP_Z(\cE)) \dashrightarrow \PP_{S}(\cE')
\end{equation*}
(note that the left hand side is not a del Pezzo variety, 
but rather the blowup of a del Pezzo variety of degree~$1$ at its base point,
but the construction of the pseudoisomorphism in Proposition~\ref{prop:dpb} in this case works in the same way). 

On the other hand, combining the sequence~\eqref{eq:dpb-sequence} defining~$\cE'$ with~\eqref{eq:ce-ideal}
(or repeating the argument of Lemma~\ref{lem:dpb-ideal} for~$\PP_{S}(\cE')$) 
we obtain an exact sequence
\begin{equation*}
0 \longrightarrow \cE' \longrightarrow H^0(S, \cO_{S}(-K_{S})) \otimes \cO_{S} \longrightarrow \cO_{S}(-K_{S}) \longrightarrow 0.
\end{equation*}
It follows that~$\PP_{S}(\cE') \cong \bH(S)$.
Combining this isomorphism with the pseudoisomorphism constructed above, we obtain the pseudoisomorphism~$\chi$. 

Note that the morphism~$\Phi_{|A|}$ is given by the anticanonical class of~$\Bl_{x_0}(\PP_Z(\cE))$.
Since~$\chi$ is a pseudoisomorphism, it corresponds to the anticanonical class of~$\bH(S)$, 
which is equal to a multiple of the hyperplane class of~$\PP(H^0(S, \cO(-K_{S})))$.
Therefore, the diagram commutes.

Moreover, by Proposition~\ref{prop:dpb} the indeterminacy locus of~$\chi^{-1}$ 
is equal to the projectivization over the line~$L_0$ 
of the trivial subbundle in~$H^0(S,\cO(-K_{S}-L_0)) \otimes \cO_{L_0} \subset \cE'\vert_{L_0}$ 
of codimension~$2$,
hence its image under~$\operatorname{pr}$ is equal to~$\PP(H^0(S,\cO(-K_{S}-L_0)))$. 
Since~$\chi$ is a pseudoisomorphism, the image of the indeterminacy locus of~$\chi$ is the same.

Finally, $\chi$ is compatible with the projection to~$Z$,
and the divisors~$E_0$ and~$\Pi_{L_0}$ are the only divisors supported over~$z_0 \in Z$, 
hence~$\chi_*(E_0) = \Pi_{L_0}$.
\end{proof} 

\subsection{Classification of del Pezzo bundles}

In this section we classify maximal del Pezzo bundles on minimal del Pezzo surfaces 
and sketch a classification in other cases.

\begin{proposition}
\label{prop:bundle-p2}
Let~$\cE = \cE_{\PP^2,k}$ be a maximal del Pezzo bundle on~$\PP^2$ with~$\rc_2(\cE) = k$.
Then~$d = \dd(\PP_{\PP^2}(\cE_{\PP^2,k})) = 9 - k \le 7$ and
\begin{itemize}
\item 
if~$d = 7$ then~$\cE_{\PP^2,2} \cong \cO(-1) \oplus \cO(-2)$;
\item 
if~$d = 6$ then~$\cE_{\PP^2,3} \cong \cO(-1)^{\oplus 3}$;
\item 
if~$d = 5$ then~$\cE_{\PP^2,4} \cong \cO(-1)^{\oplus 2} \oplus \Omega(1)$;
\item 
if~$d = 4$ then~$\cE_{\PP^2,5} \cong \cO(-1) \oplus \Omega(1)^{\oplus 2}$;
\item 
if~$d = 3$ then~$\cE_{\PP^2,6} \cong \Omega(1)^{\oplus 3}$;
\item 
if~$d = 2$ there is a canonical exact sequence $0 \to \cE_{\PP^2,7} \to \cO^{\oplus 9} \to \cO(1) \oplus \cO(2) \to 0$;
\item 
if~$d = 1$ there is an exact sequence 
$0 \to \cE_{\PP^2,8} \to \cO^{\oplus 9} \oplus \cO(1) \to \cO(2) \oplus \cO(2) \to 0$.
\end{itemize}
In particular, for~$d \ge 2$ the bundle~$\cE_{\PP^2,9-d}$ is unique up to isomorphism,
and for~$d = 1$ the variety~$\PP_{\PP^2}(\cE_{\PP^2,9-d})$ is unique up to isomorphism.
\end{proposition}

\begin{proof}
By Lemma~\ref{lem:dpv-dpb} we have~$k \ge 2$ and~$d = K_{\PP^2}^2 - k = 9 - k \le 7$.

For~$d \le 4$ we consider the description of~$\cE$ provided by Lemma~\ref{lem:dpb-ideal} 
and combine it with the following resolutions (where we use the generality position property of~$z_1,\dots,z_d$)
\begin{align*}
0 \longrightarrow \cO(1) \longrightarrow \cO(2) \oplus \cO(2) \longrightarrow \cI_z(3) \longrightarrow 0,
\\
0 \longrightarrow \cO \longrightarrow \cO(1) \oplus \cO(2) \longrightarrow \cI_{z_1,z_2}(3) \longrightarrow 0,
\\
0 \longrightarrow \cO^{\oplus 2} \longrightarrow \cO(1)^{\oplus 3} \longrightarrow \cI_{z_1,z_2,z_3}(3) \longrightarrow 0,
\\
0 \longrightarrow \cO(-1) \longrightarrow \cO(1)^{\oplus 2} \longrightarrow \cI_{z_1,z_2,z_3,z_4}(3) \longrightarrow 0.
\end{align*}
Here the first, second, and fourth sequences are the Koszul resolutions, 
and the third follows easily by decomposing~$\cI_{z_1,z_2,z_3}(3)$ 
with respect to the exceptional collection~$(\cO,\cO(1),\cO(2))$ on~$\PP^2$.
Using the fact that~$\cO(1)$ is globally generated 
and the kernel of its evaluation morphism is~$\Omega(1)$ (by the Euler sequence),
we obtain the required descriptions.

For~$d \in \{5,6\}$ we deduce from Corollary~\ref{cor:dpb-duality} 
combined with the descriptions of~$\cE_{\PP^2,5}$ and~$\cE_{\PP^2,6}$ obtained above
the exact sequences
\begin{align*}
&
0 \longrightarrow \cE_{\PP^2,4} \longrightarrow \cO^{\oplus 9} \longrightarrow \cO(1) \oplus \cT(-1)^{\oplus 2} \longrightarrow 0,
\\
&
0 \longrightarrow \cE_{\PP^2,3} \longrightarrow \cO^{\oplus 9} \longrightarrow \cT(-1)^{\oplus 3} \longrightarrow 0,
\end{align*}
and combining them with the Euler sequence we also obtain the required descriptions.
Similarly, for~$d = 7$ the required description follows immediately from Corollary~\ref{cor:dpb-duality}.

The uniqueness of~$\cE_{\PP^2,9-d}$ for~$d \ge 2$ is obvious from the above construction.
Similarly, for~$d = 1$ the bundle~$\cE_{\PP^2,8}$ is determined by a point~$z \in \PP^2$,
hence the uniqueness of~$\PP_{\PP^2}(\cE_{\PP^2,8})$ 
follows from the transitivity of the~$\PGL_3$-action on~$\PP^2$.
\end{proof}

\begin{proposition}
\label{prop:bundle-p1p1}
Let~$\cE = \cE_{\PP^1 \times \PP^1,k}$ be a maximal del Pezzo bundle on the surface~$\PP^1 \times \PP^1$ with~\mbox{$\rc_2(\cE) = k$}.
Then~$d = \dd(\PP_{\PP^1 \times \PP^1}(\cE_{\PP^1 \times \PP^1,k})) = 8 - k \le 6$ and
\begin{itemize}
\item 
if~$d = 6$ then~$\cE_{\PP^1 \times \PP^1,2} \cong \cO(-1,-1)^{\oplus 2}$;
\item 
if~$d = 5$ then~$\cE_{\PP^1 \times \PP^1,3} \cong \cO(-1,-1) \oplus \cO(-1,0) \oplus \cO(0,-1)$;
\item 
if~$d = 4$ then~$\cE_{\PP^1 \times \PP^1,4} \cong \cO(-1,0)^{\oplus 2} \oplus \cO(0,-1)^{\oplus 2}$;
\item 
if~$d = 3$ then~$\cE_{\PP^1 \times \PP^1,5} \cong \cO(-1,0) \oplus \cO(0,-1) \oplus \Omega_{\PP^3}(1)\vert_{\PP^1 \times \PP^1}$;
\item 
if~$d = 2$ then~$\cE_{\PP^1 \times \PP^1,6} \cong \Omega_{\PP^3}(1)\vert_{\PP^1 \times \PP^1}^{\oplus 2}$;
\item 
if~$d = 1$ there is an exact sequence 
\begin{equation*}
0 \longrightarrow \cE_{\PP^1 \times \PP^1,7} \longrightarrow \cO^{\oplus 8} \oplus \cO(1,1) \longrightarrow \cO(1,2) \oplus \cO(2,1) \longrightarrow 0.
\end{equation*}
\end{itemize}
In particular, for~$d \ge 2$ the bundle~$\cE_{\PP^1 \times \PP^1,8-d}$ is unique up to isomorphism,
and for~$d = 1$ the variety~$\PP_{\PP^1 \times \PP^1}(\cE_{\PP^1 \times \PP^1,8-d})$ is unique up to isomorphism.
\end{proposition}

\begin{proof}
By Lemma~\ref{lem:dpv-dpb} we have~$k \ge 2$ and~$d = K_{\PP^1 \times \PP^1}^2 - k = 8 - k \le 6$.

For~$d \le 4$ we consider the description of~$\cE$ provided by Lemma~\ref{lem:dpb-ideal} 
and combine it with the following resolutions (where we use the generality position property of~$z_1,\dots,z_d$)
\begin{align*}
0 \longrightarrow \cO(1,1) \longrightarrow \cO(1,2) \oplus \cO(2,1) \longrightarrow \cI_z(2,2) \longrightarrow 0,
\\
0 \longrightarrow \cO \longrightarrow \cO(1,1) \oplus \cO(1,1) \longrightarrow \cI_{z_1,z_2}(2,2) \longrightarrow 0,
\\
0 \longrightarrow \cO(-1,0) \longrightarrow \cO(1,1) \oplus \cO(0,1) 
\longrightarrow \cI_{z_1,z_2,z_3}(2,2) \longrightarrow 0,
\\
0 \longrightarrow \cO^{\oplus 3} \longrightarrow \cO(0,1)^{\oplus 2} \oplus \cO(1,0)^{\oplus 2} 
\longrightarrow \cI_{z_1,z_2,z_3,z_4}(2,2) \longrightarrow 0.
\end{align*}
Here the first three sequences are the Koszul resolutions, 
and the last sequence follows easily by decomposing of the sheaf~$\cI_{z_1,z_2,z_3,z_4}(2,2)$ 
with respect to the exceptional collection~$(\cO,\cO(0,1),\cO(1,0),\cO(1,1))$ on~$\PP^1 \times \PP^1$.
Using the fact that~$\cO(0,1)$, $\cO(1,0)$, $\cO(1,1)$ are globally generated 
and the kernels of their evaluation morphisms 
are~$\cO(0,-1)$, $\cO(-1,0)$, and~$\Omega_{\PP^2}(1)\vert_{\PP^1\times \PP^1}$
we obtain the required descriptions.

For~$d \in \{5,6\}$ we deduce from Corollary~\ref{cor:dpb-duality} combined with 
the descriptions of~$\cE_{\PP^1 \times \PP^1,5}$ and~$\cE_{\PP^1 \times \PP^1,6}$ 
obtained above the exact sequences
\begin{align*}
&
0 \longrightarrow 
\cE_{\PP^1 \times \PP^1,3} \longrightarrow 
\cO^{\oplus 8} \longrightarrow 
\cO(1,0) \oplus \cO(0,1) \oplus \cT_{\PP^3}(-1)\vert_{\PP^1 \times \PP^1} \longrightarrow 0,
\\
&
0 \longrightarrow 
\cE_{\PP^1 \times \PP^1,2} \longrightarrow 
\cO^{\oplus 8} \longrightarrow 
\cT_{\PP^3}(-1)\vert_{\PP^1 \times \PP^1}^{\oplus 2} \longrightarrow 0,
\end{align*}
and also obtain the required descriptions.

The uniqueness is proved in the same way as in Proposition~\ref{prop:bundle-p2}.
\end{proof}

\begin{remark}
\label{rem:dpb-other}
Using Proposition~\ref{prop:dpb} and the descriptions of del Pezzo bundles on~$\PP^2$ and~$\PP^1 \times \PP^1$,
one can also describe del Pezzo bundles on other del Pezzo surfaces~$Z$.
Here we give a description in the cases where~$d \ge K_Z^2/2\ge 2$, i.e., $\rc_2 \le K_Z^2/2$,
(note that the cases with~$2 \le d < K_Z^2/2$ 
can be obtained from these via the construction of Corollary~\ref{cor:dpb-duality}).
In the following table we use the description of~$Z$ as the blowup of~$\PP^2$ in~$9 - K_Z^2$ points
and denote by~$\h$ the pullback of the line class from~$\PP^2$ 
by~$\e_i$, $1 \le i \le 9 - K_Z^2$, the classes of the exceptional lines,
and write~$\Omega(\h)$ for the pullback of~$\Omega(1)$.
\begin{equation*}
\begin{array}{c|c|l}
Z & d & \text{description}
\\\hline
\FF_1 & 6 & \cE_{\FF_1,2} \cong \cO(-2\h+\e_1) \oplus \cO(-\h)
\\
\FF_1 & 5 & \cE_{\FF_1,3} \cong \cO(-\h+\e_1) \oplus \cO(-\h)^{\oplus 2}
\\
\FF_1 & 4 & \cE_{\FF_1,4} \cong \cO(-\h+\e_1) \oplus \cO(-\h) \oplus \Omega(\h)
\\\hline
Z_7 & 5 & \cE_{Z_7,2} \cong \cO(-2\h+\e_1+\e_2) \oplus \cO(-\h)
\\
Z_7 & 4 & \cE_{Z_7,3} \cong \cO(-\h+\e_1) \oplus \cO(-\h+\e_2) \oplus \cO(-\h)
\\\hline
Z_6 & 4 & \cE_{Z_6,2} \cong \cO(-2\h+\e_1+\e_2+\e_3) \oplus \cO(-\h)
\\
Z_6 & 3 & \cE_{Z_6,3} \cong \cO(-\h+\e_1) \oplus \cO(-\h+\e_2) \oplus \cO(-\h+\e_3)
\end{array}
\end{equation*}
Finally, for the quintic del Pezzo surface~$Z_5$ we have
\begin{equation*}
\cE_{Z_5,2} \cong \cU
\qquad\text{and}\qquad 
\cE_{Z_5,3} \cong \cU^\perp,
\end{equation*}
where~$\cU$ is the pullback of the tautological bundle from the embedding~$Z_5 \hookrightarrow \Gr(2,5)$,
and for the quartic del Pezzo surface~$Z_4$ we have a $1$-dimensional family of del Pezzo bundles
that can be described as 
\begin{equation*}
\cE_{Z_4,2} \cong \cS_{t}\vert_{Z_4},
\end{equation*}
where~$t$ runs through the punctured pencil of smooth quadrics passing through~$Z_4$
and~$\cS_t$ stands for the spinor bundle on the corresponding $3$-dimensional quadric.
\end{remark}

\begin{remark}
\label{rem:drb-d1}
By Lemma~\ref{lem:dpb-ideal} any maximal del Pezzo bundle~$\cE$ on a del Pezzo surface~$Z$ with~$\rc_2(\cE) = K_Z^2 - 1$
fits into the exact sequence
\begin{equation*}
0 \longrightarrow \cE \longrightarrow H^0(Z, \cI_z(-K_Z)) \otimes \cO_Z \longrightarrow \cI_z(-K_Z) \longrightarrow 0,
\end{equation*}
where~$z \in Z$ is a point in general position and the second arrow is the evaluation map.
Conversely, the point~$z \in Z$ is determined by the vector bundle~$\cE$ as the image of the base point 
of the almost del Pezzo variety~$X = \PP_Z(\cE)$ of degree~$1$ under the contraction~$X \to Z$.
\end{remark}  

\subsection{Quadric bundles}

In this section we discuss almost del Pezzo varieties~$X$ that have a structure of a flat quadric bundle~$f \colon X \to \PP^1$.
Recall from the proof of Lemma~\ref{lemma:ext-rays0a}\ref{lemma:ext-rays0a1} 
that any such~$X$ is a divisor of relative degree~$2$ in~$\PP_X(\cF)$,
where~$\cF \cong (f_*\cO_X(A_X))^\vee$.
Note also that every vector bundle over~$\PP^1$ splits, hence we can write
\begin{equation*}
\cF \cong \moplus_{i=0}^n \cO(-a_i).
\end{equation*}
Finally, note that the Picard group of~$\PP_{\PP^1}(\cF)$ is generated by the relative hyperplane class~$A$
and the class of a fiber~$F$, hence~$X$ is linearly equivalent to~$2A - kF$ for some~$k \in \ZZ$.

\begin{lemma}
\label{lem:quadric-bundles}
Let~$X \subset \PP_{\PP^1}(\oplus \cO(-a_i))$ be a divisor of type~$2A - kF$.
If~$(X,A\vert_X)$ is an almost del Pezzo variety of degree~$d = \dd(X)$ then~$d \le 6$.
Moreover,
\begin{enumerate}
\item 
\label{it:sum-ai}
$\sum a_i = d - 2$ and $k = d - 4$;
\item 
\label{it:ai-nonnegative}
all~$a_i$ are nonnegative, with the only exception of the case~$d = 1$ where~$a = (-1,0^n)$.
\end{enumerate}
\end{lemma}

\begin{proof}
\ref{it:sum-ai}
First, using Proposition~\ref{DP:|A|}\ref{DP:|A|1} we obtain
\begin{equation*}
d + n - 1 = \dim H^0(X,\cO_X(A_X)) = \dim H^0(\PP^1,\cF^\vee) = \sum a_i + n + 1,
\end{equation*}
which gives the equality~$\sum a_i = d - 2$. 
Furthermore, the relation~$A^{n+1} = (\sum a_i) A^n \cdot F$ in the Chow ring of~$\PP_{\PP^1}(\cF)$ implies
\begin{equation*}
d = A^n \cdot (2A - kF) = 2\sum a_i - k,
\end{equation*}
and therefore~$k = d - 4$.

\ref{it:ai-nonnegative}
First, $H^1(\PP^1,\cF^\vee) = H^1(X,\cO_X(A_X)) = 0$ by Lemma~\ref{lem:vanishings}, hence~$a_i \ge -1$ for all~$i$.
Assume~$a_j < 0$, i.e., $a_j = -1$ for some~$j$.
Let~$C = \PP_{\PP^1}(\cO(-a_j)) \subset \PP_{\PP^1}(\oplus \cO(-a_i))$ be the corresponding section.
Then
\begin{equation*}
A \cdot C = a_j = -1,
\qquad 
X \cdot C = (2A - kF) \cdot C = -2 - (d - 4) = 2 - d.
\end{equation*}
The nef property of~$A_X$ shows that~$C \not\subset X$, hence~$X \cdot C \ge 0$, hence~$d \le 2$.
On the other hand, for~$d \le 2$ we have 
\begin{equation*}
\langle A_X, A_X - F \rangle = d - 2 \le 0,
\end{equation*}
so Lemma~\ref{lem:ad=1} shows that the divisor class~$A_X - F$ cannot be effective, hence we have~$a_i \le 0$ for all~$i$.
As~$\sum a_i = d - 2$, the only possibility is~$d = 1$ and~$a = (-1,0^n)$.

It remains to prove that~$d \le 6$.
For this note that a general intersection of~$n - 3$ divisors in the linear system~$|A_X|$ 
is an almost del Pezzo threefold~$X'$.
Clearly, it can be represented as a divisor of type~$2A - kF$ in~$\PP_{\PP^1}(\cF')$, 
where~$\cF'$ is a vector bundle from an exact sequence
\begin{equation*}
0 \longrightarrow \cF' \longrightarrow \cF \longrightarrow \cO^{\oplus (n - 3)} \longrightarrow 0,
\end{equation*}
so that~$\rank(\cF') = 4$ and~$\deg(\cF') = \deg(\cF) = -\sum a_i = 2 - d$.
Since~$X'$ is terminal, it has at most isolated singular points, hence the fibers of~$X' \to \PP^1$ are generically smooth.
This means that the morphism~$\cF' \to {\cF'}^\vee(-k)$ induced by the divisor~$X \in |2A - kF|$
is generically injective, hence
\begin{equation*}
2 - d \le d -2 - 4k = 14 - 3d,
\end{equation*}
and the inequality~$d \le 6$ follows.
\end{proof}

In the next proposition we give a classification of almost del Pezzo quadric bundles.
A more explicit description of the case~$d = 5$ is given in Lemma~\ref{lem:schubert} and Lemma~\ref{lem:schubert-flop-r2}.

\begin{proposition}
\label{prop:quadric-over-p1}
Let~$X \subset \PP_{\PP^1}(\oplus \cO(-a_i))$ be a divisor of type~$2A - kF$.
If~$X$ is a non-conical almost del Pezzo variety with~$d = \dd(X)$ and $\rr(X) = 2$ then~$d \le 5$.
Moreover,
\begin{itemize}
\item 
if~$d = 5$ then~$\dim(X) \le 5$, $a = (0^{n-2},1^3)$, $k = 1$;
\item 
if~$d = 4$ then~$a = (0^{n-1},1^2)$, $k = 0$,
and~$X$ is a complete intersection 
in~$\PP^1 \times \PP^{n+2}$ of three divisors of bidegree~$(1,1)$, $(1,1)$, and~$(0,2)$;
\item 
if~$d = 3$ then~$a = (0^{n},1)$, $k = -1$,
and~$X$ is a complete intersection in~$\PP^1 \times \PP^{n+1}$ of two divisors of bidegree~$(1,2)$ and~$(1,1)$;
\item 
if~$d = 2$ then~$a = (0^{n+1})$, $k = -2$,
and~$X$ is a divisor in~$\PP^1 \times \PP^{n}$ of bidegree~$(2,2)$;
\item 
if~$d = 1$ then~$a =(-1, 0^{n})$, $k = -3$,
and~$X$ is a complete intersection in~$\PP^1 \times \PP^{2n}$ of~$n$ divisors of bidegree~$(1,1)$ 
and one divisor of bidegree~$(1,2)$.
\end{itemize}
\end{proposition}

Note that for~$d \le 4$ the dimension~$\dim(X) = n$ may be arbitrary (greater than~$2$).

\begin{proof}
Recall from Lemma~\ref{lem:quadric-bundles} that~$d \le 6$.
Assume~$d = 6$.
Then the argument of Lemma~\ref{lem:quadric-bundles} shows that a general 3-dimensional linear section~$X'$ of~$X$
corresponds to a vector bundle~$\cF'$ of rank~$4$ and degree~$-4$ 
and a generically injective morphism~\mbox{$\cF' \to {\cF'}^\vee(-2)$}.
In this case the degree of the source and the target are the same, hence the morphism is an isomorphism,
and therefore~$X' \to \PP^1$ is an everywhere non-degenerate quadric surface bundle.
Therefore, its discriminant double covering is \'etale over~$\PP^1$, hence trivial, 
and hence the relative Picard number of~$X'$ over~$\PP^1$ is~2, so that~$\rr(X') = 3$.
But as we showed in Proposition~\ref{prop:class-divisor}, we have~$\rr(X') = \rr(X) = 2$.
This contradiction shows that~$d \le 5$.

Now let~$d \le 5$. 
Note that
\begin{equation*}
\langle A_X, A_X - 2F \rangle = d - 4 
\qquad\text{and}\qquad 
\langle A_X - 2F, A_X - 2F \rangle = d - 8,
\end{equation*}
hence by Lemma~\ref{lem:ad=1} the divisor class~$A_X - 2F$ cannot be effective.
This means that~$a_i \le 1$ for all~$i$.
Combining this with the bounds established in Lemma~\ref{lem:quadric-bundles}\ref{it:ai-nonnegative},
we conclude that
\begin{equation*}
a = 
\begin{cases}
(0^{n+3-d},1^{d-2}), & \text{if~$d \ge 2$ and}
\\
(-1, 0^{n}), & \text{if~$d = 1$.}
\end{cases}
\end{equation*} 

Let~$d = 5$, so that~$a = (0^{n-2},1^3)$.
If~$n \ge 6$ then the morphism
\begin{equation*}
\cO^{\oplus (n-2)} \oplus \cO(-1)^{\oplus 3} = \cF \longrightarrow \cF^\vee(-1) = \cO^{\oplus 3} \oplus \cO(-1)^{\oplus (n-2)}
\end{equation*}
contains a summand~$\cO$ in the kernel.
This means that there is a section of~$X \to \PP^1$ that consists of singular points of the fibers 
and which is contracted to a point~$\bv \in X_\can$ by the anticanonical morphism of~$X$.
Since any singular quadric is covered by lines passing through its singular point,
we conclude that~$X_\can$ is covered by lines passing through~$\bv$, hence~$X_\can$ is a cone.
Thus, $n \le 5$. 

Let~$2 \le d \le 4$.
To obtain a complete intersection description of~$X$ note that 
the vector bundle~$\cF \cong\cO^{\oplus (n + 3 - d)} \oplus \cO(-1)^{\oplus (d - 2)}$ fits into an exact sequence
\begin{equation*}
0 \longrightarrow \cF \longrightarrow \cO_{\PP^1}^{\oplus (n + d - 1)} \longrightarrow 
\cO_{\PP^1}(1)^{\oplus (d - 2)} \longrightarrow 0,
\end{equation*}
hence the projective bundle~$\PP_{\PP^1}(\cF)$ is a complete intersection in~$\PP^1 \times \PP^{n+d-2}$ 
of~$d - 2$ divisors of bidegree~$(1,1)$, and~$X$ is its intersection with an extra divisor of bidegree~$(-k,2)$.

In the last case~$d = 1$ 
we use the exact sequence
\begin{equation*}
0 \longrightarrow \cF(-1) \longrightarrow \cO_{\PP^1}^{\oplus (2n + 1)} \longrightarrow \cO_{\PP^1}(1)^{\oplus n} \longrightarrow 0
\end{equation*}
to describe~$\PP_{\PP^1}(\cF)$ and~$X$ as complete intersections.
\end{proof}

\begin{remark}
It is easy to show that the quadric bundles~$X \to \PP^1$ such that~$\dd(X) = 6$ and~$\rr(X) = 3$
are isomorphic to~$\FF_0 \times_{\PP^1} \FF_0 \cong \PP^1 \times \PP^1 \times \PP^1$,
or to~$\FF_1 \times_{\PP^1} \FF_1$.
\end{remark}

We conclude this section with a lemma explaining the relation 
of quadric bundle structures to~$\PP^{n-2}$-bundle structures in the case~$\dd(X) = 5$.

\begin{lemma}
\label{lem:schubert}
If~$f \colon X \to \PP^1$ is a quadric bundle such that~$X$ 
is a non-conical almost del Pezzo variety with~$\dd(X) = 5$ and~$\rr(X) = 2$
then there is a flop~$X \dashrightarrow \PP_{\PP^2}(\cE)$, 
where~$\cE$ is a del Pezzo bundle that fits into an exact sequence
\begin{equation*}
0 \longrightarrow \cE \longrightarrow \cO(-1)^{\oplus 2} \oplus \Omega(1) \longrightarrow \cO^{\oplus (5-n)} \longrightarrow 0.
\end{equation*}
\end{lemma}

\begin{proof}
The map~$f \colon X \to \PP^1$ is an extremal $K$-negative contraction.
Since~$\rr(X) = 2$, the variety~$X$ (possibly after a flop) has another $K$-negative extremal contraction.
By Lemma~\ref{lem:lattice-q-p2} the variety~$X$ is primitive, hence the other contraction cannot be birational.
On the other hand, the other contraction cannot be a quadric fibration over~$\PP^1$, 
because the $\PP^1$-class~$F$ given by the fiber of~$f$ is the unique $\PP^1$-class on~$X$.
Therefore, the other contraction must be a $\PP^{n-2}$-bundle over~$\PP^2$.
By Proposition~\ref{prop:bundle-p2} and Lemma~\ref{lem:dpb-maximal} 
the corresponding bundle~$\cE$ fits into the required exact sequence.
\end{proof}

In~\S\ref{sec:schubert} we will describe the flop explicitly (see Lemma~\ref{lem:schubert-flop-r2})
and show that the corresponding del Pezzo variety is a Schubert divisor in~$\Gr(2,5)$.    

\section{Classification of del Pezzo varieties}
\label{sec:proofs}

In this section we prove all the main results of the paper as stated in the Introduction.

\subsection{Reduction to primitive varieties}

We start with proving the reduction theorem. 

\begin{proof}[Proof of Theorem~\xref{thm:imprimitive}]
If~$X$ is primitive there is nothing to prove.
So, assume~$X$ is imprimitive. 
Consider the class group~$\Cl(X)$ and let~$E_1,\dots,E_k \in \Cl(X)$ 
be a maximal sequence of pairwise orthogonal exceptional classes, i.e.,
\begin{equation*}
\langle A_X, E_i \rangle = 1,
\qquad\text{and}\qquad
\langle E_i, E_j \rangle = -\updelta_{i,j},
\qquad
1 \le i \le j \le k.
\end{equation*}
By Proposition~\ref{prop:exceptional-contraction} there is a $\QQ$-factorialization~$\xi \colon \hX \to X$
such that~\mbox{$\hX \cong \Bl_{\hP_1,\dots,\hP_k}(\hX_0)$},
where~$\hX_0$ is a $\QQ$-factorial almost del Pezzo variety and~$\hP_1,\dots,\hP_k \in \hX_0$ are smooth points.
We set~$X_0 \coloneqq (\hX_0)_\can$.
Note that the points~$\hP_i$ do not lie on the $K$-trivial curves of~$\hX_0$ by Lemma~\ref{lemma:ext-rays0},
hence there is a commutative diagram
\begin{equation*}
\xymatrix@C=3em{
\hX \ar@{=}[r] \ar[d]_\xi &
\Bl_{\hP_1,\dots,\hP_k}(\hX_0) \ar[r]^-{\hat\sigma} \ar[d]^{\tilde\xi_0} &
\hX_0 \ar[d]^{\xi_0}
\\
X &
\Bl_{P_1,\dots,P_k}(X_0) \ar[r]^-\sigma &
X_0,
}
\end{equation*}
where~$\xi_0$ and~$\tilde\xi_0$ are small contractions, 
$P_i = \xi_0(\hP_i)$ are smooth points of~$X_0$, and~$\sigma$ and~$\hat\sigma$ are the blowups.
The exceptional divisors~$\hE_i$ of~$\hat\sigma$ are the preimages of the exceptional divisors~$E_i$ of~$\sigma$.
By Corollary~\ref{cor:targets} the blowup~$\Bl_{P_1,\dots,P_k}(X_0)$ is almost del Pezzo, 
and hence Lemma~\ref{lem:dp-weak-dp} implies
\begin{equation*}
(\Bl_{P_1,\dots,P_k}(X_0))_\can \cong
(\Bl_{\hP_1,\dots,\hP_k}(\hX_0))_\can = \hX_\can \cong X.
\end{equation*}

It remains to check that~$X_0$ is primitive.
For this just note that~$\Cl(X_0) \subset \Cl(X)$ 
is the orthogonal of the exceptional classes~$E_1,\dots,E_k$ by Lemma~\ref{lem:contractions-class}\ref{it:class-blowup}, 
so if~$X_0$ is imprimitive, Corollary~\ref{cor:imprimitivity} shows that there is an exceptional class~$E_0 \in \Cl(X_0)$,
which is orthogonal to~$E_1,\dots,E_k$ by Lemma~\ref{lem:contractions-class}\ref{it:class-blowup},
which contradicts to the maximality of this collection.
\end{proof}

\begin{remark}
It is not true in general that the primitive contraction~$X_0$ is uniquely determined by~$X$
(essentially because there may be different maximal sequences of pairwise orthogonal exceptional classes),
see for instance Corollary~\ref{cor:blowups}.
\end{remark} 

\subsection{Biregular classification}
\label{ss:proofs-bireg}

In this subsection we prove Theorem~\ref{cla:dP}.
We start with two lemmas.

\begin{lemma}
\label{lem:special-sextic}
The variety~$\hX_{6,3,3} \coloneqq \Bl_{P_1,P_2}(\PP^3)$ is almost del Pezzo of degree~$6$
and its anticanonical model is a singular hyperplane section~$X_{6,3,3} \subset \PP^2 \times \PP^2$.
\end{lemma}

\begin{proof}
Let~$V$ be a $4$-dimensional vector space, let~$L_1,L_2 \subset V$ be two distinct $1$-dimensional subspaces, 
set~$\bar{V}_i \coloneqq V/L_i$, and let~$p_i \colon V \to \bar{V}_i$ be the projections.
Consider the map 
\begin{equation*}
\PP(V) \dashrightarrow \PP(\bar{V}_1) \times \PP(\bar{V}_2) 
\qquad 
v \longmapsto (p_1(v),p_2(v)).
\end{equation*}
It extends to a regular 
morphism~$\xi \colon \hX_{6,3,3} \coloneqq \Bl_{P_1,P_2}(\PP(V)) \to \PP(\bar{V}_1) \times \PP(\bar{V}_2)$, 
where~$P_i$ are the points of~$\PP(V)$ corresponding to~$L_i$.
Since~$\hX_{6,3,3}$ is almost del Pezzo of degree~$6$ by Proposition~\ref{prop:dP4:constr-i} 
and the composition of~$\xi$ with the Segre embedding of~$\PP(\bar{V}_1) \times \PP(\bar{V}_2)$ 
is the anticanonical morphism of~$\hX_{6,3,3}$, the anticanonical model
\begin{equation*}
X_{6,3,3} \coloneqq \xi(\hX_{6,3,3}) \subset \PP(\bar{V}_1) \times \PP(\bar{V}_2) \cong \PP^2 \times \PP^2
\end{equation*}
is a sextic del Pezzo threefold.
By construction~$X_{6,3,3}$ is the hyperplane section corresponding to the bilinear form
\begin{equation*}
\bar{V}_1 \otimes \bar{V}_2 \longrightarrow 
(V / (L_1 \oplus L_2)) \otimes (V / (L_1 \oplus L_2)) \longrightarrow
\wedge^2(V / (L_1 \oplus L_2)),
\end{equation*}
which is obviously singular 
at the point~$(p_1(P_2),p_2(P_1)) \in \PP(\bar{V}_1) \times \PP(\bar{V}_2)$.
\end{proof}

\begin{remark}
\label{rem:x633-flop}
Lemma~\ref{lem:dpb-converse} and Remark~\ref{rem:dpb-other} provide a pseudoisomorphism
\begin{equation*}
\Bl_{P_1,P_2}(\PP^3) \cong \Bl_{P_2}(\PP_{\PP^2}(\cE_{\PP^2,2})) \dashrightarrow 
\PP_{\FF_1}(\cE_{\FF_1,2}) \cong \PP_{\FF_1}(\cO(-2\h + \e) \oplus \cO(-\h)) \cong \FF_1 \times_{\PP^1} \FF_1,
\end{equation*}
thus the fiber product~$\FF_1 \times_{\PP^1} \FF_1$ is another $\QQ$-factorialization of~$X_{6,3,3}$.
\end{remark}

\begin{lemma}
\label{lem:three-schuberts}
The varieties
\begin{equation}
\label{eq:hx5}
\begin{aligned}
\hX_{5,2,5} &\coloneqq \PP_{\PP^2}(\cO(-1)^{\oplus 2} \oplus \Omega(1)),
\quad \\
\hX_{5,3,4} &\coloneqq \Bl_P(\PP^2 \times \PP^2),
\quad \\
\hX_{5,4,3} &\coloneqq \Bl_P(\PP^1 \times \PP^1 \times \PP^1)
\end{aligned}
\end{equation} 
are almost del Pezzo of degree~$5$ and their anticanonical models are linear sections of~$\Gr(2,5)$
by~$1$, $2$, and~$3$ general Schubert divisors, respectively.
\end{lemma}

\begin{proof}
Let~$V$ be a $5$-dimensional vector space.
Consider the diagram
\begin{equation*}
\xymatrix{
& 
\Fl(1,2;V) \ar[dl]_{p_1} \ar[dr]^{p_2}
\\
\PP(V) &&
\Gr(2,V),
}
\end{equation*}
where~$p_1$ is the projectivization of the twisted tangent bundle~$\cT_{\PP(V)}(-2)$
and~$p_2$ is the projectivization of the tautological bundle~$\cU$ on~$\Gr(2,V)$.
Let~$K \subset V$ be a $3$-dimensional subspace.
Passing to the preimage of the plane~$\PP(K) \subset \PP(V)$ we obtain the diagram
\begin{equation*}
\xymatrix{
& 
\Fl(1,2;V) \times_{\PP(V)} \PP(K) \ar[dl]_{f} \ar[dr]^{\xi}
\\
\PP(K) &&
X_{5,2,5},
}
\end{equation*}
where~$f$ is the projectivization of~$\cT_{\PP(V)}(-2) \vert_{\PP(K)} \cong \cO(-1)^{\oplus 2} \oplus \Omega(1)$
and~$\xi$ is a Springer resolution of the degeneracy locus~$X_{5,2,5} \subset \Gr(2,V)$ of the morphism
\begin{equation*}
\cU \hookrightarrow V \otimes \cO \longrightarrow V/K \otimes \cO.
\end{equation*}
Thus, $X_{5,2,5} \subset \Gr(2,V)$ is a Schubert divisor 
and~$\xi$ has nontrivial fibers (isomorphic to~$\PP^1$) over the plane~$\Gr(2,K)$,
hence~$\xi$ is a small morphism, 
$\hX_{5,2,5} \cong \PP_{\PP^2}(\cO(-1)^{\oplus 2} \oplus \Omega(1))$ is almost del Pezzo, 
and~$X_{5,2,5}$ is its anticanonical model. 

Now consider $3$-dimensional subspaces~$K_1,K_2 \subset V$ with~$\dim(K_1 \cap K_2) = 1$
and the map 
\begin{equation*}
\PP(K_1) \times \PP(K_2) \dashrightarrow \Gr(2,V),
\qquad 
(v_1,v_2) \longmapsto \langle v_1,v_2 \rangle.
\end{equation*}
It extends to a regular morphism~$\xi \colon \hX_{5,3,4} \coloneqq \Bl_P(\PP(K_1) \times \PP(K_2)) \to \Gr(2,V)$, 
where~$P$ is the point that corresponds to~$K_1 \cap K_2$.
Since~$\hX_{5,3,4}$ is almost del Pezzo of degree~$5$ by Proposition~\ref{prop:dP4:constr-i} 
and the composition of~$\xi$ with the Pl\"ucker embedding of~$\Gr(2,V)$ 
is the anticanonical morphism of~$\hX_{5,3,4}$, the image
\begin{equation*}
X_{5,3,4} \coloneqq \xi(\hX_{5,3,4}) \subset \Gr(2,V)
\end{equation*}
is a quintic del Pezzo fourfold.
By construction~$X_{5,3,4}$ is the intersection of the Schubert hyperplanes corresponding to~$K_1$ and~$K_2$.

Finally, consider three $3$-dimensional subspaces~$K_i \subset V$, $1 \le i \le 3$, such that
\begin{equation*}
\dim(K_1 \cap K_2) = 
\dim(K_1 \cap K_3) = 
\dim(K_2 \cap K_3) = 1
\qquad\text{and}\qquad 
K_1 \cap K_2 \cap K_3 = 0
\end{equation*} 
and consider the map
\begin{equation*}
\PP(K_1^\perp) \times \PP(K_2^\perp) \times \PP(K_3^\perp) \dashrightarrow \Gr(2,V),
\qquad 
(f_1,f_2,f_3) \longmapsto \langle f_1,f_2,f_3 \rangle^\perp.
\end{equation*}
It extends to a regular morphism~$\xi \colon \hX_{5,4,3} \coloneqq \Bl_P(\PP(K_1^\perp) \times \PP(K_2^\perp) \times \PP(K_3^\perp)) \to \Gr(2,V)$, 
where~$P$ is the point that corresponds to the unique $3$-dimensional subspace~$K_0 \subset V$ 
such that~$K_0^\perp \cap K_i^\perp \ne 0$ for~$1 \le i \le 3$.
Since $\hX_{5,4,3}$ is almost del Pezzo of degree~$5$ by Proposition~\ref{prop:dP4:constr-i} 
and the composition of~$\xi$ with the Pl\"ucker embedding of~$\Gr(2,V)$ 
is the anticanonical morphism of~$\hX_{5,4,3}$, the image
\begin{equation*}
X_{5,4,3} \coloneqq \xi(\hX_{5,4,3}) \subset \Gr(2,V)
\end{equation*}
is a quintic del Pezzo threefold.
By construction~$X_{5,4,3}$ is the intersection of the Schubert hyperplanes corresponding to~$K_1$, $K_2$, and~$K_3$.
\end{proof}

\begin{proof}[Proof of Theorem~\xref{cla:dP}]
We start with the cases where~$d \ge 6$.
First, assume~$\rr(X) = 1$.
By Proposition~\ref{prop:dP4:constr-i} if~$P \in X$ is a general point the blowup~$\tX \coloneqq \Bl_P(X)$ 
is an imprimitive almost del Pezzo variety with~\mbox{$\rr(\tX) = 2$} and~$\dd(\tX) = d - 1 \ge 5$.
Therefore, by Lemma~\ref{lem:another-blowup} there is a flop~$\tX \dashrightarrow \tX'$ such that~$\tX'$
admits a non-birational $K$-negative extremal contraction.
This contraction cannot be a quadric bundle over~$\PP^1$ by Lemma~\ref{lem:lattice-q-p2} because~$\tX$ is imprimitive,
therefore, Proposition~\ref{propo:ext-rays} implies that~$\tX' \cong \PP_Z(\cE)$
where~$Z$ is a del Pezzo surface with~$\rr(Z) = \rr(\tX') - 1 = 1$, i.e., $Z \cong \PP^2$.
Finally, by Lemma~\ref{lem:lattice-p-p2} the imprimitivity of~$\tX'$ implies that~$\dd(\tX') = 7$,
and Proposition~\ref{prop:bundle-p2} gives
\begin{equation*}
\tX' \cong 
\PP_{\PP^2}(\cO(-1) \oplus \cO(-2)) \cong \Bl_P(\PP^3).
\end{equation*}
But~$\Bl_P(\PP^3)$ is a smooth del Pezzo variety, hence it does not have small birational modifications,
hence~$\tX \cong \tX'$
and since by Lemma~\ref{lem:another-blowup} the exceptional class in~$\Cl(\tX)$ is unique, 
we conclude that~$X \cong \PP^3$ and~$\dd(X) = 8$.

Now assume that~$X$ is primitive, $\dd(X) \ge 6$, but~$\rr(X) \ge 2$.
Let~$\tX$ be a $\QQ$-factorialization of~$X$.
It must have a non-birational $K$-negative extremal contraction, hence by
Propositions~\ref{prop:bundle-p2}, Proposition~\ref{prop:bundle-p1p1}, and~Proposition~\ref{prop:quadric-over-p1},
the only possibilities are
\begin{align*}
\tX \cong X_{6,2,4} &\coloneqq \PP_{\PP^2}(\cO(-1)^{\oplus 3}) \cong \PP^2 \times \PP^2
\qquad\text{or}\qquad \\
\tX \cong X^*_{6,3,3} &\coloneqq \PP_{\PP^1 \times \PP^1}(\cO(-1,-1)^{\oplus 2}) \cong \PP^1 \times \PP^1 \times \PP^1,
\end{align*}
or~$\tX$ is a smooth hyperplane section of the first of these 
(because a singular hyperplane section of~$\PP^2 \times \PP^2$ is imprimitive by Lemma~\ref{lem:special-sextic}). 

Finally, assume~$\dd(X) \ge 6$ and~$X$ is imprimitive.
By Theorem~\ref{thm:imprimitive} there is a primitive del Pezzo variety~$X_0$ with~$\dd(X_0) = \dd(X) + k \ge 7$
such that~$X$ is the anticanonical model of a blowup of~$X_0$.
The above argument proves that~$X_0 \cong \PP^3$, hence~$k \le 2$,
and either~\mbox{$X \cong \Bl_P(\PP^3)$}, the del Pezzo variety of degree~7,
or~$X \cong \Bl_{P_1,P_2}(\PP^3)_\can$, 
which is a singular hyperplane section of~$\PP^2 \times \PP^2$ by Lemma~\ref{lem:special-sextic}.  

Next we consider the case where~$\dd(X) = 5$.
If~$X$ is primitive, we have~$\rr(X) \le 3$ by Corollary~\ref{cor:primitive-rank},
and if~$X$ is imprimitive, it is pseudoisomorphic to~$\Bl_{P}(X_0)$ for a del Pezzo variety of degree~$6$,
and as we have seen above~$\rr(X_0) \le 3$, hence~$\rr(X) \le 4$. 

First, assume~$\rr(X) = 1$.
If~$P \in X$ is a general point the blowup~\mbox{$\tX \coloneqq \Bl_P(X)$}
is an imprimitive almost del Pezzo variety (by Proposition~\ref{prop:dP4:constr-i})
with~\mbox{$\rr(\tX) = 2$} and~$\dd(\tX) = 4$,
hence it is pseudoisomorphic to an almost del Pezzo variety~$\tX'$ 
that admits another $K$-negative extremal contraction.
This contraction cannot be a projective bundle over~\mbox{$\PP^1 \times \PP^1$} because~$\rr(\tX) = 2$,
cannot be birational, because by Lemma~\ref{lem:another-blowup} the exceptional class in~$\Cl(\tX)$ is unique, 
and cannot be a quadric bundle over~$\PP^1$, because the latter is primitive by Lemma~\ref{lem:lattice-q-p2}.
Therefore, Proposition~\ref{propo:ext-rays} implies that~$\tX'$ must be a $\PP^{n-2}$-bundle over~$\PP^2$, 
and Proposition~\ref{prop:bundle-p2} together with Lemma~\ref{lem:dpb-maximal} imply that~$\tX'$ 
is a linear section of~\mbox{$\PP_{\PP^2}(\cO(-1) \oplus \Omega(1)^{\oplus 2})$}.
On the other hand, in~\cite{KP-Mu} it was shown that there is a flop
\begin{equation*}
\Bl_P(\Gr(2,5)) \dashrightarrow \PP_{\PP^2}(\cO(-1) \oplus \Omega(1)^{\oplus 2}).
\end{equation*}
Therefore, $\tX$ is pseudoisomorphic to a linear section of~$\Bl_P(\Gr(2,5))$, 
hence~$X$ is isomorphic to a linear section of~$\Gr(2,5)$.

Next, assume~$\rr(X) = 2$.
Then~$X$ cannot be imprimitive, because, as we have shown before, 
there are no del Pezzo varieties~$X_0$ with~$\dd(X_0) = 6$ and~$\rr(X_0) = 1$.
It also cannot be pseudoisomorphic to a projective bundle over~\mbox{$\PP^1 \times \PP^1$} because~$\rr(X) = 2$.
If it is a quadric bundle over~$\PP^1$, it is pseudoisomorphic to a $\PP^{n-2}$-bundle over~$\PP^2$ by Lemma~\ref{lem:schubert}.
Therefore, by Proposition~\ref{propo:ext-rays} and Lemma~\ref{lem:dpb-maximal} we may assume that~$X$ is 
a linear section of~$\PP_{\PP^2}(\cE)$, 
where~$\cE$ is a maximal del Pezzo bundle with~$\rc_2(\cE) = 4$.
By Proposition~\ref{prop:bundle-p2} and Lemma~\ref{lem:three-schuberts} 
we have~$\PP_{\PP^2}(\cE) = \hX_{5,2,5}$ and any linear section of its anticanonical image
is a linear section of~$\Gr(2,5)$.

Next, assume~$\rr(X) = 3$.
Then~$X$ must be imprimitive by Lemma~\ref{lem:lattice-p-p1p1}, 
hence it is pseudoisomorphic to~$\Bl_P(X_0)$, where~$\dd(X_0) = 6$ and~$\rr(X_0) = 2$.
As we showed above, the variety~$X_0$ must be isomorphic to a linear section of~$\PP^2 \times \PP^2$, 
hence~$X$ is isomorphic to a linear section of the anticanonical image of~$\Bl_P(\PP^2 \times \PP^2)$.
Applying Lemma~\ref{lem:three-schuberts} we see that~$X$ is a linear section of~$\Gr(2,5)$.

Finally, assume~$\rr(X) = 4$.
Again~$X$ must be imprimitive, and moreover, it must be pseudoisomorphic to~$\Bl_P(\PP^1 \times \PP^1 \times \PP^1)$.
Applying Lemma~\ref{lem:three-schuberts} we see that~$X$ is a linear section of~$\Gr(2,5)$. 

In the remaining cases~$\dd(X) \le 4$ the theorem is classical and follows from classification 
of del Pezzo surfaces of the corresponding degrees as complete intersections 
(see, e.g., \cite[\S2]{Isk:Fano1e}, \cite{Shin1989}, \cite{Fujita1990}).
\end{proof}

\subsection{ADE classification}
\label{ss:proofs}

In this subsection we explain the ADE classification of del Pezzo varieties
described in Theorems~\ref{thm:primitive-classification}, \ref{thm:intro-clx} and~\ref{thm:intro-bircla}.

\begin{proof}[Proof of Theorem~\xref{thm:primitive-classification}]
If~$\rr(X) = 1$ it is enough to show that~$\dd(X)$ is in the given list, which follows from Theorem~\ref{cla:dP}
because in the cases where~$\dd(X) \in \{6,7\}$ we have~$\rr(X) \ge 2$.

If~$\rr(X) \ge 2$, we deduce from Proposition~\ref{propo:ext-rays} and Corollary~\ref{cor:dP4:constr2} 
that~$X$ has an extremal $K$-negative contraction to~$\PP^1$, $\PP^2$, or~$\PP^1 \times \PP^1$, 
which is a quadric bundle in the first case,
and a $\PP^{n-2}$-bundle in the last two. 
For contractions to~$\PP^1\times\PP^1$ 
we apply Lemma~\ref{lem:dpb-maximal} and Proposition~\ref{prop:bundle-p1p1} to prove~$\dd(X) \le 6$,
and then Lemma~\ref{lem:lattice-p-p1p1} to exclude odd~$\dd(X)$.
Similarly, for contractions to~$\PP^2$ 
we apply Lemma~\ref{lem:dpb-maximal} and Proposition~\ref{prop:bundle-p2} to prove~\mbox{$\dd(X) \le 7$},
and then Lemma~\ref{lem:lattice-p-p2} to exclude~$\dd(X) = 7$.
Finally, for contractions to~$\PP^1$ we apply Proposition~\ref{prop:quadric-over-p1} to prove~$\dd(X) \le 5$,
and then apply Lemma~\ref{lem:schubert} and Lemma~\ref{lem:lattice-q-p2} 
to exclude the cases~$\dd(X) = 5$ and~$\dd(X) = 3$, respectively.
\end{proof}

\begin{proof}[Proof of Theorem~\xref{thm:intro-clx}]
The most important part
of the theorem has been proved in Proposition~\ref{prop:class-divisor} and Lemma~\ref{lem:clx-xix-blowup};
it only remains to describe the possible sublattices~\mbox{$\Xi(X) \subset \Cl(S)$}.
We start with the cases where~$X$ is primitive.

First, consider the case where~$\rr(X) = 1$ and~$X \ne \PP^3$ (i.e., $\dd(X) \le 5$ by Theorem~\ref{cla:dP}).
Then~$\Cl(X)$ is generated by~$A_X$, hence~$i^*\Cl(X)$ is generated by~$A_S = -K_S$, hence
\begin{equation*}
\Xi(X) = K_S^\perp \subset \Cl(S).
\end{equation*}
This lattice is well-known to be the root lattice of Dynkin type~$\rE_8$, $\rE_7$, $\rE_6$, $\rD_5$, or~$\rA_4$
(see, e.g., \cite[Theorem~25.4]{Manin:book:74}).
Similarly, if~$X = \PP^3$, then~$\Cl(X)$ is generated by~$\frac12A_X$, 
we have~$i^*(\frac12A_X) = \frac12A_S$, where~$S \cong \PP^1 \times \PP^1$, 
and its orthogonal is the root lattice of type~$\rA_1$.

Now consider the case where~$\rr(X) \ge 2$. 
Replacing $X$ with its $\QQ$-factorialization (it does not change~$\Cl(X)$ by Lemma~\ref{lem:pseudo-clx})
we may assume that~$X$ is a $\QQ$-factorial almost del Pezzo variety
with an extremal non-birational contraction as in Proposition~\ref{propo:ext-rays}.

First, assume that~$f \colon X = \PP_Z(\cE) \to Z$ is a projective bundle over a minimal del Pezzo surface~$Z$.
In this case, $S$ has a natural birational morphism~$f_S \colon S \to Z$ 
(the restriction of~$f$, see the proof of Lemma~\ref{lemma:ext-rays0a}\ref{lemma:ext-rays0a2}), 
and by Lemma~\ref{lem:contractions-class}\ref{it:class-pn-2} 
the image of~$\Cl(X)$ in~$\Cl(S)$ is generated by~$A_X$ and~$f_S^*\, \Cl(Z)$.
Note that~$f_S$ is the blowup of~$k =\dd(Z) - \dd(X)$ points.
If~$\e_1,\dots,\e_k \in \Cl(S)$ are the classes of the exceptional divisors of~$f_S$,
the orthogonal to~$f_S^*\, \Cl(Z)$ in~$\Cl(S)$ is the sublattice generated by~$\e_1,\dots,\e_k$,
and the orthogonal in~$\langle \e_1,\dots,\e_k \rangle$ to~\mbox{$K_S = f_S^*K_Z + \e_1 + \dots + \e_k$} is generated by the roots
\begin{equation*}
\balpha_1 = \e_1 - \e_2,
\quad 
\balpha_2 = \e_2 - \e_3,
\quad 
\dots,
\quad 
\balpha_{k-1} = \e_{k-1} - \e_k,
\end{equation*}
which obviously generate the root lattice of type~$\rA_{k-1}$.

Next, assume that~$f \colon X \to \PP^1$ is a quadric fibration with~$\rr(X) = 2$.
Then~$\dd(X) \le 5$ by Proposition~\ref{prop:quadric-over-p1}, and if~$\dd(X) = 5$ 
then~$X$ has a $\QQ$-factorial crepant model with a structure of~$\PP^{n-2}$-bundle over~$\PP^2$ by Lemma~\ref{lem:schubert}
which was already considered above.
So, we may assume that~$\dd(X) \le 4$.
Then the induced morphism~$f_S \colon S \to \PP^1$ is a conic bundle with~$4 \le k \le 7$ degenerate fibers.
We can choose one component~$\e_i$, $1 \le i \le k$, in each of the degenerate fibers 
in such a way that the contraction of these components is~$\bar{S} \cong \PP^1 \times \PP^1$,
so that the map~$f_S$ is the composition~$S \to \bar{S} \to \PP^1$, 
where the second map is the projection to the second factor.
Let~$\f_1$ and~$\f_2$ be the classes of the rulings of~$\bar{S}$.
By Lemma~\ref{lem:contractions-class}\ref{it:class-quadric} 
the group~$i^*\Cl(X)$ is generated by~$\f_2$ and~$K_S = -2\f_1 - 2\f_2 + \e_1 + \dots + \e_k$. 
The orthogonal of~$\f_2$ in~$\Cl(S)$ is the sublattice generated by~$\f_2$ and~$\e_1,\dots,\e_k$,
and the orthogonal in~$\langle \f_2,\e_1,\dots,\e_k \rangle$ to~$K_S$ is generated by
\begin{equation*}
\balpha_1 = \e_1 - \e_2,
\quad 
\balpha_2 = \e_2 - \e_3,
\quad 
\dots,
\quad 
\balpha_{k-1} = \e_{k-1} - \e_k,
\quad\text{and}\quad 
\balpha_k = \f_2 - \e_{k-1} - \e_k.
\end{equation*}
These roots obviously generate the root lattice of type~$\rD_k$.

Now, finally, if~$X$ is imprimitive, the theorem follows from a combination of Theorem~\ref{thm:imprimitive}
and Lemma~\ref{lem:clx-xix-blowup} (applied several times).
\end{proof}

\begin{proof}[Proof of Theorem~\xref{thm:intro-bircla}]
By Theorem~\ref{thm:imprimitive} we have
\begin{equation*}
X \cong (\Bl_{P_1,\dots,P_k}(X_0))_\can, 
\end{equation*}
where~$X_0$ is primitive, and by Lemma~\ref{lem:clx-xix-blowup} we have~$\Xi(X) = \Xi(X_0)$.

First, assume~$\Xi(X_0) = \rA_m$.
Then the proof of Theorem~\ref{thm:intro-clx} shows either~$\rr(X_0) = 1$ and~$\dd(X_0) \in \{5,8\}$
(i.e., $X_0 \cong \PP^3$, or~$X_0 \cong \Gr(2,5) \cap \PP^{n + 3}$ by Theorem~\ref{cla:dP}),
or a $\QQ$-factorialization~$\hX_0$ of~$X_0$ is isomorphic to~$\PP_Z(\cE)$, 
where~$Z = \PP^2$ or~$Z = \PP^1 \times \PP^1$ and~$\cE$ is a del Pezzo bundle.

Next, assume~$\Xi(X_0) = \rD_m$.
Then, again, the proof of Theorem~\ref{thm:intro-clx} shows that 
either~$\rr(X_0) = 1$ and~$\dd(X_0) = 4$, hence~$X_0$ is an intersection of two quadrics in~$\PP^{n+2}$ by Theorem~\ref{cla:dP},
or a $\QQ$-factorialization~$\hX_0$ of~$X_0$ is a quadric bundle over~$\PP^1$, 
hence~$\hX_0$ is a complete intersection of the specified type by Proposition~\ref{prop:quadric-over-p1}
(recall that~$\dd(X_0) \ne 3$ by Lemma~\ref{lem:lattice-q-p2}).

Finally, assume~$\Xi(X_0) = \rE_m$.
Then, the proof of Theorem~\ref{thm:intro-clx} shows that~$\rr(X_0) = 1$ and~$\dd(X_0) \le 3$, 
and the description of~$X_0$ as a hypersurface in a weighted projective space is established in Theorem~\ref{cla:dP}.
\end{proof}

\subsection{Type~$\rA$}

In this section we provide a detailed classification in type~$\rA$.

\begin{proof}[Proof of Theorem~\xref{thm:intro-dpa}]
By Theorem~\ref{thm:intro-bircla}\ref{it:bircla-am} we have
\begin{equation*}
X \cong (\Bl_{P_1,\dots,P_{r-k}}(\PP_{Z_0}(\cE_0)))_\can,
\end{equation*}
where~$Z_0$ is a minimal del Pezzo surface and~$\cE_0$ is a del Pezzo bundle on~$Z_0$.
Applying Lemma~\ref{lem:dpb-converse} several times we conclude that~$X \cong \PP_Z(\cE)_{\can}$ 
for a (non minimal) del Pezzo surface~$Z$ and a del Pezzo bundle~$\cE$.
The formula for~$\dd(X)$ is proved in Lemma~\ref{lem:dpv-dpb} 
and the formula for~$\rr(X)$ follows from~$\rr(Z) = 10 - K^2_Z$.
\end{proof} 

\begin{proof}[Proof of Theorem~\xref{thm:intro-dpb}]
This is a combination of Lemma~\ref{lem:dpv-dpb}, Lemma~\ref{lem:dpb-maximal},  
Proposition~\ref{prop:bundle-p2}, Proposition~\ref{prop:bundle-p1p1}, 
Remark~\ref{rem:dpb-other} and Remark~\ref{rem:drb-d1}.
\end{proof} 

\begin{proof}[Proof of Corollary~\xref{cor:drn}]
If~$X$ is of type~$\rD_m$ or~$\rE_m$, then~$\dd(X) + \rr(X) \le 6$ by Theorem~\ref{thm:intro-bircla},
so if~$\dd(X) + \rr(X) \ge 7$, then~$X$ is of type~$\rA_m$ 
and the required inequality follows from Theorem~\ref{thm:intro-dpa} and Corollary~\ref{cor:bundle-rdn} because~$n \ge 3$.
\end{proof}

\begin{proof}[Proof of Theorem~\xref{thm:intro-a-detailed}]
Let~$X$ be a maximal del Pezzo variety of type~$\rA$.
By Theorem~\ref{thm:intro-dpa} we have~$X \cong \PP_Z(\cE)_\can$ 
where~$Z$ is a del Pezzo surface and~$\cE$ is a del Pezzo bundle.
Since~$K_Z^2 = 11 - \rr(X)$, the surface~$Z$ is unique up to isomorphism for~$\rr(X) \in \{2,4,5,6\}$;
moreover if additionally~$\dd(X) \ge 2$, or~$\dd(X) = 1$ and~$\rr(X) \le 5$, 
a del Pezzo bundle~$\cE$ on~$Z$ is also unique by Theorem~\ref{thm:intro-dpb}.
We are left with the four cases listed in part~\ref{it:a-moduli}, and with the case~$\rr(X) = 3$.

Assume~$\dd(X) = 1$ and~$\rr(X) \ge 6$ (hence~$K_Z^2 \le 5$).
Since~$\PP_Z(\cE)$ is a crepant model of~$X$, and~$X$ has a finite number of crepant models,
the isomorphism class of~$X$ depends on as many parameters as pairs~$(Z,\cE)$ do.
Now the surface~$Z$ depends on~$10 - 2K^2_Z = 2\rr(X) - 12$ parameters, 
and by Theorem~\ref{thm:intro-dpb}\ref{it:almost-uniqueness} the bundle~$\cE$ depends on a general point~$z \in Z$
which gives other~$2$ parameters.
Thus, we have a family of del Pezzo varieties of dimension
\begin{equation*}
2\rr(X) - 12 + 2 = 2\rr(X) - 10 = 2\rr(X) - \dd(X) - 9.
\end{equation*}

Similarly, if~$\dd(X) = 2$ and~$\rr(X) = 7$ (hence~$K_Z^2 = 4$), 
the surface~$Z$ depends on~$2$ parameters and the bundle~$\cE$ depends on~$1$ parameter (see Remark~\ref{rem:dpb-other}).
Altogether, we have~$3$ parameters, and this number matches~$2\rr(X) - \dd(X) - 9$.

Now assume~$\rr(X) = 3$, hence~$K_Z^2 = 8$.
We have two del Pezzo surfaces of degree~$8$, one is~$\PP^1 \times \PP^1$ and the other is~$\FF_1$,
and by Theorem~\ref{thm:intro-dpb}
each of these surfaces~$Z$ has a unique (up to isomorphism of the projectivization)
del Pezzo bundle~$\cE_{Z,k}$ for each~$2 \le k \le 7$.
Note that~$\PP_{\FF_1}(\cE_{\FF_1,k})$ is always imprimitive by Corollary~\ref{cor:imprimitivity},
while~$\PP_{\PP^1 \times \PP^1}(\cE_{\PP^1 \times \PP^1,k})$ is primitive 
if and only if its degree is in~$\{2,4,6\}$ by Lemma~\ref{lem:lattice-p-p1p1}.
This gives the cases in part~\ref{it:a-pair} of the theorem.
It remains to show that
\begin{equation*}
\big(\PP_{\PP^1 \times \PP^1}(\cE_{\PP^1 \times \PP^1,k}) \big)_\can \cong
\big(\PP_{\FF_1}(\cE_{\FF_1,k})\big)_\can,
\qquad 
\text{for~$k \in \{3,5,7\}$.}
\end{equation*}
Since by Lemma~\ref{lem:lattice-p-p1p1} and Corollary~\ref{cor:imprimitivity} both sides are imprimitive,
using Lemma~\ref{lem:dpb-converse} we can write both of them as~$\big(\Bl_x(X_0)\big)_\can$,
where
\begin{equation*}
X_0 \coloneqq \PP_{\PP^2}(\cE_{\PP^2,k})_{\can}
\end{equation*}
and~$x \in X_0$
is a point in general position (i.e., such that the blowup of~$x$ is almost del Pezzo).
So, it is enough to show that such point is unique up to automorphisms of~$X_0$.

When~$k = 3$, we have~$X_0 \cong \PP^2 \times \PP^2$ by Proposition~\ref{prop:bundle-p2}, and the uniqueness is obvious.

When~$k = 5$, the variety~$X_0$ has degree~$4$, hence it is imprimitive by Lemma~\ref{lem:lattice-p-p2}, hence
\begin{equation*}
X_0 \cong \big(\Bl_{x'}(\Gr(2,5))\big)_\can
\end{equation*}
since~$\Gr(2,5)$ is the unique maximal del Pezzo variety of with~$d = 5$ and~$r = 1$.
If~$U$ and~$U'$ are the $2$-dimensional subspaces corresponding to the points~$x$ and~$x'$, 
the generality assumption for~$x$ means that~$U \cap U' = 0$.
Since~$\PGL_5$ acts transitively on such pairs~$(U,U')$, the uniqueness follows.

Finally, let~$k = 7$.
Using the description of~$\cE_{\PP^2,7}$ in Proposition~\ref{prop:bundle-p2} 
it is easy to see that~$X_0$ is isomorphic to the double covering of~$\PP^8$ 
ramified over the quartic 
\begin{equation*}
B = \left\{ 
\det
\left(\begin{smallmatrix}
0&x_{1}&x_{2}&x_{3}\\
x_{1}&x_{9}&x_{4}&x_{5}\\
x_{2}&x_{4}&x_{8}&x_{6}\\
x_{3}&x_{5}&x_{6}&x_{7}\\
\end{smallmatrix}\right) = 0
\right\} \subset \PP^8;
\end{equation*}
in other words, $B$ is the discriminant hypersurface in the space of quadrics in~$\PP^3$ 
passing through a fixed point~$p \in \PP^3$.
Since by Proposition~\ref{prop:dP4:constr-i} the point~$x$ cannot lie 
on the ramification divisor of the morphism~$X_0 \to \PP^8$,
its image is in~$\PP^8 \setminus B$, so it is enough to check that~$\Aut(X_0)$ 
acts transitively on~$\PP^8 \setminus B$, the space of smooth quadrics in~$\PP^3$ passing through~$p$.
But this is obvious because~$\Aut(X_0)$ contains the stabilizer of~$p$ in~$\PGL_4$,
and the group~$\PGL_4$ acts transitively on the set of pairs consisting of a smooth quadric with a point.

This completes the proof of the theorem.
\end{proof}  

We note the following funny consequence.
We denote by~$X_{d+1,2,9-d} = \PP_{\PP^2}(\cE_{\PP^2,8-d})$ 
and~$X^*_{d,3,9-d} = \PP_{\PP^1 \times \PP^1}(\cE_{\PP^1 \times \PP^1,8-d})$, as in Theorem~\ref{thm:intro-a-detailed}.    

\begin{corollary}
\label{cor:blowups}
For~$d \in \{2,4,6\}$, for any point~$P \in X^*_{d,3,9-d}$, and for any pair of points~\mbox{$Q_1,Q_2 \in X_{d+1,2,9-d}$} 
such that~$\Bl_P(X^*_{d,3,9-d})$ and~$\Bl_{Q_1,Q_2}(X_{d+1,2,9-d})$ are almost del Pezzo,
there is a pseudoisomorphism
\begin{equation*}
\Bl_P(X^*_{d,3,9-d}) \dashrightarrow \Bl_{Q_1,Q_2}(X_{d+1,2,9-d}).
\end{equation*}
\end{corollary}

\begin{proof}
Indeed, the anticanonical model of the left side is the maximal del Pezzo variety~$X_{d-1,4,9-d}$,
and the same is true for the right side.
Therefore, the existence of a pseudoisomorphism follows from the uniqueness of~$X_{d-1,4,9-d}$. 
\end{proof}  

\begin{remark}
In the case where~$d = 6$ a pseudoisomorphism between
\begin{equation*}
\Bl_P(\PP_{\PP^1 \times \PP^1}(\cE_{\PP^1 \times \PP^1,2})) = \Bl_{P}(\PP^1 \times \PP^1 \times \PP^1)
\quad\text{and}\quad 
\Bl_{Q_1,Q_2}(\PP_{\PP^2}(\cE_{\PP^2,2})) \cong \Bl_{Q_1,Q_2,Q_3}(\PP^3)
\end{equation*}
is well-known, see, e.g., \cite[Theorem~3.1]{KP:rho}.
\end{remark}     

\subsection{Del Pezzo varieties of degree~$1$ and del Pezzo surfaces}
\label{ss:degree-1}  

Recall the definitions of the discriminant hypersurface~\mbox{$\fD(X) \subset \PP(H^0(X,\cO_X(A_X)))^\vee$} 
of a del Pezzo variety~$X$ of degree~$1$,
as well as the universal anticanonical divisor~$\bH(S) \subset S \times \PP(H^0(S,\cO_X(-K_S)))$,
and the divisor~$\Pi_L \subset \bH(S)$ associated to a del Pezzo surface~$S$ and a line~$L \subset S$
(see the Introduction and~\S\ref{ss:dpb-duality}). 

\begin{proof}[Proof of Theorem~\xref{thm:dp1-4-cubic}]
Let~$X$ be a maximal del Pezzo variety of type~$\rA_{n-2}$ with~$\dd(X) = 1$.
Note that~$\dim(X) = n$.
By Theorem~\ref{thm:intro-dpa} we have
\begin{equation*}
X \cong \PP_{Z}(\cE)_\can,
\end{equation*}
where~$Z$ is a del Pezzo surface of degree~$K_Z^2 = n$ 
and~$\cE$ is a maximal del Pezzo bundle on~$Z$ with~$\rk(\cE) = \rc_2(\cE) = n - 1$.
Let~$x_0 \in \PP_Z(\cE)$ be the base point, let~$z_0 \in Z$ be its image, and let~$S \coloneqq \Bl_{z_0}(Z)$.
By Proposition~\ref{prop:bl-pze-hblz} we have a pseudoisomorphism
\begin{equation}
\label{eq:x-bh-pseudo}
\chi \colon\Bl_{x_0}(\PP_Z(\cE)) \dashrightarrow \bH(S)
\end{equation}
compatible with the natural elliptic fibration structures of these spaces 
induced by the anticanonical map of~$\PP_Z(\cE)$ 
and the projection~$\operatorname{pr} \colon \bH(S) \to \PP(H^0(S, \cO(-K_S)))$.
Furthermore, $\chi$ is an isomorphism over the complement of a linear subspace of codimension~$2$.
Therefore, the discriminant divisors for these elliptic fibrations agree,
and since the first is~$\fD(X)$ by definition, and the second is the projective dual of~$S$, 
we obtain an isomorphism
\begin{equation*}
\fD(X) \cong S^\vee,
\end{equation*}
which by projective duality implies~$S \cong \fD(X)^\vee$.  

Furthermore, the pseudoisomorphism~$\chi$ identifies the canonical divisors
\begin{equation*}
K_{\Bl_{x_0}(X)} = (1 - n)(A_X - E_0)
\qquad\text{and}\qquad 
K_{\bH(S)} = (1 - n)H,
\end{equation*}
where~$H$ is the pullback to~$\bH(S)$ of the hyperplane class of~$\PP(H^0(S, \cO(-K_S)))$.
Moreover, by Proposition~\ref{prop:bl-pze-hblz} the pseudoisomorphism~$\chi$ identifies the divisor~$E_0$ 
with the divisor~$\Pi_{L_0}$, where~$L_0$ is the line on~$S$ over~$z_0$, 
hence the divisor~\mbox{$A_X = (A_X - E_0) + E_0$} on~$\Bl_{x_0}(X)$ 
is identified by~$\chi$ with the divisor~$H + \Pi_{L_0}$.
Therefore, the corresponding graded algebras and their projective spectra are identified as well, i.e.,
\begin{multline}
\label{eq:x-xsl}
X \cong \Proj \left( \moplus_{m \ge 0} H^0(X, \cO_X(mA_X)) \right) 
\cong \Proj \left( \moplus_{m \ge 0} H^0(\Bl_{x_0}(X), \cO_{\Bl_{x_0}(X)}(mA_X)) \right) \\
\cong \Proj \left( \moplus_{m \ge 0} H^0(\bH(S), \cO_{\bH(S)}(m(H + \Pi_{L_0}))) \right) 
\eqqcolon X_{S,L_0},
\end{multline}
where we use the last equality above as the definition of~$X_{S,L_0}$.

It remains to check that the variety~$X_{S,L_0}$ does not depend on the choice of a line~$L_0$ on~$S$,
i.e., that~$X_{S,L} \cong X_{S,L'}$ for any pair of lines.
Note that if~$n \ge 6$ there is nothing to prove, 
because both~$X_{S,L}$ and~$X_{S,L'}$ are maximal del Pezzo varieties with the same~$d = 1$ and~$r \le 5$,
hence they are isomorphic by Theorem~\ref{thm:intro-a-detailed}.
So, we assume~$n \le 5$.

First, assume~$L$ and~$L'$ do not intersect.
Let~$S \to Z$, $S \to Z'$ and~$S \to Z_0$ be the contraction of~$L$, $L'$, and both~$L$ and~$L'$, respectively,
so that there are two points~$z_0,z'_0 \in Z_0$ in general position such that
\begin{equation*}
Z = \Bl_{z'_0}(Z_0),
\qquad 
Z' = \Bl_{z_0}(Z_0),
\qquad\text{and}\qquad 
S = \Bl_{z_0}(Z) = \Bl_{z_0,z'_0}(Z_0) = \Bl_{z'_0}(Z').
\end{equation*}
Then by Proposition~\ref{prop:dpb} and Lemma~\ref{lem:dpb-converse}
there is a maximal del Pezzo bundle~$\cE_0$ on~$Z_0$ 
with a pseudoisomorphism~$X \dashrightarrow \Bl_{x'_0}(X_0)$, 
where~$X_0 \coloneqq \PP_{Z_0}(\cE_0)$ is an almost del Pezzo variety of degree~$2$,
and a maximal del Pezzo bundle~$\cE'$ on~$Z'$ 
with a pseudoisomorphism~\mbox{$X' \coloneqq \PP_{Z'}(\cE') \dashrightarrow \Bl_{x_0}(X_0)$}.
Moreover, as~$x_0$ is the base point of~$X$ and~$X$ is pseudoisomorphic to~$\Bl_{x'_0}(X_0)$,
we have~$x_0 = \tau(x'_0)$, where~$\tau$ is the Geiser involution of~$X_0$, and
hence the pseudoisomorphisms
\begin{equation*}
\Bl_{x_0}(X) \dashleftarrow \Bl_{x_0,x'_0}(X_0) \dashrightarrow \Bl_{x'_0}(X')
\end{equation*}
swap the divisors~$E_0$ and~$E'_0$.
Combining this with the pseudoisomorphism~\eqref{eq:x-bh-pseudo} 
and a similar pseudoisomorphism for~$\Bl_{x'_0}(\PP_{Z'}(\cE'))$,
we obtain a pseudoautomorphism of~$\bH(S)$ that swaps the divisors~$\Pi_{L}$ and~$\Pi_{L'}$. 
Therefore, it induces an isomorphism of the graded algebras from~\eqref{eq:x-xsl}
and of their projective spectra,
i.e., an isomorphism~\mbox{$X_{S,L} \cong X_{S,L'}$}.

Finally, assume that~$L,L'$ is an arbitrary pair of distinct lines.
Since~$n \le 5$ and~$S$ is a del Pezzo surface of degree~$n - 1 \le 4$, 
we can find a line~$L''$ on~$S$ which intersects neither~$L$ nor~$L'$.
Then we have isomorphisms~$X_{S,L} \cong X_{S,L''}$, and similarly~$X_{S,L'} \cong X_{S,L''}$.
Composing these isomorphisms, we obtain~$X_{S,L} \cong X_{S,L'}$.
\end{proof} 

\section{Effective, moving, and ample cones}
\label{sec:cones}

In this section we describe the structure of the space~$\Cl(X) \otimes \RR$ associated with an almost del Pezzo variety~$X$.
This space contains three important cones:

\begin{itemize}
\item 
the \emph{effective cone} $\Eff(X)$ --- the closed convex cone generated by the classes of effective divisors; 
\item 
the \emph{moving cone} $\Mov(X)$ --- the closed convex cone generated by the classes of movable divisors,
i.e., divisors~$M$ such that
the linear system~$|M|$ is non-empty and has no fixed components;
\item 
the \emph{ample cone} $\Amp(X)$ --- the closed convex cone generated by the classes of ample divisors.
\end{itemize}
The first two cones~$\Eff(X)$ and~$\Mov(X)$ are invariant under pseudoisomorphisms:
\begin{equation*}
\Eff(X') = \psi_*\Eff(X)
\qquad\text{and}\qquad 
\Mov(X') = \psi_*\Mov(X)
\end{equation*}
for any pseudoisomorphism~$\psi \colon X \dashrightarrow X'$,
so they are the same for all crepant models of a given del Pezzo variety.
The third cone~$\Amp(X)$, on a contrary, depends on a crepant model (and determines it).
We will usually consider~$\Amp(X)$ only when~$X$ is $\QQ$-factorial; 
in this case it has full dimension and coincides with the cone of nef classes.
Clearly, 
\begin{equation*}
\Amp(X)\subset \Mov(X)\subset \Eff(X).
\end{equation*}
Note that a priori the elements of~$\Eff(X)$ or~$\Mov(X)$ could be represented by non-effective or even non-rational classes. 
Elements of~$\Eff(X)$ are called \emph{pseudoeffective}. 

\subsection{Description via special classes}

In this section we describe the cones~$\Eff(X)$ and~$\Mov(X)$ in terms of~$\PP^1$- and~$\PP^2$-classes in~$\Cl(X)$,
see Definition~\ref{def:exP1P2}.
Since this definition is given purely in terms of the canonical bilinear form on~$X$, 
we conclude from Proposition~\ref{prop:class-divisor} that~$D$ is a $\PP^s$-class 
if and only if~$D\vert_Y$ is a $\PP^s$-class for a general fundamental divisor~$Y \subset X$
if and only if~$D\vert_S$ is a $\PP^s$-class for a general linear surface section~$S \subset X$.

We start with an evident observation about $\PP^s$-classes on del Pezzo surfaces.

\begin{lemma}
\label{lem:2dP}
If~$X$ is a del Pezzo surface and~$D \in \Cl(X)$ is a $\PP^s$-class with~$s \in \{1,2\}$
then~$D$ is semiample and induces a morphism~$X \to \PP^s$ which is a conic bundle for~$s = 1$ and a blowup for~$s = 2$.
In particular, $D$ is not ample for~$\dd(X) \le 8$.
\end{lemma} 

The following results generalize this to higher dimensions.

\begin{lemma}
\label{lemma:P1-P2:f}
Let~$X$ be an almost del Pezzo variety and let $D\in \Cl(X)$ be a $\PP^s$-class with~$s \in \{1,2\}$.
If~$D$ is nef then the linear system~$|D|$ is base point free and defines a contraction~$X\to \PP^s$.
\end{lemma}

\begin{proof}
By the base point free theorem \cite[Theorem~3.3]{KM:book} the class~$D$ is semiample
and defines a contraction~$f \colon X\to Z$ 
such that~$D=f^*D_Z$ for an ample divisor class~$D_Z$.
Consider the alternatives for~$Z$ given by Corollary~\ref{cor:targets}.

If~$Z$ is an almost del Pezzo variety birational to~$X$ then~$D_Z$ is also a $\PP^s$-class 
by Lemma~\ref{lem:pseudo-clx} and~Lemma~\ref{lem:contractions-class},
hence the restriction of~$D_Z$ to a general linear surface section~$S_Z \subset Z$ 
is an ample $\PP^s$-class, which is impossible by Lemma~\ref{lem:2dP}.
Similarly, If~$Z$ is a del Pezzo surface, $D_Z$ is again a $\PP^s$-class by Lemma~\ref{lem:contractions-class},
so Lemma~\ref{lem:2dP} shows that~$Z = \PP^2$ and~$s = 2$.
Finally, if~$Z = \PP^1$ then~$D$ is a $\PP^1$-class.
\end{proof}

\begin{proposition}
\label{prop:P1-P2}
Let~$X$ be a del Pezzo variety and let $D\in \Cl(X)$ be a 
$\PP^{s}$-class, \mbox{$s \in \{1,2\}$}.
Then there exists a $\QQ$-factorialization~$\xi \colon \hX\to X$ 
such that the strict transform~\mbox{$(\xi^{-1})_*D$} is nef and defines a contraction~$\hX \to \PP^s$.
\end{proposition}

\begin{proof}
First, we prove that no $\QQ$-factorial almost del Pezzo variety~$X$ can have a $D$-negative
and $K_X$-negative extremal contraction.
For this we use induction on~$n = \dim(X)$.
If~$n = 2$ there is nothing to prove because~$D$ is nef by Lemma~\ref{lem:2dP}.
Assume~$n > 2$ and let~$Y \subset X$ be a general fundamental divisor.
Let~$f \colon X \to Z$ be a $D$-negative and $K_X$-negative extremal contraction.
By Proposition~\ref{propo:ext-rays} this is the blowup of a smooth point,
or a $\PP^{n-2}$-bundle over a del Pezzo surface, or a quadric bundle over~$\PP^1$.
In the first case, if~$E \cong \PP^{n-1}$ is the exceptional divisor, $Y \cap E = \PP^{n-2}$.
In the other cases, the fibers of~$Y \to Z$ are hyperplane sections of the fibers of~$X \to Z$,
and at least one of them has positive dimension (this is obvious when~$n > 3$ or~$f$ is a quadric bundle, 
and when~$n = 3$ and $f$ is a $\PP^1$-bundle this follows from~\eqref{eq:dpb-inequality} as~$\dd(Y) = \dd(X) > \dd(Z)$).
Thus, in all these cases there is a curve~$C \subset Y$ which generates in~$X$ the extremal ray contracted by~$f$.
This curve is $D\vert_Y$-negative and~$K_Y$-negative, 
hence there must be a $D\vert_Y$-negative and $K_Y$-negative extremal contraction of~$Y$,
which is impossible by the induction hypothesis. 

Now we consider the $\QQ$-factorialization~$\xi \colon \hX \to X$ provided by Lemma~\ref{lem:d-mmp}.
As we proved above, the alternative~\ref{cone:case1} is impossible, 
hence~$\hD = (\xi^{-1})_*D$ is nef, and Lemma~\ref{lemma:P1-P2:f} gives the required contraction.
\end{proof} 

Now we are ready to describe the cones.
For any convex cone~$\rC \in \RR^N$ a ray~$\rR \subset \rC$ is called \emph{extremal} 
if a nonzero element~$r \in \rR$ cannot be represented as a sum~$r = r_1 + r_2$ with~$r_1, r_2 \in \rC \setminus \rR$.
This is equivalent to the existence of a linear function~$\nu \colon \RR^N \to \RR$ 
such that~$\nu(\rC) \subset \RR_{\ge 0}$ and~$\rC \cap \{\nu = 0\} = \rR$.
We start with an easy observation.

\begin{lemma}
\label{rem:extrmal:Eff}
Let $X$ be an almost del Pezzo variety. 
\begin{enumerate}
\item 
\label{it:ray-exceptional}
Any exceptional class $E \in \Cl(X)$ generates an extremal ray of~$\Eff(X)$.
\item 
\label{it:ray-p1-p2}
Any $\PP^1$- or~$\PP^2$-class lies in the boundaries of~$\Eff(X)$ and~$\Mov(X)$,
and if~\mbox{$\rr(X) = 2$} it generates an extremal ray in each of these cones.
\end{enumerate}
\end{lemma}

\begin{proof}
\ref{it:ray-exceptional}
If~$E$ is exceptional there is a $\QQ$-factorialization~$\xi \colon \hX \to X$ 
such that~$\hX$ is the blowup of an almost del Pezzo variety~$X'$ at a smooth point 
and~$\hE = (\xi^{-1})_*E$ is the exceptional divisor of the blowup (Proposition~\ref{prop:exceptional-contraction}).
Then the function~$\nu(D) \coloneqq \langle A_{X'}, D \rangle$ 
has the required property by Lemma~\ref{lem:contractions-class}\ref{it:class-blowup} and Lemma~\ref{lem:ad=1}.

\ref{it:ray-p1-p2}
If~$D$ is a $\PP^s$-class, $s \in \{1,2\}$, there is a $\QQ$-factorialization~$\xi \colon \hX \to X$ 
and a morphism~$f \colon \hX \to \PP^s$ such that~$\hD = (\xi^{-1})_*D$ is the pullback of the hyperplane class 
(Proposition~\ref{prop:P1-P2}).
The argument of Lemma~\ref{lemma:ext-rays0a} shows that 
the general fiber of~$f$ is a projective space~$\PP^{n-2}$ or a quadric~$Q^{n-1}$,
hence a curve of minimal degree in the general fiber of~$f$ is a free rational curve.
In this case the function~$\nu(D) \coloneqq D \cdot C$ has the required property.
\end{proof}

The next lemma will be used in the proof of the following proposition.

\begin{lemma}
\label{lem:nef-semiample}
Let~$X$ be an almost del Pezzo variety and let~$M$ be a movable divisor on~$X$.
Then for~$m\gg 0$ the base locus of~$|mM|$ is contained in the exceptional locus 
of the anticanonical morphism~$X \to X_\can$.
\end{lemma}

\begin{proof}
Let~$\xi \colon \hX \to X_\can$ be a $\QQ$-factorialization provided by Lemma~\ref{lem:d-mmp} for~$D = M$.
Since~$M$ is movable, case~\ref{cone:case1} is impossible, hence the strict transform~$\hM$ of~$M$ to~$\hX$ is nef. 
By the base point free theorem~\cite[Theorem~3.3]{KM:book} the class~$|m\hM|$ is base point free for~$m \gg 0$. 
On the other hand, the composition~$X \dashrightarrow X_\can \dashrightarrow \hX$
is an isomorphism away from the exceptional loci of~$X \to X_\can$ and~$\hX \to X_\can$, hence the claim.
\end{proof}

Recall that~$X$ has a finite number 
of $\QQ$-factorializations~$\chi_i \colon X_i \to X$ (see Lemma~\ref{lemma:repant-models}).

\newcounter{tmp}
\begin{proposition}
\label{prop:cone}
Let $X$ be a del Pezzo variety with $\rr(X)>1$. 
\begin{enumerate}
\item 
\label{it:mov-amp}
The moving cone~$\Mov(X)$ has the following chamber decomposition
\begin{equation}
\label{eq:Mov:chamber}
\Mov(X) = \bigcup_i {\chi_i}_* \big(\Amp (X_i)\big)
\end{equation}
where~$\chi_i \colon X_i \to X$ are all $\QQ$-factorializations of~$X$,
and the class~$A_X$ is contained in the interior of ~$\Mov(X)$.
\setcounter{tmp}{\value{enumi}}
\end{enumerate}
Moreover, the cones~\mbox{$\Eff(X)$} and~$\Mov(X)$ in~$\Cl(X) \otimes \RR$ are rational polyhedral and
\begin{enumerate}\setcounter{enumi}{\value{tmp}}
\item 
\label{it:eff-rays:e}
each extremal ray of~$\Eff(X)$ is generated by an exceptional, $\PP^1$-, or $\PP^2$-class;
\item 
\label{it:eff-rays:m}
each extremal ray of~$\Mov(X)$ is generated by a $\PP^1$-class, $\PP^2$-class, or a big class~$H$ 
such that for~$m\gg 0$ the linear system~$|mH|$ 
defines a birational map~$\psi \colon X \dashrightarrow X'$ to a del Pezzo variety~$X'$ 
with~$\rr(X')=1$.
\end{enumerate}
\end{proposition}

\begin{proof} 
\ref{it:mov-amp} 
The chamber decomposition of~$\Mov(X)$ follows from a combination of~\cite[Proposition~1.11]{hu-keel:mds}, 
\cite[Corollary~1.3.2]{BCHM}, and~\cite[Lemma-Definition~2.6]{P-Sh:JAG}. 

To prove that~$A_X$ lies in the interior of~$\Mov(X)$, let~$D$ be any divisor on~$X$.
For each $\QQ$-factorialization~$\chi_i \colon X_i \to X$
the Mori cone~$\NE(X_i)$ is generated by a finite number of extremal rays~$\rR_{i,j}$; 
each of these is generated by a class of a curve~$\gamma_{i,j}$. 
Therefore there exists a constant~$c\in \RR$
such that the strict transforms~$D_i =(\chi_i^{-1})_*D$ on~$X_i$ of~$D$ satisfy~$D_i\cdot\gamma_{i,j} > c$;
we can and will assume that~$c$ is negative.
Then for~$0\le \epsilon < -1/c$ the $\RR$-divisor~\mbox{$A_{X_i}+\epsilon D_i$} is \emph{almost nef}, 
i.e., all curves having negative intersection with it are $A_{X_i}$-trivial, 
hence~$X_i$ cannot have~\mbox{$(A_{X_i}+\epsilon D_i)$}-negative $K$-negative contractions.
Now consider the $\QQ$-factorialization provided by Lemma~\ref{lem:d-mmp},
it coincides with one of the $\QQ$-factorializations~\mbox{$\chi_i \colon X_i \to X$} discussed above,
hence case~\ref{cone:case1} is impossible.
Therefore, the divisor class~$A_{X_i}+\epsilon D_i$ is nef, and thus
\begin{equation*}
A_{X_i}+\epsilon D_i\in \Amp(X_i)\subset \Mov(X_i) = \Mov(X).
\end{equation*}
Since $D$ is an arbitrary divisor, this shows that $A_{X}$ is contained in the interior of~$\Mov(X)$.

\ref{it:eff-rays:e}, \ref{it:eff-rays:m}
Let~$D \in \Cl(X) \otimes \RR$ be a class generating an extremal ray of~$\Eff(X)$ or~$\Mov(X)$.
Let~$\xi \colon \hX \to X$ be a $\QQ$-factorialization provided by Lemma~\ref{lem:d-mmp}, 
let~$\hD \coloneqq (\xi^{-1})_*D$,
and consider the two cases~\ref{cone:case1} and~\ref{cone:case2}.

First, consider case~\ref{cone:case1} and let~$f \colon \hX \to Z$ be the corresponding extremal contraction.
In this case~$\hD \not\in \Mov(\hX)$, because~$f$ is~$\hD$-negative. 
Thus,~$\hD$ is a generator of an extremal ray in~$\Eff(X)$.
Since~$\hD$ is pseudoeffective, the contraction~$f$ is birational. 
Let~$E$ be the exceptional divisor of~$f$ and let~$D_Z \coloneqq f_*\hD$.
Then~$D_Z \in \Eff(Z)$ and $\hD = f^*D_Z + aE$, where~$a>0$ because~$f$ is~$\hD$-negative.
Since~$D$ generates an extremal ray of~$\Eff(X)$, we must have~$D_Z = 0$, 
i.e., $D$ is proportional to the exceptional class~$E$. 

Now consider case~\ref{cone:case2}. 
Since~$\hX$ is $\QQ$-factorial, $\Amp(\hX)$ is the nef cone,
so~\mbox{$\hD \in \Amp(\hX)$} and generates in~$\Amp(\hX)$ an extremal ray; in particular~$\hD$ is rational.
By the base point free theorem~\cite[Theorem~3.3]{KM:book} there is a morphism~$f \colon \hX \to Z$ 
such that~$\hD = f^*D_Z$ for some ample~$D_Z \in \Pic(Z)$.
Since~$D$ generates an extremal ray, we have~$\rk \Pic(Z) = 1$.
If~$f$ is non-birational, using Corollary~\ref{cor:targets} 
we conclude that either~$Z = \PP^2$ (hence~$D$ is proportional to a $\PP^2$-class)
or~$Z = \PP^1$ (hence~$D$ is proportional to a $\PP^1$-class).
Finally, if~$f$ is birational, then~$Z$ is a del Pezzo variety by Corollary~\ref{cor:targets},
hence~$D_Z$ is a multiple of~$A_Z$ (because~$\rk \Pic(Z) = 1$). 
Thus we may assume that $\hD=f^*A_Z$. 
In particular, $\hD$ is big.
But since any big class can be written as a sum of effective and ample 
classes (see, e.g., \cite[Lemma~0-3-3]{KMM})
and since any ample class lies in the interior of~$\Mov(\hX) = \Mov(X)$, the class~$D$ cannot be an extremal ray of~$\Eff(X)$.
It remains to show that~$\rr(Z)=1$. 

Assume for a contrary that~$\rr(Z) > 1$. 
Then~\ref{it:mov-amp} implies that~$A_Z = M_1 + M_2$, 
where~$M_1$ and~$M_2$ are movable divisors non-proportional to~$A_Z$.
The argument of Corollary~\ref{cor:targets} shows that there is a factorization
\begin{equation*}
f \colon \hX \overset{\hat f}\longrightarrow \hZ \overset{\xi}\longrightarrow Z 
\end{equation*}
where $\xi$ is a $\QQ$-factorialization and~$\hat f$ is a blowup of points not lying on~$K$-trivial curves.
Then $A_{\hZ}=\hM_1 + \hM_2$, where~$\hM_i$ is the strict transform of~$M_i$. 
Note that $\hat{f}^*\hM_1$ and $\hat{f}^*\hM_2$ are movable by Lemma~\ref{lem:nef-semiample}, 
hence $\hD = \hat{f}^* A_{\hZ}= \hat{f}^*\hM_1 + \hat{f}^*\hM_2$ is not an extremal ray of~$\Mov(\hX)$, a contradiction.
\end{proof}

\subsection{Del Pezzo varieties with~$\rr(X) = 2$}
\label{ss:r2}

Let~$X$ be a del Pezzo variety of dimension~$n\ge 3$ with~$\rr(X)=2$.
If~$X$ is not $\QQ$-factorial, 
then there exist exactly two 
$\QQ$-factorializations~$\xi_i \colon \hX_i \to X$, $i = 1,2$.
Indeed, in this case~$\Mov(X)$ is a two-dimensional cone,
and by Proposition~\ref{prop:cone} the class~$A_X$ is in its interior 
and at the same time in each chamber of the decomposition~\eqref{eq:Mov:chamber}, 
hence there are exactly two such chambers.
Furthermore, two $\QQ$-factorializations $\hX_1$ and $\hX_2$
are connected by a flop $\chi \colon \hX_1 \dashrightarrow \hX_2$.
If~$X$ is $\QQ$-factorial, we put~$\hX_1 = \hX_2 = X$ \and~$\chi = \operatorname{id}$.
In both cases we obtain a diagram:
\begin{equation}
\label{eq:diag:r=2}
\vcenter{
\xymatrix{
&
\hX_1 \ar[dl]_{f_1} \ar[dr]^{\xi_1} \ar@{-->}[rr]^{\chi} &&
\hX_2 \ar[dl]_{\xi_2} \ar[dr]^{f_2}
\\
Z_1 &&
X &&
Z_2
}}
\end{equation} 
where~$f_i$ are $K$-negative extremal contractions.
If~$f_i$ is birational we denote by~$E_i$ the corresponding exceptional divisor,
if it is a quadric bundle over~$\PP^1$ we denote by~$F_i$ the corresponding $\PP^1$-class, and
if it is a $\PP^{n-2}$-bundle over~$\PP^2$ we denote by~$G_i$ the corresponding $\PP^2$-class.

\begin{proposition}[cf. {\cite[Theorem~5.3]{P:GFano1}}, {\cite{Jahnke-Pet}}]
\label{prop:rho=2}
Let~$X$ be a 
del Pezzo variety of dimension $n\ge 3$ with $\rr(X)=2$.
Then in the above notation up to symmetry of the diagram~\eqref{eq:diag:r=2}
one has one of the following situations:
\smallskip
\begin{center}
\renewcommand\arraystretch{1.1}
\setlength{\tabcolsep}{0.6em}
\begin{tabular}{|c|c|p{0.25\textwidth}|p{0.25\textwidth}|p{0.2\textwidth}|}
\hline
$\dd(X)$ & $\Xi(X)$&$f_1$ & $f_2$ & $\hfill\Cl(X)\hfill$
\\[4pt]
\hline
$1$ & \type{A_7}
& $\PP^{n-2}$-bundle & $\PP^{n-2}$-bundle &
$G_1 + G_2 = 6A_X$
\\
$1$ & \type{D_7} 
& quadric bundle & quadric bundle & 
$F_1 + F_2 = 4A_X$
\\
$1$ & \type{E_7}
& birational & birational& 
$E_1 + E_2 = 2A_X$
\\
$2$ & \type{A_6}
& $\PP^{n-2}$-bundle & $\PP^{n-2}$-bundle &
$G_1 + G_2 = 3A_X$
\\
$2$ & \type{D_6}
& quadric bundle & quadric bundle & 
$F_1 + F_2 = 2A_X$
\\
$2$ & \type{E_6}
& birational & birational &
$E_1 + E_2 = A_X\hphantom{1}$
\\
$3$ & \type{A_5}
& $\PP^{n-2}$-bundle & $\PP^{n-2}$-bundle &
$G_1 + G_2 = 2A_X$
\\
$3$ & \type{D_5}
& quadric bundle & birational & 
$F_1 + E_2 = A_X\hphantom{1}$
\\
$4$ & \type{A_4}
& $\PP^{n-2}$-bundle & birational & 
$G_1 + E_2 = A_X\hphantom{1}$
\\
$4$ & \type{D_4}
& quadric bundle & quadric bundle & 
$F_1 + F_2 = A_X\hphantom{1}$
\\
$5$ & \type{A_3}
& $\PP^{n-2}$-bundle & quadric bundle &
$G_1 + F_2 = 6A_X$
\\
$6$ & \type{A_2}
& $\PP^{n-2}$-bundle & $\PP^{n-2}$-bundle & 
$G_1 + G_2 = A_X\hphantom{1}$
\\
$7$ & \type{A_1}
& $\PP^{n-2}$-bundle & birational & 
$2G_1 + E_2 = A_X\hphantom{1}$
\\
\hline
\end{tabular}
\end{center}
\smallskip
\end{proposition}

\begin{proof}
Assume~$d \coloneqq \dd(X) \le 6$.
By Proposition~\ref{propo:ext-rays} each~$f_i$ is either a birational contraction, or a morphism to~$\PP^1$, or a morphism to~$\PP^2$.
Let~$D_i$ denote the corresponding exceptional class (in the first case), 
or~$\PP^1$-class (in the second case), 
or~$\PP^2$-class (in the last case).
Since~$\rr(X) = 2$, there is a linear relation in~$\Cl(X)$ between~$D_1$, $D_2$, and~$A_X$:
\begin{equation*}
x_1D_1 + x_2D_2 = kA_X,
\end{equation*}
where~$x_1$, $x_2$, and~$k$ have no common divisors.
On the one hand, by Lemma~\ref{lem:contractions-class} 
for each~$i \in \{1,2\}$ the classes~$A_X$ and~$D_i$ generate~$\Cl(X)$, hence~$x_i = \pm 1$.
On the other hand, by Lemma~\ref{rem:extrmal:Eff} the classes~$D_1$ and~$D_2$ generate the effective cone of~$X$,
hence~$x_1$, $x_2$, and~$k$ have the same sign. 
Therefore, we may assume~$x_1 = x_2 = 1$ and~$k > 0$.

Taking the product of this relation with~$A_X$ we obtain
\begin{equation*}
s_1 + s_2 = kd,
\end{equation*}
where~$s_i = \langle A_X, D_i \rangle \in \{1,2,3\}$ and~$d = \dd(X)$.
In particular,~$s_1 + s_2$ is divisible by~$d$.
This immediately implies that we have the following possibilities (up to permutation of the~$s_i$) 
\begin{itemize}
\item 
for~$d = 6$ we have~$s_1 = s_2 = 3$ and~$k = 1$;
\item 
for~$d = 5$ we have~$s_1 = 3$, $s_2 = 2$, and~$k = 1$;
\item 
for~$d = 4$ we have~$s_1 = 3$, $s_2 = 1$, or~$s_1 = s_2 = 2$, and~$k = 1$;
\item 
for~$d = 3$ we have~$s_1 = 2$, $s_2 = 1$, and~$k = 1$, or~$s_1 = s_2 = 3$, and~$k = 2$.
\end{itemize}
Finally, when~$d \in \{1, 2\}$, the variety~$X$ is acted upon by an involution~$\tau$
(Bertini or Geiser, see Remark~\ref{rem:involutions})
which acts on~$\Eff(X)$ and swaps~$D_1$ and~$D_2$.
Indeed, in these cases the involution~$\tau$ acts non-trivially on~$\Cl(X)$, 
because the quotient~$X/\tau$ is isomorphic to~$\PP^n$ (if~$d = 2$) or~$\PP(1^n,2)$ (if~$d = 1$),
hence any~$\tau$-invariant divisor class on~$X$ is proportional to~$A_X$. Thus,
we have~$s_1 = s_2$, 
which gives the following possibilities
\begin{itemize}
\item 
for~$d = 2$ we have~$s_1 = s_2 = k \in \{1,2,3\}$, and
\item 
for~$d = 1$ we have~$s_1 = s_2 = k/2 \in \{1,2,3\}$.
\end{itemize}
This exhaust all the cases in the table (except for the last row).
Conversely, looking at Proposition~\ref{prop:bundle-p2} and Proposition~\ref{prop:quadric-over-p1} 
we see that all cases of types~$\rA_m$, $m \ge 2$, and~$\rD_m$ indeed appear,
and blowing up a del Pezzo variety~$X_0$ of degree~$3$ and~$2$ with~$\rr(X_0) = 1$ at a general point,
we construct the cases of types~$\rE_6$ and~$\rE_7$, respectively.

The last case~$d = 7$ is obvious, as~$X = \Bl_P(\PP^3)$.
\end{proof}

\begin{remark}
By Proposition~\ref{prop:cone} and Proposition~\ref{prop:rho=2} 
the effective cone~$\Eff(X)$ is generated by the summands in the last column of the table,
and to get the generators of the moving cone~$\Mov(X)$ one should replace the exceptional classes~$E_i$ 
by the big classes~$H_i = A_X + E_i$.
Note also that~$\Mov(X)$ is the union of two ample cones (with the wall generated by~$A_X$)
in all cases except for~$d \in \{6,7\}$, 
where~$A_X$ is ample and there is a single ample cone.
\end{remark}

In fact, in all these cases the diagram~\eqref{eq:diag:r=2} is easy to construct:
\begin{itemize}
\item 
when~$\dd(X) = 7$ and~$\Xi(X) = \rA_1$ the diagram is obvious;
\item 
when~$\dd(X) = 6$ and~$\Xi(X) = \rA_2$, the diagram is symmetric via the transposition involution of~$X = \PP^2 \times \PP^2$
or~$X = \Fl(1,2;3)$;
\item 
when~$\dd(X) = 5$ and~$\Xi(X) = \rA_3$ the explicit flop is constructed in Lemma~\ref{lem:schubert};
\item 
when~$\dd(X) = 4$ and~$\Xi(X) = \rA_4$ the explicit flop is constructed in~\cite[Theorem~2.2 and~\S2.3]{KP-Mu};
\item 
when~$\dd(X) = 4$ and~$\Xi(X) = \rD_4$, the diagram is symmetric via two Springer resolutions of the determinantal quadric~$X$
inside a fixed quadric hypersurface;
\item 
when~$\dd(X) = 3$ and~$\Xi(X) = \rA_5$, the diagram is symmetric via two Springer resolutions of the determinantal cubic~$X$;
\item 
when~$\dd(X) = 3$ and~$\Xi(X) = \rD_5$, the diagram can be constructed 
by projecting an intersection of two quadrics in~$\PP^{n+2}$ from a general point
to get a cubic hypersurface in~$\PP^{n+1}$ containing a~$\PP^{n-1}$,
and then projection out of this~$\PP^{n-1}$ to get a quadric bundle over~$\PP^1$;
\item 
when~$\dd(X) \in \{1,2\}$ in all types the diagram~\eqref{eq:diag:r=2} is symmetric via the Bertini or Geiser involution of~$X$.
\end{itemize}  

\subsection{Del Pezzo varieties with~$\rr(X) = 3$}
\label{ss:r3}

In this section we show the pictures of the effective and moving cones 
of all del Pezzo varieties~$X$ with~$\rr(X) = 3$ (see also Appendix~\ref{sec:details} for other material).
For each variety we show a slice of the cones, and write the type of the variety;
types~$\rA_m^*$, $m \in \{1,3,5\}$, stand for the primitive varieties 
that have a structure of~$\PP^{n-2}$-bundle over~$\PP^1 \times \PP^1$.

In the pictures of the cones listed below we use the following conventions.
We write~$E_i$ for exceptional classes, $F_i$ for $\PP^1$-classes, $G_i$ for $\PP^2$-classes, 
and~$H_i$ for big classes as in Proposition~\ref{prop:cone}\ref{it:eff-rays:m}.
The dark grey part of the picture is a slice of the moving cone~$\Mov(X)$, 
and its union with the light grey part is a slice of the effective cone~$\Eff(X)$.
We also draw the walls showing how the moving cone splits into the union of ample cones~$\Amp(X_i)$ 
of crepant $\QQ$-factorial models~$X_i$ of~$X$.

The computations that lead to these descriptions are analogous to those made before, so we omit them.

\def\graycc{10}
\def\grayc{30}
\def\sizec{0.1em}

\renewcommand\arraystretch{1.0}
\setlength{\tabcolsep}{0.9em}
\begin{longtable}{c|ccc}
\hline
$\dd(X) = 1$ & 

\begin{tikzpicture}[baseline=(current bounding box.center)] %%%%%%%%%%% d = 1, type = A_6
\path 
node[regular polygon, regular polygon sides=6, draw,fill=gray!\graycc, inner sep=1.7em, ] 
(hexagon) {} 
(hexagon.corner 1) node[right] {$\scriptstyle F_1$} 
(hexagon.corner 2) node[left] {$\scriptstyle E_1$} 
(hexagon.corner 3) node[left] {$\scriptstyle F_4$} 
(hexagon.corner 4) node[left] {$\scriptstyle F_3$} 
(hexagon.corner 5) node[right] {$\scriptstyle E_2$}
(hexagon.corner 6) node[right] {$\scriptstyle F_2$} 
plot[ mark=, samples at={1, ..., 6}, ] 
(hexagon.corner \x) ;
\coordinate (F1) at (hexagon.corner 1); 
\coordinate (E1) at (hexagon.corner 2);
\coordinate (F3) at (hexagon.corner 3);
\coordinate (F2) at (hexagon.corner 4);
\coordinate (E2) at (hexagon.corner 5);
\coordinate (F4) at (hexagon.corner 6);
\coordinate (G1) at (0,1); 
\coordinate (G2) at (0,-1); 
\coordinate (G3) at (-0.87,0.53); 
\coordinate (G4) at (0.87,-0.53); 
\coordinate (A) at (0,0);

\fill[black] (G1) circle (\sizec) node[above]{$\scriptstyle G_1$};
\fill[black] (G3) circle (\sizec) node[left]{$\scriptstyle G_3$};
\fill[black] (G4) circle (\sizec) node[right]{$\scriptstyle G_4$};
\fill[black] (G2) circle (\sizec) node[below]{$\scriptstyle G_2$}; 
\draw[name path=M, gray!\graycc] (E1)-- (E2);
\draw[name path=L1, gray!\graycc] (G1)-- (G3);
\draw[name path=L2, gray!\graycc] (G2)-- (G4);
\path [name intersections={of=L1 and M}]{};
\coordinate (Q1) at (intersection-1);
\path [name intersections={of=L2 and M}]{};
\coordinate (Q2) at (intersection-1);

\draw[fill=gray!\grayc] (G1)-- (G3) -- (F3) -- (F2) -- (G2) -- (G4) -- (F4) -- (F1) -- (G1); 

\draw[black] (Q2)--(Q1);
\foreach \n in {1,2,3,4}
{
\draw[black] (A)--(F\n);
\draw[black] (A)--(G\n);
\fill[black] (F\n) circle (\sizec);
\fill[black] (G\n) circle (\sizec);
}
\fill[black] (E1) circle (\sizec);
\fill[black] (E2) circle (\sizec);
\fill[black] (A) circle (\sizec) node[right,yshift=4]{$\scriptstyle A$}; 
\end{tikzpicture}
&
\begin{tikzpicture}[baseline=(current bounding box.center)] %%%%%%%%%%% d = 1, type = D_6
\path 
node[regular polygon, regular polygon sides=4, draw,fill=gray!\graycc, inner sep=1.7em, ] 
(hexagon) {} 
(hexagon.corner 1) node[right] {$\scriptstyle E_3$} 
(hexagon.corner 2) node[left] {$\scriptstyle E_1$} 
(hexagon.corner 3) node[left] {$\scriptstyle E_4$} 
(hexagon.corner 4) node[right] {$\scriptstyle E_2$} 
plot[ mark=, samples at={1, ..., 4}, ] 
(hexagon.corner \x) ;
\coordinate (F1) at (0.,1);
\coordinate (F2) at (0.,-1); 
\coordinate (F3) at (-1,0.);
\coordinate (F4) at (1,0.); 
\coordinate (A) at (0,0);
\coordinate (E1) at (hexagon.corner 2);
\coordinate (E2) at (hexagon.corner 4);
\coordinate (E3) at (hexagon.corner 1);
\coordinate (E4) at (hexagon.corner 3);
\foreach \n in {1,2,3,4}
{
\draw[name path=L\n, gray!\graycc] (A)-- (E\n);
}
\draw[name path=M1, gray!\graycc] (F1)-- (F3);
\draw[name path=M4, gray!\graycc] (F2)-- (F3);
\draw[name path=M3, gray!\graycc] (F1)-- (F4);
\draw[name path=M2, gray!\graycc] (F2)-- (F4);

\path [name intersections={of=L1 and M1}]{};
\coordinate (Q1) at (intersection-1);
\path [name intersections={of=L2 and M2}]{};
\coordinate (Q2) at (intersection-1);
\path [name intersections={of=L3 and M3}]{};
\coordinate (Q3) at (intersection-1);
\path [name intersections={of=L4 and M4}]{};
\coordinate (Q4) at (intersection-1);

\fill[black] (F1) circle (\sizec) node[above]{$\scriptstyle F_1$}; 
\fill[black] (F2) circle (\sizec) node[below]{$\scriptstyle F_2$}; 
\fill[black] (F3) circle (\sizec) node[left]{$\scriptstyle F_3$}; 
\fill[black] (F4) circle (\sizec) node[right]{$\scriptstyle F_4$}; 

\draw[fill=gray!\grayc] (F1) -- (F4) -- (F2) -- (F3)-- (F1) ; 

\foreach \n in {1,2,3,4}
{
\draw[black] (A)-- (F\n);
\draw[black] (A)-- (Q\n);
\fill[black] (E\n) circle (\sizec); 
}

\fill[black] (A) circle (\sizec) node[above,xshift=3,yshift=2]{$\scriptstyle A$}; 
\end{tikzpicture} 
&
\begin{tikzpicture}[baseline=(current bounding box.center)] %%%%%%%%%%% d = 1, type = E_6
\path 
node[regular polygon, regular polygon sides=6, draw,fill=gray!\graycc, inner sep=1.7em, ] 
(hexagon) {} 
(hexagon.corner 1) node[right] {$\scriptstyle E_1$} 
(hexagon.corner 2) node[left] {$\scriptstyle E_2$} 
(hexagon.corner 3) node[left] {$\scriptstyle E_3$} 
(hexagon.corner 4) node[left] {$\scriptstyle E_4$} 
(hexagon.corner 5) node[right] {$\scriptstyle E_5$}
(hexagon.corner 6) node[right] {$\scriptstyle E_6$} 
plot[ mark=, samples at={1, ..., 6}, ] 
(hexagon.corner \x) ; 

\coordinate (A) at (0,0);
\coordinate (E7) at (hexagon.corner 1); 

\coordinate (E1) at (hexagon.corner 1); 
\draw[name path=AE1, gray!\graycc] (A) -- (E1);
\coordinate (E2) at (hexagon.corner 2); 
\draw[name path=AE2, gray!\graycc] (A) -- (E2);
\coordinate (E3) at (hexagon.corner 3); 
\draw[name path=AE3, gray!\graycc] (A) -- (E3);
\coordinate (E4) at (hexagon.corner 4); 
\draw[name path=AE4, gray!\graycc] (A) -- (E4);
\coordinate (E5) at (hexagon.corner 5); 
\draw[name path=AE5, gray!\graycc] (A) -- (E5);
\coordinate (E6) at (hexagon.corner 6); 
\draw[name path=AE6, gray!\graycc] (A) -- (E6);

\path 
node[regular polygon, regular polygon sides=6, draw,black,fill=gray!\grayc, inner sep=0.9em, rotate=30 ] 
(hexagon) {}
(hexagon.corner 1) node[above,yshift=-0.3em] {$\scriptscriptstyle H_1$}
(hexagon.corner 2) node[left,xshift=0.3em] {$\scriptscriptstyle H_2$}
(hexagon.corner 3) node[left,xshift=0.3em] {$\scriptscriptstyle H_3$}
(hexagon.corner 4) node[below,yshift=0.2em] {$\scriptscriptstyle H_4$}
(hexagon.corner 5) node[right,xshift=-0.3em] {$\scriptscriptstyle H_5$}
(hexagon.corner 6) node[right,xshift=-0.3em] {$\scriptscriptstyle H_6$} 
plot[ mark=, samples at={1, ..., 6}, ] 
(hexagon.corner \x) ; 

\foreach \n in {1,2,3,4,5,6}
{
\coordinate (H\n) at (hexagon.corner \n);
\fill[black] (H\n) circle (\sizec) ; 
\fill[black] (E\n) circle (\sizec) ; 
}

\draw[name path=L1, gray!\grayc] (H1)-- (H2);
\draw[name path=L2, gray!\grayc] (H2)-- (H3);
\draw[name path=L3, gray!\grayc] (H3)-- (H4);
\draw[name path=L4, gray!\grayc] (H4)-- (H5);
\draw[name path=L5, gray!\grayc] (H5)-- (H6);
\draw[name path=L6, gray!\grayc] (H6)-- (H1);
\draw[name path=L7, gray!\grayc] (H1)-- (H2);

\path [name intersections={of=AE1 and L6}]{};
\coordinate (Q1) at (intersection-1);

\path [name intersections={of=AE2 and L1}]{};
\coordinate (Q2) at (intersection-1);

\path [name intersections={of=AE3 and L2}]{};
\coordinate (Q3) at (intersection-1);

\path [name intersections={of=AE4 and L3}]{};
\coordinate (Q4) at (intersection-1);

\path [name intersections={of=AE5 and L4}]{};
\coordinate (Q5) at (intersection-1);

\path [name intersections={of=AE6 and L5}]{};
\coordinate (Q6) at (intersection-1);
\foreach \n in {1,2,3,4,5,6}
{
\draw[black] (Q\n) -- (A) ;
}
\draw[black] (H1)--(H2)--(H3)--(H4)--(H5)--(H6)--(H1);
\fill[black] (A) circle (\sizec) node[above]{$\scriptstyle A$}; 
\end{tikzpicture}
\\
&

{$\Xi(X) = \rA_6$}
&

{$\Xi(X) = \rD_6$}
&

{$\Xi(X) = \rE_6$}
\\[1ex]\hline
$\dd(X) = 2$ 
&
\begin{tikzpicture}[baseline=(current bounding box.center)] %%%%%%%%%%% d = 2, type = A_5
\path 
node[regular polygon, regular polygon sides=4, draw,fill=gray!\graycc, inner sep=1.7em, ] 
(hexagon) {} 
(hexagon.corner 1) node[right] {$\scriptstyle E_1$} 
(hexagon.corner 2) node[left] {$\scriptstyle F_1$} 
(hexagon.corner 3) node[left] {$\scriptstyle E_2$} 
(hexagon.corner 4) node[right] {$\scriptstyle F_2$} 
plot[ mark=, samples at={1, ..., 4}, ] 
(hexagon.corner \x) ;
\coordinate (G3) at (1.02,0); 
\coordinate (G4) at (-1.02,0); 
\coordinate (G2) at (0,-1.02); 
\coordinate (G1) at (0,1.02); 
\coordinate (A) at (0,0); 
\coordinate (F1) at (hexagon.corner 2); 
\coordinate (F2) at (hexagon.corner 4);
\coordinate (E1) at (hexagon.corner 1); 
\coordinate (E2) at (hexagon.corner 3); 

\draw[name path=G1G3, gray!\graycc] (G1) -- (G3);
\draw[name path=G2G4, gray!\graycc] (G2) -- (G4);
\draw[name path=E1A, gray!\graycc] (E1) -- (A);
\draw[name path=E2A, gray!\graycc] (E2) -- (A);
\path [name intersections={of=E1A and G1G3}]{};
\coordinate (Q1) at (intersection-1);
\path [name intersections={of=E2A and G2G4}]{};
\coordinate (Q2) at (intersection-1);
\draw[fill=gray!\grayc] (F1) -- (G1) -- (G3) -- (F2)-- (G2)-- (G4) -- (F1);

\fill[black] (A) circle (\sizec) node[right,yshift=4,xshift=1]{$\scriptstyle A$}; 
\fill[black] (G1) circle (\sizec) node[above]{$\scriptstyle G_1$};
\fill[black] (G2) circle (\sizec) node[below]{$\scriptstyle G_2$};
\fill[black] (G3) circle (\sizec) node[right]{$\scriptstyle G_3$};
\fill[black] (G4) circle (\sizec) node[left]{$\scriptstyle G_4$};

\fill[black] (E1) circle (\sizec) ;
\fill[black] (E2) circle (\sizec);
\fill[black] (F1) circle (\sizec) ;
\fill[black] (F2) circle (\sizec) ;

\draw[black] (Q1) -- (A);
\draw[black] (Q2) -- (A);

\draw[black] (G1) -- (A);
\draw[black] (G2) -- (A);
\draw[black] (G3) -- (A);
\draw[black] (G4) -- (A);
\draw[black] (F1) -- (A);
\draw[black] (F2) -- (A);
\end{tikzpicture} 
&
\begin{tikzpicture}[baseline=(current bounding box.center)] %%%%%%%%%%% d = 2, type = D_5
\path 
node[regular polygon, regular polygon sides=4, draw,fill=gray!\graycc, inner sep=1.7em, ] 
(hexagon) {} 
(hexagon.corner 1) node[right] {$\scriptstyle E_1$} 
(hexagon.corner 2) node[left] {$\scriptstyle E_2$} 
(hexagon.corner 3) node[left] {$\scriptstyle E_3$} node[below] {\vphantom{$\scriptstyle G_1$}}
(hexagon.corner 4) node[right] {$\scriptstyle E_4$} 
plot[ mark=, samples at={1, ..., 4}, ] 
(hexagon.corner \x) ; 

\coordinate (A) at (0,0);

\foreach \n in {1,2,3,4}
{
\coordinate (E\n) at (hexagon.corner \n); 
}

\coordinate (H1) at (0,0.5); 
\coordinate (H2) at (0,-0.5); 
\coordinate (F1) at (-1.02,0); 
\coordinate (F2)at (1.02,0); 
\coordinate (E3) at (hexagon.corner 3); 
\coordinate (E4)at (hexagon.corner 4);

\draw[name path=L11, gray!\graycc] (F1) -- (H1);
\draw[name path=L12, gray!\graycc] (F1) -- (H2);
\draw[name path=L21, gray!\graycc] (F2)-- (H1);
\draw[name path=L22, gray!\graycc] (F2)-- (H2);

\draw[name path=AE1, gray!\graycc] (A) -- (E1);
\draw[name path=AE2, gray!\graycc] (A) -- (E2);
\draw[name path=AE3, gray!\graycc] (A)-- (E3);
\draw[name path=AE4, gray!\graycc] (A)-- (E4);

\path [name intersections={of=L22 and AE4}]{};
\coordinate (Q4) at (intersection-1);
\path [name intersections={of=L21 and AE1}]{};
\coordinate (Q1) at (intersection-1);
\path [name intersections={of=L12 and AE3}]{};
\coordinate (Q3) at (intersection-1);
\path [name intersections={of=L11 and AE2}]{};
\coordinate (Q2) at (intersection-1);

\fill[black] (H1) circle (\sizec) node[above]{$\scriptstyle H_1$};
\fill[black] (H2) circle (\sizec) node[below]{$\scriptstyle H_2$};
\fill[black] (F2) circle (\sizec) node[right]{$\scriptstyle F_2$};
\fill[black] (F1) circle (\sizec) node[left]{$\scriptstyle F_1$};

\draw[fill=gray!\grayc] (F1)-- (H1) -- (F2) -- (H2) -- (F1);
\draw[black] (A)-- (F1);
\draw[black] (A)-- (F2);
\draw[black] (A)-- (Q1);
\draw[black] (A)-- (Q2);
\draw[black] (A)-- (Q3);
\draw[black] (A)-- (Q4);
\foreach \n in {1,2,3,4}
{
\fill[black] (E\n) circle (\sizec);
}

\fill[black] (A) circle (\sizec) node[right,yshift=4,xshift=1]{$\scriptstyle A$}; 
\end{tikzpicture} 
&
\begin{tikzpicture}[baseline=(current bounding box.center)] %%%%%%%%%%% d = 2, type = A_5, extra
\path 
node[regular polygon, regular polygon sides=6, draw, fill=gray!\grayc, inner sep=1.7em, ] 
(hexagon) {} 
(hexagon.corner 1) node[right] {$\scriptstyle F_1$} 
(hexagon.corner 2) node[left] {$\scriptstyle F_2$} 
(hexagon.corner 3) node[left] {$\scriptstyle F_3$} 
(hexagon.corner 4) node[left] {$\scriptstyle F_4$} node[below] {\vphantom{$\scriptstyle G_1$}}
(hexagon.corner 5) node[right] {$\scriptstyle F_5$}
(hexagon.corner 6) node[right] {$\scriptstyle F_6$} 
plot[ mark=, samples at={1, ..., 6}, ] 
(hexagon.corner \x) ;

\coordinate (A) at (0,0);
\foreach \n in {1,2,3,4,5,6}
{
\coordinate (F\n) at (hexagon.corner \n); 
\draw[black] (F\n) -- (A);
\fill[black] (F\n) circle (\sizec) ; 
}

\fill[black] (0,0) circle (\sizec) node[above,yshift=2]{$\scriptstyle A$}; 
\end{tikzpicture}
\\
&
$\Xi(X) = \rA_5$
&
$\Xi(X) = \rD_5$
&
$\Xi(X) = \rA_5^*$
\\[1ex]\hline
$\dd(X) = 3$
&
\begin{tikzpicture}[baseline=(current bounding box.center)] %%%%%%%%%%% d = 3, type = A_4
\path 
node[regular polygon, regular polygon sides=4, draw,fill=gray!\graycc, inner sep=1.7em, ] 
(hexagon) {} 
(hexagon.corner 1) node[right] {$\scriptstyle E_2$} 
(hexagon.corner 2) node[left] {$\scriptstyle E_1$} 
(hexagon.corner 3) node[below] {$\scriptstyle F_2$} 
(hexagon.corner 4) node[below] {$\scriptstyle F_1$} 
plot[ mark=, samples at={1, ..., 4}, ] 
(hexagon.corner \x) ; 
\coordinate (H) at (0,0.4); 
\coordinate (G1) at (-1.0,0); 
\coordinate (G2) at (1.0,0); 
\coordinate (A) at (0,0); 
\coordinate (F1) at (hexagon.corner 4); 
\coordinate (F2) at (hexagon.corner 3); 
\coordinate (E2) at (hexagon.corner 1); 
\coordinate (E1) at (hexagon.corner 2); 
\draw[name path=L1, gray!\graycc] (A) -- (E1);
\draw[name path=L2, gray!\graycc] (A) -- (E2);
\draw[name path=M1, gray!\grayc] (G1)-- (H);
\draw[name path=M2, gray!\grayc] (G2)-- (H);
\path [name intersections={of=L1 and M1}]{};
\coordinate (Q1) at (intersection-1);
\path [name intersections={of=L2 and M2}]{};
\coordinate (Q2) at (intersection-1);

\fill[black] (H) circle (\sizec) node[above]{$\scriptstyle H$};
\fill[black] (G2) circle (\sizec) node[right]{$\scriptstyle G_2$};
\fill[black] (G1) circle (\sizec) node[left]{$\scriptstyle G_1$};
\draw[fill=gray!\grayc] (G1)-- (H)-- (G2) -- (F1) -- (F2) -- (G1);
\fill[black] (A) circle (\sizec) node[below]{$\scriptstyle A$}; 
\draw[black] (Q1)-- (A);
\draw[black] (Q2)-- (A);
\draw[black] (G1)-- (A);
\draw[black] (G2)-- (A);
\draw[black] (F1)-- (A);
\draw[black] (F2)-- (A);
\fill[black] (E1) circle (\sizec) node[below]{};
\fill[black] (E2) circle (\sizec) node[below]{};
\fill[black] (F1) circle (\sizec) node[below]{};
\fill[black] (F2) circle (\sizec) node[below]{};
\end{tikzpicture} 
&
\begin{tikzpicture}[baseline=(current bounding box.center)] %%%%%%%%%%% d = 3, type = D_4
\path
% \draw[path, name path=De1, thin] 
node[regular polygon, regular polygon sides=3, draw,fill=gray!\graycc, inner sep=1.1em, ] 
(hexagon) {} 
(hexagon.corner 1) node[right] {$\scriptstyle E_1$} 
(hexagon.corner 2) node[left] {$\scriptstyle E_2$} 
(hexagon.corner 3) node[right] {$\scriptstyle E_3$} 
plot[ mark=, samples at={1, ..., 3}, ] 
(hexagon.corner \x) ; 
\coordinate (E1) at (hexagon.corner 1); 
\coordinate (E2) at (hexagon.corner 2); 
\coordinate (E3) at (hexagon.corner 3); 

\coordinate (F1) at (0.55,0.38);
\coordinate (F2) at (-0.55,0.38); 
\coordinate (F3) at (0,-0.63); 
\coordinate (A) at (0,0); 
\draw[name path=E1F3, gray!\graycc] (F3) -- (E1);
\draw[name path=E2F1, gray!\graycc] (F1) -- (E2);
\draw[name path=E3F2, gray!\graycc] (F2) -- (E3);

\fill[black] (F1) circle (\sizec) node[right]{$\scriptstyle F_1$}; 
\fill[black] (F2) circle (\sizec) node[left]{$\scriptstyle F_2$}; 
\fill[black] (F3) circle (\sizec) node[below]{$\scriptstyle F_3$}; 

\draw[name path=F1F2,fill=gray!\grayc] (F1) -- (F2) ; 
\draw[name path=F2F3,fill=gray!\grayc] (F2) -- (F3); 
\draw[name path=F1F3,fill=gray!\grayc] (F3)-- (F1) ; 
\draw [fill=gray!\grayc] (F1) -- (F2) -- (F3) -- (F1) ;

\path [name intersections={of=E1F3 and F1F2}]{};
\coordinate (H1) at (intersection-1);
\draw[thick] (H1) -- (A) ;

\path [name intersections={of=E2F1 and F2F3}]{};
\coordinate (H2) at (intersection-1);
\draw[thick] (H2) -- (A) ;

\path [name intersections={of=E3F2 and F1F3}]{};
\coordinate (H3) at (intersection-1);
\draw[thick] (H3) -- (A) ;

\fill[black] (A) circle (\sizec) node[below]{$\scriptstyle A$}; 
\fill[black] (E1) circle (\sizec); 
\fill[black] (E2) circle (\sizec); 
\fill[black] (E3) circle (\sizec); 
\end{tikzpicture} 
\\
&

{$\Xi(X) = \rA_4$}
&

{$\Xi(X) = \rD_4$}
%%%%%%%%%%%%%%%%%%%%%%%%%%%%%%%%%%%%%%%%%%%%%%%%%%%%%%%%%%%%%%%%%%%%%%%%%%%%%%%%%%%%%%%%%%%%%%%%%%%%%%%%
\\[1ex]\hline
$\dd(X) = 4$
&
\begin{tikzpicture}[baseline=(current bounding box.center)] %%%%%%%%%%% d = 4, type = A_3
\path 
node[regular polygon, regular polygon sides=4, draw,fill=gray!\graycc, inner sep=1.7em, ] 
(hexagon) {} 
(hexagon.corner 1) node[right] {$\scriptstyle F_1$} 
(hexagon.corner 2) node[left] {$\scriptstyle E_1$} 
(hexagon.corner 3) node[left] {$\scriptstyle F_2$} 
(hexagon.corner 4) node[right] {$\scriptstyle E_2$} 
plot[ mark=, samples at={1, ..., 4}, ] 
(hexagon.corner \x) ;
\coordinate (G2) at (1.02,0); 
\coordinate (G1) at (0,1.02); 
\coordinate (F1) at (hexagon.corner 1); 
\coordinate (F2) at (hexagon.corner 3); 
\coordinate (E1) at (hexagon.corner 2); 
\coordinate (E2) at (hexagon.corner 4); 
\fill[black] (G2) circle (\sizec) node[right]{$\scriptstyle G_2$};
\fill[black] (G1) circle (\sizec) node[above]{$\scriptstyle G_1$};
\draw[fill=gray!\grayc] (G1) -- (F1)-- (G2)-- (F2) -- (G1);
\coordinate (A) at (0,0); 
\fill[black] (0,0) circle (\sizec) node[below]{$\scriptstyle A$}; 
\coordinate (FG2) at (0.34,-0.33); 
\coordinate (FG1) at (-0.38,0.28); 
\draw[] (A) -- (G1);
\draw[] (A) -- (G2);
\draw[] (A) -- (F1);
\draw[] (A) -- (F2);
\draw[] (A) -- (FG1);
\draw[] (A) -- (FG2);
\fill[black] (F1) circle (\sizec); 
\fill[black] (F2) circle (\sizec); 
\fill[black] (E1) circle (\sizec); 
\fill[black] (E2) circle (\sizec); 
\end{tikzpicture}
&
\begin{tikzpicture}[baseline=(current bounding box.center)] %%%%%%%%%%% d = 4, type = A_3, extra
\path 
node[regular polygon, regular polygon sides=4, draw,fill=gray!\grayc, inner sep=1.7em, ] 
(hexagon) {} 
(hexagon.corner 1) node[right] {$\scriptstyle F_1$} 
(hexagon.corner 2) node[left] {$\scriptstyle F_2$} 
(hexagon.corner 3) node[left] {$\scriptstyle F_3$} 
(hexagon.corner 4) node[right] {$\scriptstyle F_4$} 
plot[ mark=, samples at={1, ..., 4}, ] 
(hexagon.corner \x) ; 
\coordinate (F1) at (hexagon.corner 1); 
\coordinate (F2) at (hexagon.corner 2); 
\coordinate (F3) at (hexagon.corner 3); 
\coordinate (F4) at (hexagon.corner 4); 
\coordinate (A) at (0,0); 
\fill[black] (A) circle (\sizec) node[above]{$\scriptstyle A$}; 
\draw[] (A) -- (F1);
\draw[] (A) -- (F2);
\draw[] (A) -- (F3);
\draw[] (A) -- (F4);
\fill[black] (F1) circle (\sizec); 
\fill[black] (F2) circle (\sizec); 
\fill[black] (F3) circle (\sizec); 
\fill[black] (F4) circle (\sizec); 
\end{tikzpicture} 
\\
&
$\Xi(X) = \rA_3$
&
$\Xi(X) = \rA_3^*$
%%%%%%%%%%%%%%%%%%%%%%%%%%%%%%%%%%%%%%%%%%%%%%%%%%%%%%%%%%%%%%%%%%%%%%%%%%%%%%%%%%%%%%%%%%%%%%%%%%%%%%%%
\\[1ex]\hline
$\dd(X) = 5$
&
{
\begin{tikzpicture}[baseline=(current bounding box.center)] %%%%%%%%%%% d = 5, type = A_2
\path 
node[regular polygon, regular polygon sides=3, draw,fill=gray!\graycc, inner sep=1.1em, ] 
(hexagon) {} 
(hexagon.corner 1) node[right] {$\scriptstyle E_1$} 
(hexagon.corner 2) node[below] {$\scriptstyle F_1$} 
(hexagon.corner 3) node[below] {$\scriptstyle F_2$} 
plot[ mark=, samples at={1, ..., 3}, ] 
(hexagon.corner \x) ; 
\coordinate (G1) at (-0.53,0.38); 
\coordinate (G2) at (0.53,0.38); 
\coordinate (E1) at (hexagon.corner 1); 
\coordinate (F1) at (hexagon.corner 2); 
\coordinate (F2) at (hexagon.corner 3); 
\fill[black] (G1) circle (\sizec) node[left]{$\scriptstyle G_1$}; 
\fill[black] (G2) circle (\sizec) node[right]{$\scriptstyle G_2$}; 
\draw[fill=gray!\grayc] (G1)-- (G2) -- (F2) -- (F1) -- (G1); 
\draw[name path=L1, black] (G1) -- (F2);
\draw[name path=L2, black] (G2) -- (F1);
\path [name intersections={of=L1 and L2}]{};
\coordinate (A) at (intersection-1);
\fill[black] (A) circle (\sizec) node[below]{$\scriptstyle A$}; 
\fill[black] (F1) circle (\sizec); 
\fill[black] (F2) circle (\sizec); 
\fill[black] (E1) circle (\sizec); 
\end{tikzpicture} 
}
\\
&
{$\Xi(X) = \rA_2$}
%%%%%%%%%%%%%%%%%%%%%%%%%%%%%%%%%%%%%%%%%%%%%%%%%%%%%%%%%%%%%%%%%%%%%%%%%%%%%%%%%%%%%%%%%%%%%%%%%%%%%%%%
\\[1ex]\hline
$\dd(X) = 6$
&
\begin{tikzpicture}[baseline=(current bounding box.center)] %%%%%%%%%%% d = 6, type = A_1
\path 
node[regular polygon, regular polygon sides=3, draw,fill=gray!\graycc, inner sep=1.1em, ] 
(hexagon) {} 
(hexagon.corner 1) node[right] {$\scriptstyle F_1$} 
(hexagon.corner 2) node[left] {$\scriptstyle E_1$} 
(hexagon.corner 3) node[right] {$\scriptstyle E_2$} 
plot[ mark=, samples at={1, ..., 3}, ] 
(hexagon.corner \x) ;
\coordinate (F1) at (hexagon.corner 1) {};
\coordinate (E1) at (hexagon.corner 2) {};
\coordinate (E2) at (hexagon.corner 3) {};
\coordinate (G1) at (-0.53,0.38) {}; 
\coordinate (G2) at (0.53,0.38) {}; 
\coordinate (H1) (0, -1.3);
\draw[name path=L1, white] (G2) -- (E1);
\draw[name path=L2, white] (G1) -- (E2);
\path [name intersections={of=L1 and L2}]{};
\coordinate (H1) at (intersection-1);

\coordinate (F1) at (hexagon.corner 1);
\fill[black] (G1) circle (\sizec) node[left]{$\scriptstyle G_1$}; 
\fill[black] (G2) circle (\sizec) node[right]{$\scriptstyle G_2$}; 
\draw[fill=gray!\grayc] (G2)-- (F1) -- (G1) -- (H1) -- (G2); 
\fill[black] (0,0.38) circle (\sizec) node[above]{$\scriptstyle A$}; 
\draw[] (G1)-- (G2);
\fill[black] (F1) circle (\sizec); 
\fill[black] (F2) circle (\sizec); 
\fill[black] (E1) circle (\sizec); 
\draw[fill=gray!\graycc] (G1) -- (H1) -- (G2)-- (E2) --(E1) -- (G1) ; 
\fill[black] (H1) circle (\sizec) node[below]{$\scriptstyle H_1$}; 
\end{tikzpicture}
&
\begin{tikzpicture}[baseline=(current bounding box.center)] %%%%%%%%%%% d = 6, type = A_1, extra
\path 
node[regular polygon, regular polygon sides=3, draw,fill=gray!\grayc, inner sep=1.1em, ] 
(hexagon) {} 
(hexagon.corner 1) node[right] {$\scriptstyle F_1$} 
(hexagon.corner 2) node[left] {$\scriptstyle F_2$} 
(hexagon.corner 3) node[right] {$\scriptstyle F_3$} 
plot[ mark=, samples at={1, ..., 3}, ] 
(hexagon.corner \x) ; 
\coordinate (E1) at (hexagon.corner 1) {};
\coordinate (E2) at (hexagon.corner 2) {};
\coordinate (E3) at (hexagon.corner 3) {};
\fill[black] (0,0) circle (\sizec) node[above]{$\scriptstyle A$};
\fill[black] (E1) circle (\sizec); 
\fill[black] (E2) circle (\sizec); 
\fill[black] (E3) circle (\sizec); 
\end{tikzpicture}
\\
&
{$\Xi(X) = \rA_1$}
&
$\Xi(X) = \rA_1^*$
\\[1ex]\hline
\end{longtable} 

\bigskip 

\section{Schubert varieties}
\label{sec:schubert} 

Recall from Theorem~\ref{cla:dP}\ref{cla:dP5} that every non-conical del Pezzo variety of degree~$5$ 
is a linear section of~$\Gr(2,5)$; in this section we classify such varieties. 

Let~$V$ be a $5$-dimensional vector space.
Consider the Grassmannian~$\Gr(2,V) \subset \PP(\wedge^2 V)$ in the Pl\"ucker embedding.
Note that the group~$\PGL(V)$ acts on the space~$\PP(\wedge^2 V^\vee)$ of hyperplane sections of~$\Gr(2,V)$ with two orbits:
\begin{itemize}
\item 
The open orbit, formed by points of~$\PP(\wedge^2V^\vee)$ corresponding to skew-symmetric forms of rank~$4$;
the corresponding hyperplane sections of~$\Gr(2,V)$ are smooth.
\item 
The closed orbit, formed by points of~$\PP(\wedge^2V^\vee)$ corresponding to skew-symmetric forms of rank~$2$;
the corresponding hyperplane sections of~$\Gr(2,V)$ are singular.
\end{itemize}
More precisely, given a $3$-dimensional subspace~$K \subset V$, 
the subvariety of~$\Gr(2,V)$ parameterizing $2$-dimensional subspaces~$U \subset V$ such that~$U \cap K \ne 0$
is a singular hyperplane section of~$\Gr(2,V)$, called a \emph{Schubert divisor}.
Let
\begin{equation*}
X_{5,r,n} \subset \Gr(2,V),
\end{equation*}
be an intersections of~$r - 1$ general Schubert divisors and~$7-r-n$ general hyperplanes.
We prove the following

\begin{theorem}
\label{thm:quintic-detailed}
There are exactly~$10$ isomorphism classes of non-conical del Pezzo varieties 
of degree~$\dd(X) = 5$ and dimension~$\dim(X) = n \ge 3$, 
namely varieties
\begin{equation*}
X_{5,1,n},\ 3 \le n \le 6,
\qquad 
X_{5,2,n},\ 3 \le n \le 5,
\qquad 
X_{5,3,n},\ 3 \le n \le 4,
\qquad\text{and}\qquad 
X_{5,4,3}.
\end{equation*}
Moreover, $\rr(X_{5,r,n}) = r$ and
\begin{enumerate}
\item 
\label{it:dp51}
the varieties~$X_{5,1,n}$, $3 \le n \le 6$, are all smooth;
\item 
\label{it:dp52}
the varieties~$X_{5,2,n}$, $3 \le n \le 5$, have~$\Sing(X_{5,2,n}) = \PP^{n-3}$, 
and admit two small resolutions; one is
\begin{equation}
\label{eq:hx52n}
\hX_{5,2,n} \coloneqq \PP_{\PP^2}\left(\Ker(\cO(-1)^{\oplus 2} \oplus \Omega(1) \twoheadrightarrow \cO^{\oplus (5 - n)})\right)
\end{equation}
and the other is a flat quadric bundle over~$\PP^1$;
\item 
\label{it:dp53}
the varieties~$X_{5,3,n}$, $3 \le n \le 4$, have~$\Sing(X_{5,3,n}) = \PP^{n-3} \sqcup \PP^{n-3}$, 
and admit four small resolutions; one is
\begin{equation}
\label{eq:hx53n}
\hX_{5,3,4} \coloneqq \Bl_P\left(\PP^2 \times \PP^2\right)
\qquad\text{or}\qquad 
\hX_{5,3,3} \coloneqq \Bl_P(\Fl(1,2;3)),
\end{equation}
another is a $\PP^{n-2}$-bundle over~$\PP^1 \times \PP^1$ and the other two are $\PP^{n-2}$-bundles over~$\FF_1$;
\item 
\label{it:dp54}
the variety~$X_{5,4,3}$, has~$3$ singular points
and admits eight small resolutions; one of them is
\begin{equation}
\label{eq:hx543}
\hX_{5,4,3} \coloneqq \Bl_P\left(\PP^1 \times \PP^1 \times \PP^1\right)
\end{equation}
another is~$\Bl_{P_1,P_2,P_3}(\PP^3)$,
three others are blowups at a point of~$\PP^{1}$-bundles over~$\FF_1$, 
and the last three are $\PP^{1}$-bundles over the del Pezzo surface of degree~$7$.
\end{enumerate}
\end{theorem}

We split the proof in three steps.
We start with a discussion of linear sections of~$\Gr(2,V)$.
If~$W \subset \wedge^2V$ is a vector subspace we denote by
\begin{equation*}
X_W \coloneqq \Gr(2,V) \cap \PP(W)
\end{equation*}
the corresponding linear section.
We consider the space of linear equations of~$X_W$, 
i.e., the annihilator~$W^\perp \subset \wedge^2V^\vee$ of~$W$.
Each vector in~$W^\perp$ is a skew-symmetric form on~$V$;
if this form is degenerate, the corresponding hyperplane section is a Schubert divisor.
Note that if~$\dim(X_W) \ge 3$ then~$\dim(W^\perp) \le 3$.

\begin{lemma}
\label{lem:length-dual}
If~$X_W$ is a del Pezzo variety and~$n = \dim(X_W) \ge 3$ then the scheme
\begin{equation*}
X_W^\natural \coloneqq \Gr\left(2,V^\vee\right) \cap \PP(W^\perp)
\end{equation*}
is a finite reduced scheme of length~$\ell \coloneqq \ell(X_W^\natural) \le \dim(W^\perp)$. 
Moreover, if~$K_1,\dots,K_\ell \subset V$ are the $3$-dimensional subspaces 
corresponding to the points~$\lambda_i \in X_W^\natural$ then
\begin{equation}
\label{eq:ki-kj}
\dim(K_i \cap K_j) = 1
\quad\text{for all~$1 \le i < j \le \ell$}
\quad\text{and}\quad 
K_1 \cap K_2 \cap K_3 = 0
\quad\text{if~$\ell = 3$}.
\end{equation} 
Finally, if~$\ell(X_W^\natural) = \dim(W^\perp)$ then~$W^\perp$ is spanned by~$\lambda_1,\dots,\lambda_\ell$.
\end{lemma}

\begin{proof}
If~$X_W$ is a del Pezzo variety then a general linear surface section of~$X$ is a del Pezzo surface,
i.e., for general~$W_0 \subset W$, $\dim(W_0) = 6$, the linear section~$X_{W_0}$ is smooth.
Then by~\cite[Proposition~2.24]{DK18} the scheme~$X_{W_0}^\natural$ is a reduced scheme of length~$5$.
But
\begin{equation*}
X_W^\natural = \Gr(2,V^\vee) \cap \PP(W^\perp) = X_{W_0}^\natural \cap \PP(W^\perp), 
\end{equation*}
hence it is also finite and reduced.

On the other hand, $X_W^\natural \subset \PP(W^\perp)$ is an intersection of quadrics (because~$\Gr(2,V^\vee)$ is),
therefore its length is bounded by~$1$ if~$\dim(W^\perp) = 1$, by~$2$ if~$\dim(W^\perp) = 2$, and by~$4$ if~$\dim(W^\perp) = 3$.
It remains to note that the length cannot be equal to~$4$ by~\cite[Appendix~A]{KP-rF};
indeed the standard resolution of~$\Gr(2,V^\vee) \cong \Gr(2,5)$ shows 
that it satisfies the property~$\mathbf{N}_2$ of Green and Lazarsfeld, 
hence it is not tetragonal (see~\cite[Definition~A.1 and Proposition~A.4]{KP-rF}), 
hence~$\ell(X_W^\natural) \le 3$.

If~$\dim(K_i \cap K_j) = 2$, the line spanned by~$\lambda_i$ and~$\lambda_j$ lies in~$\Gr(2,V^\vee)$,
hence in~$X_W^\natural$.
Similarly, if~$L \coloneqq K_1 \cap K_2 \cap K_3 \ne 0$, 
then~$\lambda_1,\lambda_2,\lambda_3 \in \Gr(2,L^\perp) \subset \Gr(2,V^\vee)$,
hence~$X_W^\natural$ is a conic.
In both cases this contradicts to finiteness of~$X_W^\natural$.
Finally, if~$\ell(X_W^\natural) = \dim(W^\perp)$, 
but~$\lambda_i$ do not generate~$W^\perp$ 
then~$\ell(X_W^\natural) = 3$ and the points~$\lambda_i$ are collinear.
But since~$\Gr(2,V^\vee)$ is an intersection of quadrics, it follows that~$X_W^\natural$ contains a line,
which is again impossible.
\end{proof}

The next step is the classification of varieties for which~$\ell(X_W^\natural) = \dim(W^\perp)$;
these are, in fact, maximal del Pezzo varieties of degree~$5$.
Recall the varieties~$\hX_{5,2,5}$, $\hX_{5,3,4}$, and~$\hX_{5,4,3}$ 
defined in~\eqref{eq:hx52n}, \eqref{eq:hx53n}, and~\eqref{eq:hx543}.

\begin{proposition}
\label{prop:schubert-maximal}
For each~$\ell \le 3$ there is a unique isomorphism class of quintic del Pezzo varieties~$X_W \subset \Gr(2,V)$ 
with~$\ell(X_W^\natural) = \dim(W^\perp) = \ell$.
Moreover, 
\begin{itemize}
\item 
if~$\ell = 0$ then $X_W = \Gr(2,V)$ is smooth and~$\rr(X_W) = 1$;
\item 
if~$\ell = 1$ then~$\hX_{5,2,5}$
is a small resolution of~$X_W$, hence~$\rr(X_W) = 2$;
\item 
if~$\ell = 2$ then~$\hX_{5,3,4}$
is a small resolution of~$X_W$, hence~$\rr(X_W) = 3$;
\item 
if~$\ell = 3$ then~$\hX_{5,4,3}$
is a small resolution of~$X_W$, hence~$\rr(X_W) = 4$.
\end{itemize}
In all these cases~$\Sing(X_W)$ is the union of~$\ell$ disjoint projective spaces~$\PP^{3-\ell}$.
\end{proposition}

\begin{proof}
The uniqueness of~$X_W$ is easy:
in each case the space~$W^\perp$ is generated by~$\ell$ degenerate skew-symmetric forms~$\lambda_i$, 
the corresponding subspaces~$K_i$ satisfy~\eqref{eq:ki-kj} by Lemma~\ref{lem:length-dual}, 
and there is a single $\GL(V)$-orbit
on $\ell$-tuples of such subspaces for~$\ell \le 3$.
Therefore, we just need to show that the varieties~\eqref{eq:hx52n} with~$n = 5$, 
\eqref{eq:hx53n} with~$n = 4$, and~\eqref{eq:hx543}
are almost del Pezzo and that their anticanonical models have~$\ell = 1$, $\ell = 2$, and~$\ell = 3$, respectively.

On the other hand, we checked in Lemma~\ref{lem:three-schuberts} that~$\hX_{5,2,5}$, $\hX_{5,3,4}$, and~$\hX_{5,4,3}$
are almost del Pezzo and their anticanonical models~$X_{5,2,5}$, $X_{5,3,4}$, and~$X_{5,4,3}$ 
can be written as linear sections~$X_W$ of~$\Gr(2,V)$ with~$\ell(X_W^\natural) \ge \ell$, 
where~$\ell = 1$, $2$, and~$3$, respectively.
Applying Lemma~\ref{lem:length-dual} we obtain~$\ell(X_W^\natural) = \ell$.

It remains to describe the singular locus of~$X_{5,2,5}$, $X_{5,3,4}$, and~$X_{5,4,3}$.
Since the varieties~$\hX_{5,2,5}$, $\hX_{5,3,4}$, and~$\hX_{5,4,3}$ are smooth, 
this singular locus is the image of the locus contracted by~$\xi \colon \hX_{5,r,7-r} \to X_{5,r,7-r}$.
To understand this, we recall the descriptions of~$\xi$ provided by Lemma~\ref{lem:three-schuberts}.

When~$r = 2$ the locus contracted by~$\xi$ is the union 
of the $\PP^1$-fibers of~$\Fl(1,2;V)$ over the points of~$\Gr(2,K) \subset \Gr(2,V)$,
hence~$\Sing(X_{5,2,5}) = \Gr(2,K) \cong \PP^2$.

When~$r = 3$ the locus contracted by~$\xi$ is the union of strict transforms 
of two coordinate planes of~$\PP(K_1) \times \PP(K_2)$ through~$P$, 
hence~$\Sing(X_{5,3,4})$ is the union of two disjoint lines.

Finally, when~$r = 4$ the locus contracted by~$\xi$ is the union of strict transforms 
of the three coordinate axes of~$\PP(K_1^\perp) \times \PP(K_2^\perp) \times \PP(K_3^\perp)$ through~$P$, 
hence~$\Sing(X_{5,4,3})$ is the union of three points.
\end{proof}

Now we are ready for the last step.

\begin{proof}[Proof of Theorem~\xref{thm:quintic-detailed}]
In the case where~$\ell(X_W^\natural) = \dim(W^\perp)$ we have already proved the uniqueness
and constructed one small resolution (see Proposition~\ref{prop:schubert-maximal}).
To show that the other small resolution for~$X_{5,2,5}$ is a quadric bundle 
we apply Lemma~\ref{lem:schubert} or Proposition~\ref{prop:rho=2}.
To construct the other small resolutions for~$X_{5,3,4}$ 
we apply Lemma~\ref{lem:dpb-converse} (it gives~$\PP^2$-bundles over~$\FF_1$)
and~\cite[Theorem~3.1 and~\S4.1]{KP:rho} (it gives a $\PP^2$-bundle over~$\PP^1 \times \PP^1$).
Similarly, to construct the other small resolutions for~$X_{5,4,3}$ 
we first apply~\cite[Theorem~3.1]{KP:rho} (it gives~$\Bl_{P_1,P_2,P_3}(\PP^3)$)
and then apply Lemma~\ref{lem:dpb-converse} (it gives blowups of~$\PP^1$-bundles over~$\FF_1$ 
and $\PP^1$-bundles over the del Pezzo surface of degree~$7$).

Now we prove the uniqueness in the cases where~$\ell \coloneqq \ell(X_W^\natural) < \dim(W^\perp) \le 3$.
In the case~$\ell = 0$, the subspace~$W^\perp \subset \wedge^2V^\vee$ consists of forms of corank~$1$,
hence it is unique up to $\GL(V)$-action by~\cite[\S\S5--7]{PVDV}, so assume~$\ell \in \{1,2\}$.

If~$\ell = 1$ let~$W_0 \subset \wedge^2V$ be the hyperplane 
corresponding to the unique point~$\lambda \in X_W^\natural$ 
and let~$K \subset V$ be the corresponding $3$-dimensional subspace.
Then~$X_W$ is a linear section of~$X_{W_0} = X_{5,2,5}$ of codimension~$5 - n$
and~$\dim(\Gr(2,K) \cap \PP(W)) = 2 - (5 - n) = n - 3$,
since otherwise there is a point~$\lambda' \ne \lambda$ in~$\PP(W^\perp)$ 
such that~$K$ is isotropic for~$\lambda'$, 
and then~$X_W^\natural$ is not reduced at~$\lambda$.
Recall the morphism~$\xi \colon \hX_{5,2,5} \to X_{5,2,5}$ from Lemma~\ref{lem:three-schuberts}
and note that the above observation shows that the induced morphism
\begin{equation*}
\xi \colon \hX_W \coloneqq \xi^{-1}(X_W) \to X_W
\end{equation*}
is small.
On the other hand, all fibers of the projection~$\hX_W \subset \hX_{5,2,5} \to \PP(K)$ are linear spaces,
and if the fiber over a point~$v \in \PP(K)$ has dimension greater than~$n - 2$, 
then there is a point~$\lambda' \ne \lambda$ in~$\PP(W^\perp)$ such that~$v \in \Ker(\lambda')$,
and then the line spanned by~$\lambda$ and~$\lambda'$ is contained in~$\Gr(2,v^\perp)$,
which implies~$\ell(X_W^\natural) \ge 2$.
Therefore, $\hX_W \cong \PP_{\PP(K)}(\cE)$, where~$\cE$ is a vector bundle of rank~$n - 1$ that fits into exact sequence
\begin{equation*}
0 \longrightarrow \cE \longrightarrow \cO(-1)^{\oplus 2} \oplus \Omega(1) \xrightarrow{\ \varphi\ } \cO^{\oplus (5 - n)} \longrightarrow 0
\end{equation*}
with surjective~$\varphi$, so that~$\hX_W = \hX_{5,2,n}$, see~\eqref{eq:hx52n}.
Note also that~$\varphi\vert_{\Omega(1)} \colon \Omega(1) \to \cO^{\oplus (5 - n)}$ 
is generically surjective (because~$\dim(\Gr(2,K) \cap \PP(W)) = n - 3$).
It remains to show that the projective bundle~$\PP_{\PP(K)}(\cE)$ corresponding to such~$\varphi$
is unique up to isomorphism. 

First, assume~$n = 4$.
Since the morphism~$\varphi\vert_{\Omega(1)} \colon \Omega(1) \to \cO$ is nontrivial, 
its kernel is isomorphic to~$\cO(-1)$, and its cokernel is isomorphic to the structure sheaf of a point~$z \in \PP^2$.
Therefore, we have an exact sequence
\begin{equation*}
0 \longrightarrow \cO(-1) \longrightarrow \cE \longrightarrow \cO(-1)^{\oplus 2} \longrightarrow \cO_z \longrightarrow 0.
\end{equation*}
Now~$\PP^2$ is homogeneous, hence the position of~$z$ does not matter.
Further, there is a single isomorphism class of epimorphisms~$\cO(-1)^{\oplus 2} \to \cO_z$,
the kernels of these epimorphisms are isomorphic to~$\cO(-1) \oplus \cI_z(-1)$, and
\begin{equation*}
\Ext^1(\cO(-1) \oplus \cI_z(-1), \cO(-1)) \cong \Ext^1(\cI_z(-1), \cO(-1)) 
\end{equation*}
is $1$-dimensional, 
hence the projective bundle~$\PP_{\PP(K)}(\cE)$ is unique.

Next, assume~$n = 3$.
The morphism~$\varphi\vert_{\Omega(1)} \colon \Omega(1) \to \cO^{\oplus 2}$ is injective
and its cokernel is isomorphic to the sheaf~$\cO_L(1)$, where~$L \subset \PP^2$ is a line.
Therefore, we have an exact sequence
\begin{equation*}
0 \longrightarrow \cE \longrightarrow \cO(-1)^{\oplus 2} \longrightarrow \cO_L(1) \longrightarrow 0.
\end{equation*}
The position of~$L$ also does not matter, and there is a single isomorphism class 
of epimorphisms~$\cO(-1)^{\oplus 2} \to \cO_L(1)$,
hence again~$\PP_{\PP(K)}(\cE)$ is unique.

Now let~$\ell = 2$ and~$\dim(W^\perp) = 3$.
Let~$W_0 \subset \wedge^2V$ be the subspace of codimension~$2$ 
such that~$W_0^\perp$ is spanned by the two points~$\lambda_1,\lambda_2 \in X_W^\natural$ 
and let~$K_1,K_2 \subset V$ be the corresponding $3$-dimensional subspaces.
Then~$X_W$ is a linear section of~\mbox{$X_{W_0} = X_{5,3,4}$} 
and~$\dim(\Gr(2,K_i) \cap \PP(W)) = 0$ by the same reason as before.
Recall from Lemma~\ref{lem:three-schuberts} the morphism~$\xi \colon \hX_{5,3,4} \to X_{5,3,4}$ 
and note that the induced morphism
\begin{equation*}
\xi \colon \hX_W \coloneqq \xi^{-1}(X_W) \longrightarrow X_W
\end{equation*}
is small.
It also follows that the intersection of~$\hX_W \subset \hX_{5,3,4}$ 
with the exceptional divisor of~$\hX_{5,3,4} = \Bl_P(\PP(K_1) \times \PP(K_2))$ is a plane,
hence the image~$\sigma(\hX_W) \subset \PP(K_1) \times \PP(K_2)$ is a hyperplane section smooth at~$P$ 
(where~$\sigma \colon \hX_{5,3,4} \to \PP(K_1) \times \PP(K_2)$ is the blowup).
Therefore, $\rr(\sigma(\hX_W)) = \rr(\hX_W) - 1 = 2$, 
hence~$\sigma(\hX_W) \cong \Fl(1,2;3)$.
\end{proof}

In the rest of the section we provide a few more details about Schubert varieties.

\begin{remark}
The class group of the variety~$X = X_{5,4,3}$ has~$4$ exceptional classes:
the exceptional class~$E_0$ that comes from the small resolution~$\hX_{5,4,3} = \Bl_P(\PP^1 \times \PP^1 \times \PP^1)$,
and three pairwise orthogonal exceptional classes~$E_1$, $E_2$, $E_3$, 
that come from the small resolution~$\Bl_{P_1,P_2,P_3}(\PP^3)$.
It is easy to see that~$E_i$, $0 \le i \le 3$, generate~$\Cl(X)$, and that
\begin{equation*}
\langle E_0, E_i \rangle = 1,
\quad 
\langle E_i, E_j \rangle = 0,
\quad
1 \le i < j \le 3,
\quad\text{and}\quad
A_X = 2E_0 + E_1 + E_2 + E_3.
\end{equation*}
Moreover, $E_i$, $0 \le i \le 3$, are the extremal rays of the effective cone~$\Eff(X)$,
while the moving cone~$\Mov(X)$ is generated by the~$7$ classes
\begin{equation*}
F_i \coloneqq E_0 + E_i,
\quad
G_{i,j} \coloneqq E_0 + E_i + E_j,
\quad\text{and}\quad
H \coloneqq E_0 + E_1 + E_2 + E_3,
\quad 
1 \le i < j \le 3.
\end{equation*}
\end{remark}

To conclude this section we describe explicitly the flop 
associated with the small resolution~$\xi \colon \hX_{5,2,5} \to X_{5,2,5}$
(cf.~\cite[\S3(C)]{Jahnke-Pet}).

\begin{lemma}
\label{lem:schubert-flop-r2}
If\/~$V$ is a $5$-dimensional vector space and~$K \subset V$ is a $3$-dimensional subspace,
there is a diagram
\begin{equation*}
\xymatrix@C=0em{
& 
\PP_{\PP(K)}(\Omega(1) \oplus (V/K) \otimes \cO(-1)) \ar[dl]_f \ar[dr]^\xi \ar@{-->}[rr]^\chi &&
\Gr_{\PP(V/K)}(2, K \otimes \cO \oplus \cO(-1)) \ar[dl]_{\xi^+} \ar[dr]^{f^+}
\\
\PP(K) &&
X_{5,2,5} &&
\PP(V/K),
}
\end{equation*}
where~$f$ is a $\PP^3$-bundle, $f^+$ is an everywhere non-degenerate $Q^4$-bundle,
$\xi$ and~$\xi^+$ are small resolutions and~$\chi$ is a flop,
with the flopping locus~$\PP_{\PP(K)}(\Omega(1)) \cong \Fl(1,2;K)$
and the flopped locus~$\Gr_{\PP(V/K)}(2,K\otimes \cO) \cong \PP(V/K) \times \Gr(2,K) \cong \PP^1 \times \PP^2$.
\end{lemma}

\begin{proof}
Consider the standard flip diagram
\begin{equation*}
\xymatrix@C=-5.3em{
&
\PP_{\PP(K) \times \PP(V/K)}\big(\wedge^2K \otimes \cO \oplus \cO(-1,-1)\big) \ar[dl] \ar[dr]
\\
\PP_{\PP(K)}\big(\wedge^2K \otimes \cO \oplus (V/K) \otimes \cO(-1)\big) \ar[d] \ar[dr] &&
\PP_{\PP(V/K)}\big(\wedge^2K \otimes \cO \oplus K \otimes \cO(-1)\big) \ar[d] \ar[dl] 
\\
\PP(K) &
\rC\big(\PP(K) \times \PP(V/K)\big) &
\PP(V/K)
}
\end{equation*}
where~$\rC\big(\PP(K) \times \PP(V/K)\big) \subset \PP\big(\wedge^2K \oplus K \otimes (V/K)\big)$ 
is the cone with vertex~$\PP(\wedge^2K)$.
The flipping locus is~$\PP(K) \times \PP(\wedge^2K)$ in the left side and~$\PP(V/K) \times \PP(\wedge^2K)$ in the right side,
and the common exceptional divisor in~$M \coloneqq \PP_{\PP(K) \times \PP(V/K)}\big(\wedge^2K \otimes \cO \oplus \cO(-1,-1)\big)$
is~$E = \PP(K) \times \PP(V/K) \times \PP(\wedge^2K)$.

Denote by~$F$ and~$G$ the pullbacks of the hyperplane classes of~$\PP(V/K)$ and~$\PP(K)$, 
and by~$A$ the hyperplane class of~$\PP(\wedge^2K \oplus K \otimes (V/K))$ 
and its pullbacks to all varieties above.
Then on~$M$ we have a relation
\begin{equation*}
A = E + F + G.
\end{equation*}
Now we note that~$\PP_{\PP(K)}(\Omega(1) \oplus (V/K) \otimes \cO(-1)) \subset 
\PP_{\PP(K)}(\wedge^2K \otimes \cO \oplus (V/K) \otimes \cO(-1))$ 
and its preimage in~$M$ is a divisor of class
\begin{equation*}
A + G = 2A - F - E,
\end{equation*}
hence it is equal to the strict transform of a divisor
in~$\PP_{\PP(V/K)}(\wedge^2K \otimes \cO \oplus K \otimes \cO(-1))$ of class~$2A - F$ 
containing~$\PP(V/K) \times \PP(\wedge^2K)$.
Finally, it is easy to see that this divisor coincides with the relative Grassmannian as in the statement of the lemma.
\end{proof}

\appendix

\section{Varieties of type~$\rA_m$}
\label{sec:examples}

In this section we list varieties of type~$\rA_m$, providing when possible their explicit equations.
We consider only non-conical varieties, and in most cases, only maximal ones.
Throughout the section we use notation~$X_{d,r,n}$ for a del Pezzo variety~$X$ of type~$\rA$
with~$\dd(X) = d$, $\rr(X) = r$, and~$\dim(X) = n$
and~$X^*_{d,3,n}$, $d \in \{2,4,6\}$, for the extra primitive del Pezzo varieties of type~$\rA$
(those that have the structure of~$\PP^{n-2}$-bundle over~$\PP^1 \times \PP^1$).

\subsection{Degree~$5$, $6$, $7$ and~$8$}

By Theorem~\ref{thm:intro-bircla} all del Pezzo varieties of degree~$\dd(X) \ge 5$ have type~$\rA_m$.
Moreover, if~$\dd(X) \ge 7$, there are only two isomorphism classes,
if~\mbox{$\dd(X) = 6$}, there are exactly four isomorphism classes,
and, if~$\dd(X) = 5$ there are exactly ten isomorphism classes (four maximal varieties and six non-maximal),
see~\S\ref{sec:schubert}.

\subsubsection{$\dd(X) = 8$}

In this case~$\rr(X) = 1$, $\Xi(X) = \rA_1$, and
\begin{equation*}
X_{8,1,3} \cong \PP^3.
\end{equation*}

\subsubsection{$\dd(X) = 7$}

In this case~$\rr(X) = 2$, $\Xi(X) = \rA_1$, and
\begin{equation*}
X_{7,2,3} \cong \Bl_P(\PP^3) \cong \PP_{\PP^2}\big(\cO(-1) \oplus \cO(-2)\big).
\end{equation*} 

\subsubsection{$\dd(X) = 6$ and~$\rr(X) = 2$}

In this case $\Xi(X) = \rA_2$ and either~$\dim(X) = 4$ and
\begin{equation*}
X_{6,2,4} \cong \PP^2 \times \PP^2,
\end{equation*}
or~$\dim(X) = 3$ (in this case the variety is {\bf non-maximal}) and
\begin{equation*}
X_{6,2,3} = \Fl(1,2;3) = \{x_1y_1 + x_2y_2 +x_3y_3 = 0\} \subset \PP^2 \times \PP^2.
\end{equation*} 

\subsubsection{$\dd(X) = 6$ and~$\rr(X) = 3$}

In this case~$\dim(X) = 3$, $\Xi(X) = \rA_1$, and there are two varieties, the first is:
\begin{equation*}
X_{6,3,3} = \{x_1y_1 + x_2y_2 = 0\} \subset \PP^2 \times \PP^2; 
\end{equation*}
it has a single node and two small resolutions (see Lemma~\ref{lem:special-sextic}),
and the second is:
\begin{equation*}
X^*_{6,3,3} = \PP^1 \times \PP^1 \times \PP^1.
\end{equation*} 

\subsubsection{$\dd(X) = 5$ and~$\rr(X) = 1$}

In this case~$\dim(X) \le 6$, $\Xi(X) = \rA_4$, and there is one maximal variety
\begin{equation*}
X_{5,1,6} \cong \Gr(2,5);
\end{equation*}
and three {\bf non-maximal} (smooth linear sections of~$X_{5,1,6}$) varieties:
\begin{align*}
X_{5,1,5} &= \{x_{12} + x_{34} = 0\} \subset \Gr(2,5),
\\
X_{5,1,4} &= \{x_{12} + x_{34} = x_{23} + x_{45} = 0\} \subset \Gr(2,5),
\\
X_{5,1,3} &= \{x_{12} + x_{34} = x_{23} + x_{45} = x_{15}+x_{24} = 0\} \subset \Gr(2,5),
\end{align*}
where~$x_{ij}$, $1 \le i < j \le 5$, are the Pl\"ucker coordinates.

\subsubsection{$\dd(X) = 5$ and~$\rr(X) = 2$}

In this case~$\dim(X) \le 5$, $\Xi(X) = \rA_3$, and there is one maximal \emph{Schubert variety}
\begin{equation*}
X_{5,2,5} = \{x_{12} = 0\} \subset \Gr(2,5)
\end{equation*}
This variety has nodal singularities along the embedded plane~$\PP^2 = \langle e_{34}, e_{35}, e_{45} \rangle$.

Besides, there are two {\bf non-maximal} (general linear sections of~$X_{5,2,5}$) varieties:
\begin{align*}
X_{5,2,4} &= \{x_{12} = x_{23} + x_{45} = 0\} \subset \Gr(2,5),
\\
X_{5,2,3} &= \{x_{12} = x_{23} + x_{45} = x_{13}+x_{24}+x_{35} = 0\} \subset \Gr(2,5).
\end{align*}
These varieties have nodal singularities 
along the line~$\langle e_{34}, e_{35} \rangle$ and point~$[e_{34}]$, respectively.  

\subsubsection{$\dd(X) = 5$ and~$\rr(X) = 3$}

In this case~$\dim(X) \le 4$, $\Xi(X) = \rA_2$, there is one maximal \emph{double Schubert variety}
\begin{equation*}
X_{5,3,4} = \{x_{12} = x_{34} = 0\} \subset \Gr(2,5),
\end{equation*}
an intersection of two Schubert divisors.
This variety has nodal singularities along two lines~$\langle e_{35}, e_{45} \rangle$ and~$\langle e_{15}, e_{25} \rangle$.

Besides, there is a {\bf non-maximal} (general hyperplane section of~$X_{5,3,4}$) variety:
\begin{equation*}
X_{5,3,3} = \{x_{12} = x_{34} = x_{15}+x_{35}+x_{24} = 0\} \subset \Gr(2,5).
\end{equation*}
This variety has two nodal singularities at the points~$[e_{25}]$ and~$[e_{45}]$.

\subsubsection{$\dd(X) = 5$ and~$\rr(X) = 4$}

In this case~$\dim(X) = 3$, $\Xi(X) = \rA_1$, and there is a single maximal \emph{triple Schubert variety}
\begin{equation*}
X_{5,4,3} = \{x_{12} = x_{34} = x_{15} + x_{35} = 0\} \subset \Gr(2,5),
\end{equation*}
an intersection of three Schubert divisors.
This variety has three nodal singularities at the points~$[e_{45}]$, $[e_{25}]$, and~$[e_{24}]$.

\subsection{Degree 4} 

Starting from this degree we only list maximal del Pezzo varieties.

\subsubsection{$\dd(X) = 4$ and~$\rr(X) = 2$}

In this case~$\dim(X) \le 6$, $\Xi(X) = \rA_4$, and the maximal variety is
\begin{equation*}
X_{4,2,6} = x_{1}x_{8}+x_{2}x_{6}+x_{3}x_{4}=x_{1}x_{5}+x_{2}x_{7}+x_{3}x_{9} = 0 \} \subset \PP^8.
\end{equation*}
Its singular locus is~$\PP^1 \times \PP^2$, 
with the equations~$x_1 = x_2 = x_3 = 0$ 
and~$\rk\left(\begin{smallmatrix} x_8 & x_6 & x_4 \\ x_5 & x_7 & x_9 \end{smallmatrix}\right) \le 1$.

\subsubsection{$\dd(X) = 4$ and~$\rr(X) = 3$}

In this case~$\dim(X) \le 5$, $\Xi(X) = \rA_3$, and the maximal varieties are
\begin{align*}
X_{4,3,5} &= \{x_{1}x_{8}+x_{2}x_{6}+x_{3}x_{4}=x_{1}x_{5}+x_{2}x_{7} = 0 \} \subset \PP^7,
\\
X^*_{4,3,5} &= \{ x_{2}x_{6}+x_{3}x_{4}=x_{1}x_{5}+x_{7}x_{8} = 0 \} \subset \PP^7,
\end{align*}
The first variety~$X_{4,3,5}$ is a special hyperplane section of~$X_{4,2,6}$, given by the equation~$x_9 = 0$,
and its singular locus is the union of the two planes~$\langle e_4, e_6, e_8 \rangle$ and~$\langle e_3, e_6, e_8 \rangle$
and the quadric~$\{x_{1}= x_{2} = x_{3}= x_{4} = x_{5}x_{6}-x_{7}x_{8} =0\}$. 
The second variety~$X^*_{4,3,5}$ is the join of two quadric surfaces, 
and its singular locus is the union of those surfaces.

\subsubsection{$\dd(X) = 4$ and~$\rr(X) = 4$}

In this case~$\dim(X) \le 4$, $\Xi(X) = \rA_2$, and the maximal variety is
\begin{equation*}
X_{4,4,4} = \{ x_{2}x_{6}+x_{3}x_{4}=x_{1}x_{5}+x_{2}x_{7}= 0 \} \subset \PP^6.
\end{equation*}
This is a special hyperplane section of~$X_{4,3,5}$, given by the equation~$x_8 = 0$,
and its singular locus is the union of~$5$ lines.
It can be also represented as a special hyperplane section of~$X^*_{4,3,5}$ 
given by the equation~$x_8 = x_2$.

\subsubsection{$\dd(X) = 4$ and~$\rr(X) = 5$}

In this case~$\dim(X) = 3$, $\Xi(X) = \rA_1$, and
\begin{equation*}
X_{4,5,3} = \{ x_{1}x_{5} = -x_{2}x_{6} = x_{3}x_{4} \} \subset \PP^5.
\end{equation*}
This is a special hyperplane section of $X_{4,4,4}$, given by the equation~$x_7 = x_6$,
and its singular locus is the union of the six points~$[e_i]$, $1 \le i \le 6$.
This toric threefold is known as the \emph{tetrahedral quartic threefold}, 
see~\cite[Chapter~VIII, 2.31]{Semple-Roth:85}.

\subsection{Degree~$3$}

In this section we list the maximal cubic hypersurfaces of type~$\rA$.

\subsubsection{$\dd(X) = 3$ and~$\rr(X) = 2$}

In this case~$\dim(X) \le 7$, $\Xi(X) = \rA_5$, and the maximal variety is
\begin{equation*}
X_{3,2,7} = \left\{ \det
\left(\begin{smallmatrix}
x_{9}&x_{1}&x_{2}\\
x_{4}&x_{8}&x_{3}\\
x_{5}&x_{6}&x_{7}\\
\end{smallmatrix}\right)
= 0 \right\} \subset \PP^8.
\end{equation*}
This is the determinantal cubic.
Its singular locus is the Segre variety~$\PP^2 \times \PP^2$.

\subsubsection{$\dd(X) = 3$ and~$\rr(X) = 3$}

In this case~$\dim(X) \le 6$, $\Xi(X) = \rA_4$, and the maximal variety is
\begin{equation*}
X_{3,3,6} = \left\{ \det
\left(\begin{smallmatrix}
0&x_{1}&x_{2}\\
x_{4}&x_{8}&x_{3}\\
x_{5}&x_{6}&x_{7}\\
\end{smallmatrix}\right)
= 0 \right\} \subset \PP^7.
\end{equation*}
This is a special hyperplane section of~$X_{3,2,7}$, given by the equation~$x_9 = 0$,
and its singular locus consists of the 3-space~$\langle e_3, e_6, e_7, e_8 \rangle$
and two cubic 
scrolls
\begin{equation*}
\{x_1=x_2=0,\ \rk \left(\begin{smallmatrix}
x_{4}&x_{8}&x_{3}\\
x_{5}&x_{6}&x_{7}\\
\end{smallmatrix}\right)\le 1\}
\qquad\text{and}\qquad 
\{x_4=x_5=0,\ \rk \left(\begin{smallmatrix}
x_{1} & x_{8} & x_{6}\\
x_{2} & x_{3} & x_{7}
\end{smallmatrix}\right)\le 1\}.
\end{equation*}

\subsubsection{$\dd(X) = 3$ and~$\rr(X) = 4$}

In this case~$\dim(X) \le 5$, $\Xi(X) = \rA_3$, and the maximal variety is
\begin{equation*}
X_{3,4,5} = \left\{ \det
\left(\begin{smallmatrix}
0&x_{1}&x_{2}\\
x_{4}&0&x_{3}\\
x_{5}&x_{6}&x_{7}\\
\end{smallmatrix}\right)
= 0 \right\} \subset \PP^6.
\end{equation*}
This is a special hyperplane section of~$X_{3,3,6}$, given by the equation~$x_8 = 0$,
and its singular locus consists of four planes~$\langle e_2, e_3, e_7 \rangle$, 
$\langle e_2, e_5, e_7 \rangle$, $\langle e_3, e_6, e_7 \rangle$, $\langle e_5, e_6, e_7 \rangle$,
and two quadric surfaces
$\{x_1=x_2=x_6=\det\left(\begin{smallmatrix} x_4 & x_3 \\ x_5 & x_7 \end{smallmatrix}\right) = 0\}$,
$\{x_3=x_4=x_5=\det\left(\begin{smallmatrix} x_1 & x_2 \\ x_6 & x_7 \end{smallmatrix}\right) =0\}$.

\subsubsection{$\dd(X) = 3$ and~$\rr(X) = 5$}

In this case~$\dim(X) \le 4$, $\Xi(X) = \rA_2$, and the maximal variety is
\begin{equation*}
X_{3,5,4} = \left\{ \det
\left(\begin{smallmatrix}
0&x_{1}&x_{2}\\
x_{4}&0&x_{3}\\
x_{5}&x_{6}&0\\
\end{smallmatrix}\right)
= 0 \right\} \subset \PP^5.
\end{equation*}
This is a special hyperplane section of~$X_{3,4,5}$, given by the equation~$x_7 = 0$,
and its singular locus is a configurations of nine lines~$\langle e_i, e_j \rangle$, where~$i$ is odd and~$j$ is even. 
This is a toric cubic hypersurface, known as the \emph{Perazzo primal}
(see~\cite{Baker1931}, \cite[Exercise~9.16]{Dolgachev-ClassicalAlgGeom}). 

\subsubsection{$\dd(X) = 3$ and~$\rr(X) = 6$}

In this case~$\dim(X) = 3$, $\Xi(X) = \rA_1$, and
\begin{equation*}
X_{3,6,3} = \left\{ \sum_{i=1}^6 x_i = 
\det
\left(\begin{smallmatrix}
0&x_{1}&x_{2}\\
x_{4}&0&x_{3}\\
x_{5}&x_{6}&0\\
\end{smallmatrix}\right)
=0 \right\} \subset \PP^5
\end{equation*}
This is a special hyperplane section of~$X_{3,5,4}$, given by the equation~$\sum x_i = 0$,
and it has~10 nodes.
In fact, this is the \emph{Segre cubic} (see~\cite{Baker1931}, \cite[Sect.~3.2]{Hunt:book:96}, 
\cite[\S9.4.4]{Dolgachev-ClassicalAlgGeom}).  

\subsection{Degree~$2$}

In this section we list the maximal double coverings of type~$\rA$. 

\subsubsection{$\dd(X) = 2$ and~$\rr(X) = 2$}
\label{sss:x22n}

In this case~$\dim(X) \le 8$, $\Xi(X) = \rA_6$, and the maximal variety is 
\begin{equation*}
X_{2,2,8} = 
\left\{ y^2 = \det
\left(\begin{smallmatrix}
0&x_{1}&x_{2}&x_{3}\\
x_{1}&x_{9}&x_{4}&x_{5}\\
x_{2}&x_{4}&x_{8}&x_{6}\\
x_{3}&x_{5}&x_{6}&x_{7}\\
\end{smallmatrix}\right)
\right\} \subset \PP(1^9,2).
\end{equation*}
Its singular locus is the union of~$\PP^5 = \langle e_4, e_5, e_6, e_7, e_8, e_9 \rangle$ 
with a projection of the Segre variety~$\PP^2 \times \PP^3$.

\subsubsection{$\dd(X) = 2$ and~$\rr(X) = 3$}

In this case~$\dim(X) \le 7$, $\Xi(X) = \rA_5$, and the maximal varieties are
\begin{align*}
X_{2,3,7} &= \left\{ y^2 = \det
\left(\begin{smallmatrix}
0&x_{1}&x_{2}&x_{3}\\
x_{1}&0&x_{4}&x_{5}\\
x_{2}&x_{4}&x_{8}&x_{6}\\
x_{3}&x_{5}&x_{6}&x_{7}\\
\end{smallmatrix}\right)
\right\},
\\
X^*_{2,3,7} &= \left\{ 
\begin{aligned}
y^2 &= x_1^2x_8^2 + x_2^2x_7^2 + x_3^2x_6^2 + x_4^2x_5^2 
- 2x_1x_2x_7x_8 - 2x_1x_3x_6x_8 - 2x_1x_4x_5x_8 
\\ 
&- 2x_2x_3x_6x_7 - 2x_2x_4x_5x_7 - 2x_3x_4x_5x_6 
+ 4x_1x_5x_6x_7 + 4x_2x_3x_4x_8 
\end{aligned}
\right\}, 
\end{align*}
both in~$\PP(1^8,2)$.
The first variety~$X_{2,3,7}$ is a special hyperplane section of~$X_{2,2,8}$, given by the equation~$x_9 = 0$
and its singular locus is the union of two copies of~$\PP^4$, a projection of~$\PP^1 \times \PP^3$,
and a projection of~$\PP^2 \times \PP^2$.
The second variety~$X^*_{2,3,7}$ is the double covering ramified over the \emph{Cayley hyperdeterminant},
and its singular locus is the union of three copies of~$\PP^1 \times \PP^3$.

\subsubsection{$\dd(X) = 2$ and~$\rr(X) = 4$}

In this case~$\dim(X) \le 6$, $\Xi(X) = \rA_4$, and the maximal variety is
\begin{equation*}
X_{2,4,6} = \left\{ y^2 = \det
\left(\begin{smallmatrix}
0&x_{1}&x_{2}&x_{3}\\
x_{1}&0&x_{4}&x_{5}\\
x_{2}&x_{4}&0&x_{6}\\
x_{3}&x_{5}&x_{6}&x_{7}\\
\end{smallmatrix}\right)
\right\} \subset \PP(1^7,2).
\end{equation*}
This is a special hyperplane section of~$X_{2,3,7}$, given by the equation~$x_8 = 0$
and its singular locus is the union of four copies of~$\PP^3$
and three copies of~$\PP^1 \times \PP^2$.
It can be also represented as a special hyperplane section of~$X^*_{2,3,7}$, given by the equation~$x_8 = 0$
(up to renumbering the variables and rescaling one of them).

\subsubsection{$\dd(X) = 2$ and~$\rr(X) = 5$}

In this case~$\dim(X) \le 5$, $\Xi(X) = \rA_3$, and the maximal variety is
\begin{equation*}
X_{2,5,5} = 
\left\{ y^2 = \det
\left(\begin{smallmatrix}
0&x_{1}&x_{2}&x_{3}\\
x_{1}&0&x_{4}&x_{5}\\
x_{2}&x_{4}&0&x_{6}\\
x_{3}&x_{5}&x_{6}&0\\
\end{smallmatrix}\right) 
\right\} \subset \PP(1^6,2).
\end{equation*} 
This is a special hyperplane section of~$X_{2,4,6}$, given by the equation~$x_7 = 0$
and its singular locus is the union of eight planes~$\PP^2$ and three quadrics~$\PP^1 \times \PP^1$.

\subsubsection{$\dd(X) = 2$ and~$\rr(X) = 6$}

In this case~$\dim(X) \le 4$, $\Xi(X) = \rA_2$, and the maximal variety is
\begin{equation*}
\label{eq:deg2-a}
X_{2,6,4} =\left\{ \sum_{i=1}^6 x_i = 
y^2 - \det 
\left(\begin{smallmatrix}
0&x_{1}&x_{2}&x_{3}\\
x_{1}&0&x_{4}&x_{5}\\
x_{2}&x_{4}&0&x_{6}\\
x_{3}&x_{5}&x_{6}&0\\
\end{smallmatrix}\right)
= 0 \right\} \subset \PP(1^6,2),
\end{equation*} 
This is a special hyperplane section of~$X_{2,5,5}$, given by the equation~$\sum x_i = 0$.
It is known as the \emph{Coble fourfold} (see~\cite{Cheltsov-Kuznetsov-Shramov}),
its branch divisor is the \emph{Igusa quartic},
and its singular locus is the \emph{Cremona--Richmond configuration} of~15 lines.

\subsubsection{$\dd(X) = 2$ and~$\rr(X) = 7$}

In this case~$\dim(X) = 3$, $\Xi(X) = \rA_1$, and the variety~$X_{2,7,3}$ is a \emph{Kummer double solid},
i.e., the double covering~~$X \to \PP^3$ branched at a Kummer surface;
in particular, it has~$16$ nodes.
Note that in contrast with the other cases considered above 
varieties of this type are parameterized by a $3$-dimensional moduli space.  

\section{Roots of~$\Cl(X)$}
\label{sec:details}

In this section we use the computation from the proof of Theorem~\ref{thm:intro-clx} 
to describe uniformly the roots and exceptional classes in~$\Cl(X)$
for all del Pezzo varieties~$X$ with~$\rr(X) \ge 2$ 
(note that when~$\rr(X) = 1$ we have~$A_X^\perp = 0$, so in this case there are neither roots nor exceptional classes). 

\subsection{Type~$\rA_m$}

Let~$X$ be a del Pezzo variety of type~$\rA_m$ with~$\rr(X) \ge 2$.
Assume~$X$ has a $\QQ$-factorialization~$\hX$ such that~$\hX = \Bl_{P_1,\dots,P_k}(\PP_{\PP^2}(\cE))$
(this holds for all varieties of type~$\rA$ except for the varieties~$X^*_{d,3,n}$, that will be discussed separately).
Let~\mbox{$\sigma \colon \hX \to X_0 \coloneqq \PP_{\PP^2}(\cE)$} be the blowup.
Let~\mbox{$S \subset \hX$} be a general linear surface section 
and set~\mbox{$S_0 \coloneqq \sigma(S) \subset \PP_Z(\cE)$}; this is a general linear surface section of~$X_0$.
We have morphisms
\begin{equation*}
S \xrightarrow{\ \sigma\ } S_0 \xrightarrow{\ f\ } \PP^2;
\end{equation*}
the morphism~$f$ blows up~$m + 1$ point, and the morphism~$\sigma$ blows up~$k$ points.
Note that
\begin{equation}
\label{eq:k-m}
\rr(X) = k + 2,
\qquad\text{and}\qquad
\dd(X)+ k + m = 8,
\end{equation}
in particular~$k + m \le 7$. 
Since also~$m \ge 1$ and~$\dd(X) \ge 1$, we have~$k \le 6$.

Let~$\h\in \Cl(S)$ be the pullback of the hyperplane class of~$\PP^2$,
let~$\e^0_0,\dots,\e^0_m \in \Cl(S)$ be the pullbacks of the exceptional divisors of~$f$,
and let~$\e_1,\dots,\e_k\in \Cl(S)$ be the exceptional divisors of~$\sigma$, and set
\begin{equation*}
\e^0 \coloneqq \e^0_0 + \dots + \e^0_m,
\qquad 
\e \coloneqq \e_1 + \dots + \e_k.
\end{equation*}
Then the proof of Theorem~\ref{thm:intro-clx} shows that
\begin{equation*}
\Xi(X) = \left\langle \e^0_0 - \e^0_1, \dots, \e^0_{m-1} - \e^0_m \right\rangle,
\qquad 
\Cl(X) = \left\langle \h, \e^0, \e_1, \dots, \e_k \right\rangle.
\end{equation*}
Now after a simple computation we see that the roots in~$\Cl(X)$ are
\begin{equation}
\balpha = 
\begin{cases}
\pm(\e_{i_1} - \e_{i_2}), 
& \text{if~$k \ge 2$},\\
\pm(\h - \e_{i_1} - \e_{i_2} - \e_{i_3}), 
& \text{if~$k \ge 3$},\\
\pm(2\h - \e), 
& \text{if~$m = 1$ and~$k = 6$},\\
\pm( \h - \e^0 - \e_{i_1} - \dots - \e_{i_{2-m}}), 
& \text{if~$m \le 2$ and~$k \ge 2 - m$},\\
\pm(2\h - \e^0 - \e_{i_1} - \dots - \e_{i_{5-m}}), 
& \text{if~$m \le 5$ and~$k \ge 5 - m$},\\
\pm(3\h - \e^0 - \e - \e_i), 
& \text{if~$m \le 6$ and~$k = 7 - m$}.
\end{cases}
\end{equation} 
where~$1 \le i_1 < \dots < i_s \le k$,
and the exceptional elements in~$\Cl(X)$ are
\begin{equation}
\repsilon = 
\begin{cases}
\e_i,
& \text{if~$k \ge 1$},\\
\h - \e_{i_1} - \e_{i_2},
& \text{if~$k \ge 2$},\\
\h - \e^0,
& \text{if~$m = 1$},\\
2\h - \e_{i_1} - \dots - \e_{i_5},
& \text{if~$k \ge 5$},\\
2\h - \e^0 - \e_{i_1} - \dots - \e_{i_{4-m}},
& \text{if~$m \le 4$, $k \ge 4 - m$},\\
3\h - \e^0 - \e_{i_1} - \dots - \e_{i_{5-m}} - 2\e_{i_{6-m}},
& \text{if~$m \le 5$, $k \ge 6 - m$},\\
4\h - \e^0 -2\e + \e_{i_1} + \dots + \e_{i_{4-m}},
& \text{if~$m \le 4$, $k = 7 - m$},\\
4\h - 2\e^0 - 2\e + \e_{i_1} + \dots + \e_{i_{5}},
& \text{if~$m \le 2$, $k = 7 - m$},\\
5\h - \e^0 - 2\e,
& \text{if~$m = 1$, $k = 7 - m$},\\
5\h - 2\e^0 - 2\e + \e_{i_1} + \e_{i_2},
& \text{if~$m \le 5$, $k = 7 - m$},\\
6\h - 2\e^0 - 2\e -\e_i,
& \text{if~$m \le 6$, $k = 7 - m$}.
\end{cases}
\end{equation} 

We can also describe the roots and exceptional classes for the special varieties~$X^*_{d,3,n}$, $d \in \{2,4,6\}$.
Recall that these varieties are (the anticanonical models of) $\PP^{n-2}$-bundles over~$\PP^1 \times \PP^1$.
We denote by~$\f_1,\f_2 \in \Cl(X)$ the pullbacks of the rulings and by~$\ba \in \Cl(X)$ the fundamental class.
Then the roots in~$\Cl(X)$ are
\begin{equation}
\balpha = 
\begin{cases}
\pm(\f_{1} - \f_{2}), 
& \text{if~$d \in \{2,4,6\}$},\\
\pm(\ba - \f_i), 
& \text{if~$d = 2$},\\
\pm(\ba - \f_1 - \f_2), 
& \text{if~$d = 4$},\\
\pm(\ba - 2\f_i - \f_j), 
& \text{if~$d = 6$},
\end{cases}
\end{equation} 
where~$i \ne j \in \{1,2\}$ and there are no exceptional elements. 

\subsection{Type~$\rD_m$}

Let~$X$ be a del Pezzo variety of type~$\rD_m$ with~$\rr(X) \ge 2$.
By Theorem~\ref{thm:intro-bircla}\ref{it:bircla-dm} the variety~$X$ 
has a $\QQ$-factorialization~$\hX = \Bl_{P_1,\dots,P_k}(X_0)$,
where~$X_0$ is a flat quadric bundle over~$\PP^1$.
As before, let~$\sigma \colon \hX \to X_0$ be the blowup,
let~$S \subset \hX$ be a general linear surface section and set~$S_0 \coloneqq \sigma(S) \subset \PP_Z(\cE)$;
this is a general linear surface section of~$X_0$.
We have morphisms
\begin{equation*}
S \xrightarrow{\ \sigma\ } S_0 \xrightarrow{\ f\ } \PP^1,
\end{equation*}
where~$f$ is a conic bundle with~$m$ singular fibers, and the morphism~$\sigma$ blows up~$k$ points.
The equalities~\eqref{eq:k-m} still hold;
in particular~$k + m \le 7$. 
Since also~$m \ge 4$ and~$\dd(X) \ge 1$, we have~$k \le 3$.

As in the proof of Theorem~\ref{thm:intro-bircla}\ref{it:bircla-dm}
contracting one component in each singular fiber of~$f$
we factor~$f$ through a birational contraction~$f_0 \colon S_0 \to \bar{S}_0 = \PP^1 \times \PP^1$.
We denote by~$\f_1, \f_2\in \Cl(S)$ the pullbacks of the classes of the rulings of~$\bar{S}_0$, 
by~$\e^0_1,\dots,\e^0_m \in \Cl(S)$ the pullbacks of the exceptional divisors of~$f_0$.
Finally, we let~$\e_1,\dots,\e_k \in \Cl(S)$ be the exceptional divisors of~$\sigma$,
and set
\begin{equation*}
\e^0 \coloneqq \e^0_1 + \dots + \e^0_m,
\qquad 
\e \coloneqq \e_1 + \dots + \e_k.
\end{equation*}
Then the proof of Theorem~\ref{thm:intro-clx} shows that
\begin{equation*}
\Xi(X) = \left\langle \e^0_1 - \e^0_2, \dots, \e^0_{m-1} - \e^0_m, \f_2 - \e^0_{m-1} - \e^0_m \right\rangle,
\qquad 
\Cl(X) = \left\langle 2\f_1 - \e^0, \f_2, \e_1, \dots, \e_k \right\rangle.
\end{equation*}
Now after a simple computation we see that the roots in~$\Cl(X)$ are
\begin{equation}
\balpha = 
\begin{cases}
\pm(\e_{i_1} - \e_{i_2}), 
& \text{if~$k \ge 2$},\\
\pm(\f_2 - \e_{i_1} - \e_{i_2}), 
& \text{if~$k \ge 2$},\\
\pm(2\f_1 + \f_2 - \e^0 - \e_{i_1} - \dots - \e_{i_{6-m}}), 
& \text{if~$m \le 6$, $k \ge 6 - m$},\\
\pm(2\f_1 + 2\f_2 - \e^0 - \e - \e_i), 
& \text{if~$m \le 6$, $k = 7 - m$},\\
\end{cases}
\end{equation} 
where~$1 \le i_1 < \dots < i_s \le k$,
and the exceptional elements in~$\Cl(X)$ are
\begin{equation}
\repsilon = 
\begin{cases}
\e_i,
& \text{if~$k \ge 1$},\\
\f_2 - \e_i,
& \text{if~$k \ge 1$},\\
2\f_1 + \f_2 - \e^0 - \e_{i_1} - \dots - \e_{i_{5-m}},
& \text{if~$m \le 5$, $k \ge 5 - m$},\\
2\f_1 + 2\f_2 - \e^0 - \e_{i_1} - \dots - \e_{i_{5-m}} - 2\e_{i_{6-m}},
& \text{if~$m \le 5$, $k \ge 6 - m$},\\
2\f_1 + 3\f_2 - \e^0 - 2\e + \e_{i_1} + \dots + \e_{i_{5-m}},
& \text{if~$m \le 5$, $k = 7 - m$},\\
4\f_1 + 3\f_2 - 2\e^0 - 2\e + \e_i,
& \text{if~$m \le 6$, $k = 7 - m$},\\
4\f_1 + 4\f_2 - 2\e^0 - 2\e - \e_i,
& \text{if~$m \le 6$, $k = 7 - m$}.
\end{cases}
\end{equation} 

\subsection{Type~$\rE_m$}

Let~$X$ be a del Pezzo variety of type~$\rE_m$.
By Theorem~\ref{thm:intro-bircla}\ref{it:bircla-em} the variety~$X$ 
has a $\QQ$-factorialization~\mbox{$\hX = \Bl_{P_1,\dots,P_k}(X_0)$},
where~$X_0$ has type~$\rE_m$ and~$\rr(X_0) = 1$.
As before, let~$\sigma \colon \hX \to X_0$ be the blowup,
let~$S \subset \hX$ be a general linear surface section and set~$S_0 \coloneqq \sigma(S) \subset X_0$;
this is a general linear surface section of~$X_0$
and~$\sigma \colon S \to S_0$ is the blowup of~$k$ points.
Note that~$m \in \{6,7\}$ and~$1 \le k \le 8 - m$.

We denote by~$\ba_0\in \Cl(S)$ the pullback to~$S$ of the anticanonical class of the surface~$S_0$, 
by~\mbox{$\e_1,\dots,\e_k \in \Cl(S)$} the exceptional divisors,
and set~$\e = \e_1 + \dots + \e_k$.
Now we see that the roots in~$\Cl(X)$ are
\begin{equation}
\balpha = 
\begin{cases}
\pm(\e_1 - \e_2), 
& \text{if~$m = 6$, $k = 2$},\\
\pm(\ba_0 - \e - \e_i), 
& \text{if~$m \ge 6$ and~$k = 8 - m$}.
\end{cases}
\end{equation} 
where~$1 \le i \le k$,
and the exceptional elements in~$\Cl(X)$ are
\begin{equation}
\repsilon = 
\begin{cases}
\e_i,
& \text{if~$k \ge 1$},\\
\ba_0 - 2\e_i,
& \text{if~$m = 6$, $k \ge 1$},\\
2\ba_0 - 2\e - \e_i,
& \text{if~$m \ge 6$, $k = 8 - m$}.
\end{cases}
\end{equation} 

\subsection{Table}

The results of the above computations are summarized in the following table.
We list del Pezzo varieties~$X$ according to Dynkin types of~$\Xi(X)$ 
and for each variety we show its degree, the rank of the class group,
the root system~$\Delta \coloneqq \Delta(\Cl(X),A_X)$,
and the cardinalities of the sets
$\Theta_1 \coloneqq \Theta_1(\Cl(X),A_X)$ of exceptional classes,
$\Theta_2 \coloneqq \Theta_2(\Cl(X),A_X)$ of $\PP^1$-classes, and
$\Theta_3^\circ \coloneqq \Theta_3^\circ(\Cl(X),A_X)$ of $\PP^2$-classes,
see~Definition~\ref{def:exP1P2}.
The last column shows if the corresponding variety is primitive or not. 

\renewcommand\arraystretch{1.}
\setlength{\tabcolsep}{1.4em}
\begin{longtable}{lllcccc}
\\\hline
$\dd(X)$ & $\rr(X)$ & $\Delta$& $|\Theta_1|$ & $|\Theta_2|$& $|\Theta_3^\circ|$& primitive
\\\hline\endhead\hline
\multicolumn{7}{c}{{$\Xi(X) = \rA_1$\qquad $\dim(X)=3$} }
\\\hline
$1$ & 8&\type{E_7}& $126$& $756$ & $4032 $& $-$
\\
$2$ & 7&\type{D_6}& $32$& $60$ & $192$& $-$
\\ 
$3$ & 6&\type{A_5} &$15$& $15$ & $30$ & $-$
\\ 
$4$ & 5&\type{A_1\times A_3}&$8$& $6$ & $8$& $-$
\\ 
$5$ & 4&\type{A_2}&$4$& $3$ & $3$& $-$
\\
$6$ & 3&\type{A_1}&$2$& $1$ & $2$& $-$
\\
$6^*$ & 3&\type{A_2} &$0$& $3$ & $0$& $+$
\\ 
$7$ & 2&$\varnothing$&$1$& $0$ & $1$& $-$
\\
$8$ & 1&$\varnothing$&$0$& $0$ & $0$& $+$
\\\hline
\multicolumn{7}{c}{{$\Xi(X) = \rA_2$ \qquad $\dim(X)\le 4$}}
\\\hline 
$1$&7&\type{E_6} & $72$ & $270$& $864$&$-$
\\
$2$&6&\type{A_5}& $20$ & $30$& $60$&$-$
\\ 
$3$&5&\type{2A_2}& $9$ & $9$ & $12$&$-$
\\ 
$4$& 4&\type{2A_1}& $4$ & $4$& $4$&$-$
\\ 
$5$&3&\type{A_1}& $1$ & $2$& $2$&$-$
\\
$6$& 2&\type{A_1}& $0$ & $0$& $2$&$+$
\\ \hline
\multicolumn{7}{c}{{$\Xi(X) = \rA_3$ \qquad $\dim(X)\le 5$}}
\\\hline
$1$&6&\type{D_5} & $40$ &$90$& $160$&$-$
\\
$2$&5&\type{A_1\times A_3}&$12$& $14$& $16$&$-$
\\
$3$&4&\type{A_1\times A_1} & $5$&$5$&$4$&$-$
\\
$4$&3&$\varnothing$&$2$&$2$&$2$&$-$
\\
$4^*$&3& \type{A_1\times A_1}&0& $4$& 0 &$+$
\\
$5$&2&$\varnothing$ &$0$&$1$&$1$&$+$
\\\hline
\multicolumn{7}{c}{{$\Xi(X) = \rA_4$ \qquad $\dim(X)\le 6$}}
\\\hline
$1$&5&\type{A_4}&20&30&40&$-$
\\
$2$&4&\type{A_2}&6&6&6&$-$
\\
$3$&3&\type{A_1}&2&2&2 &$-$
\\
$4$&2&$\varnothing$&1&0&1&$-$
\\
$5$&1&$\varnothing$&0&0&0&$+$
\\\hline
\multicolumn{7}{c}{{$\Xi(X) = \rA_5$ \qquad $\dim(X)\le 7$}}
\\\hline
1&4&\type{A_1\times A_2}&8&12&12&$-$
\\
2&3&\type{A_1}&2&2&4&$-$
\\
$2^*$&3&\type{A_2}&0&6&0 &$+$
\\
3&2&\type{A_1}&0&0&2&$+$
\\\hline
\multicolumn{7}{c}{{$\Xi(X) = \rA_6$ \qquad $\dim(X)\le 8$ }}
\\\hline
1 &3&\type{A_1}&2&4& 4 &$-$
\\
2 &2&$\varnothing$&0&0& 2 &$+$ 
\\\hline
\multicolumn{7}{c}{{$\Xi(X) = \rA_7$ \qquad $\dim(X)\le 9$ }}
\\\hline
1 &2&$\varnothing$&0&0& 2&$+$
\\\hlineB{3}
\multicolumn{5}{c}{{$\Xi(X) = \rD_4$}}
\\\hline
1&5&\type{D_4} &24&24&0&$-$
\\
2 &4&\type{A_1\times A_1\times A_1}&8&6&0&$-$
\\ 
3&3&$\varnothing$&3&3&0 &$-$
\\ 
4&2&$\varnothing$&0&2&0 &$+$
\\\hline
\multicolumn{5}{c}{{$\Xi(X) = \rD_5$}}
\\\hline
1&4&\type{A_3}& 12&6&0 &$-$
\\
2&3&\type{A_1}&4&2&0&$-$
\\
3&2&$\varnothing$&1&1&0&$-$
\\
4&1&$\varnothing$&0&0&0& $+$
\\\hline
\multicolumn{5}{c}{{$\Xi(X) = \rD_6$}}
\\\hline
1&3&\type{A_1\times A_1}&4&4&0 &$-$
\\
2&2& \type{A_1}&0&2&0 &$+$
\\\hline
\multicolumn{5}{c}{{$\Xi(X) = \rD_7$}}
\\\hline
1&2&$\varnothing$&0&2&0&$+$
\\\hline
\multicolumn{5}{c}{{$\Xi(X) = \rE_6$}} 
\\\hline
1&3&\type{A_2}&6&0&0&$-$
\\
2&2&$\varnothing$&2&0&0&$-$
\\
3&1&$\varnothing$&0&0&0&$+$
\\\hline
\multicolumn{5}{c}{{$\Xi(X) = \rE_7$}} 
\\\hline 
1& 2& \type{A_1} &2&0&0 &$-$ 
\\ 
2& 1 &$\varnothing$ &0&0&0 &$+$
\\\hline
\multicolumn{5}{c}{{$\Xi(X) = \rE_8$}} 
\\\hline 
1& 1 &$\varnothing$ &0&0&0 &$+$
\\\hline
\end{longtable}

%%%%%%%%%%%%%%%%%%%%%%%%%%%%%%%%%%%%%%%%%%%%%%%%%%%%%%%%
%%%%%%%%%%%%%%%%%%%%%%%%%%%%%%%%%%%%%%%%%%%%%%%%%%%%%%%%
%%%%%%%%%%%%%%%%%%%%%%%%%%%%%%%%%%%%%%%%%%%%%%%%%%%%%%%%
%%%%%%%%%%%%%%%%%%%%%%%%%%%%%%%%%%%%%%%%%%%%%%%%%%%%%%%%

\providecommand*{\BibDash}{}

\def\cprime{$'$}

% \bibliography{dp}
% \bibliographystyle{alpha}

\end{document}